\documentclass[11pt]{amsart}
\usepackage{euscript}
\usepackage{amssymb}
\usepackage{amsmath}
\usepackage{epic}
\usepackage{graphics}
\usepackage{epsfig}
\usepackage{color}

\usepackage{amscd,euscript}
\usepackage[frame,cmtip,curve,arrow,matrix,line,graph]{xy}

\usepackage[colorlinks]{hyperref}
\usepackage{quiver}

\usepackage{tikz}
\usepackage{tikz-cd}
\usepackage{mathtools}
\usepackage{quiver}

\numberwithin{equation}{section}
\newtheoremstyle{my}{1.5em}{0.5em}{\em}{}{\sc}{.}{0.5em}{}

\newtheorem{theorem}{Theorem}[section]
\newtheorem{thm}{Theorem}[section]
\newtheorem{Theorem}[thm]{Theorem}
\newtheorem*{Theorem*}{Theorem}
\newtheorem{Corollary}[thm]{Corollary}
\newtheorem{corollary}[thm]{Corollary}
\newtheorem*{corollary*}{Corollary}
\newtheorem{Lemma}[thm]{Lemma}

\newtheorem{prop}[thm]{Proposition}
\newtheorem{Proposition}[thm]{Proposition}

\newtheorem*{conjecture*}{Conjecture}

\newtheorem*{question*}{Question}
\newtheorem{defn}[thm]{Definition}
\newtheorem{Definition}[thm]{Definition}

\newtheorem*{definitions*}{Definitions}
\newtheorem{rem}[thm]{Remark}
\newtheorem*{rem*}{Remark}
\newtheorem{Remark}[thm]{Remark}

\newtheorem*{remark*}{Remark}

\newtheorem*{remarks*}{Remarks}
\newtheorem*{example*}{Example}
\newtheorem{Example}[thm]{Example}
\newtheorem*{examples*}{Examples}

\newtheorem*{convention*}{Convention}
\newtheorem*{conventions*}{Conventions}

\newtheorem{Addendum}[thm]{Addendum}
\newtheorem{Notation}[thm]{Notation}

\newtheorem*{exercise*}{Exercise}
\newtheorem*{bibliographical-note*}{Bibliographical note}

\newtheorem{cor}[thm]{Corollary}
\newtheorem{lemma}[thm]{Lemma}

\newtheorem{remark}[thm]{Remark}

\newcommand{\Acknowledgements}{{\em Acknowledgements.} }

\newcommand{\scrH}{\EuScript{H}}

\newcommand{\scrA}{\EuScript{A}}
\newcommand{\scrT}{\EuScript{T}}

\newcommand{\calC}{\mathcal{C}}

\newcommand{\calF}{\mathcal{F}}

\newcommand{\scrK}{\EuScript{K}}

\newcommand{\scrC}{\EuScript{C}}
\newcommand{\scrS}{\EuScript{S}}

\newcommand{\scrW}{\EuScript{W}}

\newcommand{\scrL}{\EuScript{L}}

\newcommand{\bE}{\mathbb{E}}

\newcommand{\bS}{\mathbb{S}}
\newcommand{\bM}{\mathbb{M}}

\newcommand{\bF}{\mathbb{F}}
\newcommand{\bR}{\mathbb{R}}
\newcommand{\bZ}{\mathbb{Z}}
\newcommand{\bQ}{\mathbb{Q}}
\newcommand{\bC}{\mathbb{C}}
\newcommand{\bN}{\mathbb{N}}
\newcommand{\bP}{\mathbb{P}}

\newcommand{\cdbar}{\mathrm{\overline{\partial}}}

\newcommand{\id}{\mathrm{id}}

\renewcommand{\ker}{\mathrm{ker}}

\newcommand{\Hom}{\mathrm{Hom}}

\newcommand{\univ}{\textnormal{univ}}
\newcommand{\Herm}{\textnormal{Herm}}
\newcommand{\vdim}{\mathrm{vdim}}

\newcommand{\rank}{\mathrm{rank}}

\def\lra#1{\overset{#1}{\longrightarrow}}

\renewcommand{\hbar}{\overline{\frak{h}}}

\newcommand{\Ham}{\mathrm{Ham}}

\newcommand{\Diff}{\mathrm{Diff}}

\newcommand{\scrM}{\EuScript{M}}
\newcommand{\scrE}{\EuScript{E}}
\newcommand{\scrF}{\EuScript{F}}

\newcommand{\scrU}{\EuScript{U}}

\newcommand{\ccL}{\mathcal L}

\newcommand{\ccM}{\mathcal M}

\newcommand{\ccMbar}{\overline{\mathcal M}}

\newcommand{\bK}{\mathbb{K}}
\newcommand{\K}{\EuScript{K}}

\numberwithin{equation}{section}
\renewcommand{\leq}{\leqslant}
\renewcommand{\geq}{\geqslant}

\newcommand{\C}{\mathbb C}

\newcommand{\N}{\mathcal{N}}

\newcommand{\Z}{\mathbb{Z}}

\def\frh{\mathfrak{h}}
\def\frs{\mathfrak{s}}

\def\fre{\mathfrak{e}}
\def\fro{\mathfrak{o}}
\def\frg{\mathfrak{g}}

\newcommand{\Ext}{\operatorname{Ext}}

\renewcommand{\sc}{\operatorname{sc}}

\newcommand{\rk}{\operatorname{rk}}

\newcommand{\vfc}{\mathrm{vfc}}

\renewcommand{\Im}{\operatorname{Im}}

\title[Complex cobordism, Hamiltonian loops and global Kuranishi charts]{Complex cobordism, Hamiltonian loops, and \\ global Kuranishi charts}
\author{Mohammed Abouzaid, Mark McLean, Ivan Smith}
\address{Mohammed Abouzaid, Columbia}
\email{abouzaid@math.columbia.edu}
\address{Mark McLean, Stony Brook}
\email{mark.mclean@stonybrook.edu}
\address{Ivan Smith, Cambridge}
\email{is200@cam.ac.uk}
\date{October 2021.}

\begin{document}
\maketitle

\thispagestyle{empty}

\begin{abstract} Let $(X,\omega)$ be a closed symplectic manifold.  A loop $\phi: S^1 \to \Diff(X)$ of diffeomorphisms of $X$ defines a fibration $\pi: P_{\phi} \to S^2$.   By applying Gromov-Witten theory to moduli spaces of holomorphic sections of $\pi$, Lalonde, McDuff and Polterovich proved that if $\phi$ lifts to the Hamiltonian group $\Ham(X,\omega)$, then the rational cohomology of $P_{\phi}$ splits additively.  We prove, with the same assumptions, that the $\bE$-generalised cohomology of $P_{\phi}$ splits additively for any complex-oriented cohomology theory $\bE$, in particular the integral cohomology splits. This class of examples includes all  complex projective varieties equipped with a smooth morphism to $\bC \bP^1$, in which case the analogous rational result was proved by Deligne using Hodge theory. The argument employs virtual fundamental cycles of moduli spaces of sections of $\pi$ in Morava $K$-theory and results from chromatic homotopy theory. Our proof involves  a construction of independent interest:  we build global Kuranishi charts  for moduli spaces of pseudo-holomorphic spheres in $X$ in a class $\beta \in H_2(X;\bZ)$, depending on a choice of integral symplectic form $\Omega$ on $X$ and ample Hermitian line bundle over the moduli space of one-pointed degree $d = \langle \Omega,\beta\rangle$ stable genus zero curves in $\bC\bP^d$. \end{abstract}
{ \setcounter{tocdepth}{1}
\hypersetup{linkcolor=black}
\tableofcontents
}

\section{Introduction}

Let $(X,\omega)$ be a closed symplectic manifold and denote by $\Ham(X,\omega)$ the group of Hamiltonian diffeomorphisms of $X$, equipped with its $C^{\infty}$-topology. 
A loop $\phi: S^1 \to \Ham(X,\omega)$ defines  via the clutching construction a Hamiltonian fibration $P_{\phi} \to S^2$, i.e. a smooth fibre bundle with structure group $\Ham(X,\omega)$.  

\begin{Theorem}\label{thm:main}
The cohomology  $H^*(P_{\phi};\bZ)$ splits additively as $H^*(X;\bZ) \otimes H^*(S^2;\bZ)$. 
\end{Theorem} 

 \begin{Corollary} \label{Cor:main}
 The `sweepout' map defined by $S^1 \times X \stackrel{\phi}{\longrightarrow} X$
\begin{equation}
\delta_{\phi}: H_*(X;\bZ) \to H_{*+1}(X;\bZ)
\end{equation} vanishes.
\end{Corollary}

The result answers a twenty-year-old question of McDuff \cite{McDuff}. It is new even for a smooth projective variety equipped with a non-singular map to $\bC\bP^1$.   Combining Theorem \ref{thm:main} with topological arguments from  \cite[Section 4]{Lalonde-McDuff}, it follows that the  $\bF_p$-cohomology of a Hamiltonian fibration with base space $B$ splits additively when $B$ is (for instance) a product of complex projective spaces. For definiteness, we only discuss fibrations over $S^2$ in this paper.  Simple examples such as the Hopf surface $S^1\times S^3 \to S^2$ and the odd Hirzebruch surface $\bP(\mathcal{O} \oplus \mathcal{O}(-1)) \to \bP^1$ show respectively that the Hamiltonian hypothesis on $\phi$ is essential, and that for Hamiltonian loops the cohomology need not split multiplicatively. 
\medskip

  Theorem \ref{thm:main} is a special case of a more general result.  For a generalised cohomology theory $\bE$, we write $H^*(X;\bE) := \bE^*(X)$ for the $\bE$-cohomology groups of $X$, and $\bE_*$ for the coefficients groups $H^*(\mathrm{pt};\bE)$.
  
  \begin{Theorem} \label{thm:main_generalised}
  Suppose that $\bE$ is complex oriented.  Then, in the notation of Theorem \ref{thm:main}, there is an additive splitting $H^*(P_{\phi};\bE) \cong H^*(X;\bE) \otimes_{\bE_*} H^*(S^2;\bE)$.
  \end{Theorem}

The complex-oriented hypothesis on $\bE$ is essential; the real $KO$-theory of the odd Hirzebruch surfaces does not split additively.  A striking feature of the proof of Theorem \ref{thm:main_generalised} is that the geometry of spaces of holomorphic curves gives us access to the case of $\bE$ a Morava $K$-theory, a class of spectra which are `algebraically' natural, whilst results from chromatic homotopy theory then allow us to pass to the case of $\bE = MU$, the geometrically natural theory of complex cobordism.  See Section \ref{Sec:Examples} for several examples illustrating Theorems \ref{thm:main} and \ref{thm:main_generalised}.  

Recall that there are elements  $Q_i$ in the Milnor basis \cite{Milnor:Steenrod} of the Steenrod algebra $\scrA_p$, determined inductively in terms of the Bockstein $\beta$ and power operations $P^j$ via 
\begin{equation}
Q_0 = \beta; \ Q_{i+1} = [P^{\,p^i}, Q_i].
\end{equation}
The definition applies when $p=2$ taking $Q_{i+1} = [Sq^{2^{i}}, Q_{i}]$, giving $Q_0 = Sq^1$, $Q_1 = [Sq^2,Sq^1]$, etc. We use the techniques behind Theorem \ref{thm:main_generalised}  to show:

\begin{Corollary} \label{Cor:Steenrod}
The additive isomorphism $H^*(P_{\phi};\bZ/p) \cong H^*(X;\bZ/p) \otimes_{\bZ/p} H^*(S^2;\bZ/p)$ in Theorem \ref{thm:main_generalised} entwines the action of the subalgebra of the Steenrod algebra $\scrA_p$ generated by the operations $Q_i$. 
\end{Corollary}

The isomorphism cannot be chosen to entwine the entire Steenrod algebra: the $MU$-modules associated to the product $S^2 \times S^2$ and the odd Hirzebruch surface $\bP(\mathcal{O} \oplus \mathcal{O}(-1))$ are isomorphic, but the operation $Sq^2$ is trivial in the first case but not in the second  as can be easily seen by considering the parity of the intersection form.

\begin{Remark} \label{Rmk:Uniqueness}
If $\bE$ is a $p$-local cohomology theory (for example, cohomology with $\bZ/p$ coefficients as in Corollary \ref{Cor:Steenrod}), then the splitting of Theorem \ref{thm:main_generalised} depends only on splitting the map of spectra of Theorem \ref{Thm:Hovey} (this has nothing to do with our particular problem and can be done universally).  The additive splitting of cohomology that results in principle depends on this choice.  Starting from Remark \ref{rem:sharpen}, one can show the splitting is compatible with restriction to the fibre $H^*(X;\bE)$ and with the pullback action of $H^*(S^2;\bE)$.  The splitting for a general complex-orientable cohomology theory (including integral cohomology as in Theorem \ref{thm:main}) depends on a further choice, because our techniques do not ensure that the various mod $p$ choices are `consistent' with the choice for rational cohomology theories and glue together in the arithmetic fracture square.  
\end{Remark}

 Theorems \ref{thm:main} \&  \ref{thm:main_generalised} and Corollaries \ref{Cor:main} \& \ref{Cor:Steenrod} can be viewed as new obstructions to surjectivity of the map $\pi_1\Ham(X) \to \pi_1\Diff(X)$.  
 \medskip

The rational cohomology analogue of Theorem \ref{thm:main} was known in two situations classically:  if $P_{\phi} \to \bC\bP^1$ is a projective morphism of complex algebraic varieties, the result follows from Deligne's degeneration of the Leray spectral sequence \cite{Deligne}; if $\phi$ is a Hamiltonian circle action, it follows from Kirwan's description of cohomology of symplectic quotients \cite{Kirwan}.   The statements for rational cohomology were established for a general monotone $X$ by Lalonde, McDuff and Polterovich \cite{LMP}, and for all $X$ by McDuff \cite{McDuff}, using Gromov-Witten theory.  The restriction to rational coefficients from that viewpoint arises because moduli spaces of holomorphic curves are (closed subsets of) orbifolds, and have virtual fundamental cycles only in rational cohomology.  McDuff \cite{McDuff} asked whether the result also holds integrally.  Abouzaid and Blumberg have introduced Hamiltonian Floer homology groups with coefficients in Morava $K$-theory \cite{AbouzaidBlumberg2021}, using crucially the fact that the classifying space of a finite group satisfies Poincar\'e duality in Morava $K$-theory and hence that compact orbifolds admit fundamental classes in that theory (the `ambidexterity' property).  Following that insight, and borrowing heavily from work of Cheng \cite{Cheng}, Theorem \ref{thm:main} gives a positive answer to McDuff's question.
 
 \begin{Remark}  The Morava $K$-theories are generalized cohomology theories $K(n)$ depending on a choice of prime $p\geq 2$ (usually suppressed from the notation) and a natural number $n \geq 1$, with coefficients $\bF_p[v^{\pm 1}]$ with $|v| = 2(p^n-1)$.  Traditionally, one sets $K(0) = H \bQ$; $K(1)$ is a wedge summand of mod $p$ complex $K$-theory; some authors write $K(\infty) = H\bF_p$. The theories with $n>1$ are very hard to compute, and our argument uses only general properties: most notably, ambidexterity, which is a duality statement for the Morava $K$-theories of classifying spaces of finite groups.  Along with suitable  Eilenberg-Maclane spectra, the Morava $K$-theories are the `only fields' in the stable homotopy category localised at $p$ \cite{DHS,HKR}.  The proof of Theorem \ref{thm:main} also employs Morava $K$-theories with coefficients related to the rings $\bZ/p^{k}\bZ$, which are not completely standard, and which we discuss in Section  \ref{Sec:more_coefficients}. 
 \end{Remark}

  A noteworthy feature of our approach to Theorem \ref{thm:main} is that we construct  \emph{global} Kuranishi charts for moduli spaces of genus zero curves, see Section \ref{Sec:curves}. This allows us to construct the virtual fundamental cycle in one go, rather than via a local-to-global argument involving a Kuranishi cover.  The simplifications that arise from the use of global charts cannot be overstated; not  only does it allow us to formulate all constructions for Theorem \ref{thm:main} at the level of generalised cohomology theories, but it also makes the paper both independent of \cite{AbouzaidBlumberg2021} and significantly shorter. The precise result that we prove is the following:
 
 \begin{Theorem} \label{thm:Kuranishi}
 Let $(X,\omega,J)$ be a compact symplectic manifold with compatible almost complex structure $J$ and fix a class $\beta \in H_2(X;\bZ)$.  The moduli space of stable genus zero $J$-holomorphic maps with $n$ marked points $\ccMbar_{0,n}(X,J,\beta)$ in class $\beta$ admits a \emph{global Kuranishi chart}: there is a compact Lie group $G$, a topological $G$-manifold $\scrT$ with a $G$-equivariant vector bundle $\scrE\to\scrT$ and a $G$-equivariant section $\frak{s}$ of $\scrE$ with $\ccMbar_{0,n}(X,J,\beta) = \frak{s}^{-1}(0)/G$.
 \end{Theorem}

   The case of higher genus curves is substantially more delicate, and will be treated elsewhere.

   \begin{rem}
Our construction of global Kuranishi charts may be heuristically compared to earlier points of view as follows: one approach to constructing a fundamental chain on $\ccMbar_{0,0}(X,J,\beta)$ is to fix, locally on the moduli space, a collection of Donaldson-type divisors in $X$ whose inverse images under a stable map stabilise the domain curve. An open subset of the moduli space then maps to a space of pointed curves $\ccMbar_{0,N}$.  Our global chart $\scrT$ for $\ccMbar_{0,0}(X,J,\beta)$ comes with a globally defined map to the space of stable maps to projective space $\ccMbar_{0,0}(\bP^d,d)$. Heuristically, we replace the `hard' choice of  Donaldson divisors on $X$ with a `soft' choice of a connection on a line bundle and universal co-ordinate hyperplanes in projective space.  The heart of the matter is that we rigidify a stable map by framing the curve, and not just scattering marked points.      
   \end{rem}

 \begin{Remark}
 We do not prove that $\scrT$ is smooth, but we equip it with a  `fibrewise smooth structure' over a moduli space of domains; that, coupled with classical smoothing theory, suffices for the applications in this paper.  One could circumvent the smoothing theory by reworking some of the results of \cite{Cheng} for topological manifolds as was done in \cite{AbouzaidBlumberg2021}; the essential ingredient in this alternative approach is that the tangent spherical fibration of $\scrT$ admits a vector bundle lift, the proof of which uses the fibrewise smooth structure. The use of the spherical fibration in \cite{AbouzaidBlumberg2021} relied on explicit constructions to globalise local arguments, and it would be difficult to appeal to smoothing theory in that context.
\end{Remark}
 
  For concreteness, we first give a complete proof of Theorem \ref{thm:main}, and then revisit and elaborate on parts of that argument to establish Theorem \ref{thm:main_generalised} at the end of the paper.  We prove the latter result first in the case of complex cobordism, $\bE=MU$, by combining the geometric argument underlying Theorem \ref{thm:main} with deep results in chromatic homotopy theory; see Section \ref{Sec:chromatic}. The proof of the result over $MU$ involves showing that a stable $K(n)$-local version of the  sweepout map from Corollary \ref{Cor:main} is actually nullhomotopic, and not just cohomologically trivial.  This suffices to conclude Theorem \ref{thm:main_generalised} for a general complex-oriented $\bE$.  
  \medskip

The proof of Theorem \ref{thm:main} works directly with certain moduli spaces of sections of $P_{\phi}\to S^2$ for some fixed almost complex structure, and corresponding Kuranishi charts as obtained from Theorem \ref{thm:Kuranishi}. In particular, we do not need to construct `quantum Morava $K$-theory'  or a Seidel representation \cite{Seidel:QH_invertible} into its invertibles.  As with quantum $K$-theory for smooth projective varieties \cite{Lee:QKtheory}, the associativity of quantum Morava $K$-theory (defined as a module over the Novikov ring $\Lambda_{\bK_*}$, where $\bK = K(n)$ is a Morava $K$-theory) relies on working with a deformation of the usual Poincar\'e pairing  on $K^*(X;\Lambda_{\bK_*})$ constructed from moduli spaces of curves with two marked points.  A sequel to this paper \cite{AMS} uses global Kuranishi charts to construct and investigate symplectic quantum (ordinary and Morava) $K$-theory.  
 
  \medskip

\noindent   \textbf{Organisation of the paper.} Section \ref{Sec:outline} reviews and mildly re-interprets  the structure of the classical proof  (over $\bQ$) of Theorem \ref{thm:main} in a way which is well-adapted to generalisation.  After a background Section \ref{Sec:globalKuranishi} on global Kuranishi charts, Section \ref{Sec:Morava_virtual_class} constructs the virtual fundamental class in Morava $K$-theory, and proves Theorem \ref{thm:main} assuming the existence of global Kuranishi charts for moduli spaces of genus zero curves. Section \ref{Sec:curves} then constructs such global charts, proving Theorem \ref{thm:Kuranishi}, and completing the proof of Theorem \ref{thm:main}.  Up to this point, we are careful to work only with the language of generalised cohomology theories, avoiding technicalities with spectra.  Section \ref{Sec:review} reviews background from stable homotopy theory and gives an account of Cheng's theorems from \cite{Cheng}. Section \ref{Sec:chromatic} gives the proof of Theorem  \ref{thm:main_generalised}. The last two sections appeal to foundations of stable homotopy theory that are formulated in terms of classical spectra, so that the reader can use \cite{Adams1974} as a reference for non-equivariant constructions, and \cite{LMSM} for equivariant constructions.

\medskip

\medskip

\noindent \Acknowledgements MA would like to thank Andrew Blumberg for the collaboration \cite{AbouzaidBlumberg2021} which made it possible to imagine this work. 
The authors are grateful to Tobias Barthel, Andrew Blumberg, Tyler Lawson and Oscar Randal-Williams for generous explanations which simplified, strengthened, and corrected the stable homotopy parts of our arguments; to Dhruv Ranganathan and Dror Varolin for helpful conversations; and to Andrew Blumberg, Amanda Jenny, Tyler Lawson, Dusa McDuff, Oscar Randal-Williams
and Mohan Swaminathan for comments on and corrections to an earlier draft. 
\medskip

\noindent MA was partially supported by NSF awards DMS-1609148, DMS-1564172, and DMS-2103805, the Simons Collaboration on Homological Mirror Symmetry, a Simons Fellowship award, and the Poincar\'e visiting professorship at Stanford University.
MM was partially supported by the NSF grant DMS-1811861.  IS was partially supported by Fellowship N/01815X/1 from the EPSRC. 

\section{Examples and illustrations}\label{Sec:Examples}

We collect a number of examples illustrating Theorems \ref{thm:main} and \ref{thm:main_generalised}.

 \begin{Example} Let $X \subset \bC\bP^N$ be a smooth $n$-dimensional complex projective variety which is a $\bC\bP^k$-bundle over an $(n-k)$-fold $S$ such that the $\bC\bP^k$-fibres are linearly embedded. If $2k>n$ then $X$ has `positive defect' \cite{Ein}, i.e. its dual variety has codimension $>1$,  which means that there is a natural algebraic fibration $\Sigma \to X \to \bC\bP^1$ with fibres being smooth hyperplane sections $\Sigma \subset X$. Theorem \ref{thm:main} shows that $H^*(X;\bZ)$ splits (note that $S$, and the hyperplane section $\Sigma$, can have complicated torsion in cohomology). 
  \end{Example}

 Suppose a manifold $L$ admits a free circle action with quotient $Q$.  Viewing the action as a loop of diffeomorphisms, the corresponding fibration $L \to P \to S^2$ has total space $P = (L \times S^3)/S^1$ where $S^1$ acts diagonally, so there is a second fibration $S^3 \to P \to Q$ with Euler class the square of that defining $L\to Q$.  Comparing the Leray-Serre spectral sequences of these fibrations gives many examples of loops of diffeomorphisms for which the cohomology of $P$ splits rationally and not integrally.  Few symplectic manifolds admit free symplectic circle actions (the quotient is always the mapping torus of a symplectomorphism \cite{FGM}), and Hamiltonian circle actions always have fixed points, but the above examples are relevant to loops of diffeomorphisms of a manifold $X$ which happen to preserve a submanifold $L$: functoriality of the Leray-Serre spectral sequence and Theorem \ref{thm:main} yield constraints on the image of the map $\pi_1(\Ham(X,L)) \to \pi_1(\Diff(L))$.
 
 \begin{Example} \label{Ex:submanifold_preserved}
The Hopf circle action on $L = \bR\bP^3$ yields a fibration $P_{\phi} \to S^2$ whose cohomology splits over $\bQ$ but not over $\bZ/2$. If $L = \bR\bP^3 \subset X$ induces a surjection $H^*(X;\bZ/2) \to H^*(\bR\bP^3;\bZ/2)$, then an element of $\pi_1\, \Diff(X,L)$ which induces the Hopf rotation on $L$ cannot lift to $\pi_1\,\Ham(X,L)$. 
  \end{Example}

  \begin{Example}
  Consider the loop of diffeomorphisms $\phi$ defining the standard Hopf circle action on $\bR\bP^5$.  Since this is a free circle action, as above we have the dual descriptions of the associated bundle 
  \begin{equation} \label{eqn:two_descriptions}
  \bR\bP^5 \to P_{\phi} \to S^2 \quad \text{and} \quad S^3 \to P_{\phi} \to \bC\bP^2.
  \end{equation}
    By comparing the Serre spectral sequences of these two fibrations, one sees that the sweepout map for $\phi$ is non-trivial over $\bZ$: there is exactly one non-trivial differential 
    \begin{equation}
    \bZ = H^0(S^2;H^5(\bR\bP^5)) \to H^2(S^2;H^4(\bR\bP^5)) = \bZ/2
    \end{equation}
   It follows that the sweepout over $\bZ$ for the \emph{square} of the action vanishes, and the cohomology of $P_{\phi^2}$ splits additively over $\bZ$. However, it does not split additively in complex $K$-theory.  By \cite[Proposition 2.7.7]{Ktheory},
    \begin{equation}
    KU^i(\bR\bP^5) = \begin{cases} \bZ \oplus \bZ/4 & i \text{ even} \\ \bZ & i \text{ odd} \end{cases}
    \end{equation}
A mildly involved comparison of the two Atiyah-Hirzebruch-Serre spectral sequences associated to \eqref{eqn:two_descriptions} shows that there is a non-trivial component of the differential 
\begin{equation}
\bZ = H^0(S^2; KU^1(\bR\bP^5)) \to H^2(S^2; KU^0(\bR\bP^5)) = \bZ \oplus \bZ/4 \to \bZ/4
\end{equation}
which is a surjection $\bZ \to \bZ/4$, and the corresponding arrow is also non-trivial in the spectral sequence for $P_{\phi^2}$. Such examples become relevant to symplectic topology if one considers Hamiltonian loops preserving submanifolds, as in Example \ref{Ex:submanifold_preserved}.
\end{Example}

  \begin{Example}
  The additive splitting of $H^*(P_{\phi};\bK)$ does not hold for all generalised cohomology theories. Consider the Hirzebruch surfaces $F_0 = \bC\bP^1\times \bC\bP^1$ and $F_1 = \bC\bP^2 \# \overline{\bC\bP}^2$, which are both $S^2$-bundles over $S^2$ associated to appropriate Hamiltonian loops. In \cite{Bahri-Bendersky}, the additive $KO$-theory of any toric manifold $X$ is computed, starting from an Adams-type spectral sequence 
  \begin{equation}
  \Ext_{\scrA(1)}^{s,t}(H^*(X;\bZ/2);\bZ/2) \Rightarrow ko_{s-t}(X)
  \end{equation}
  where $\scrA(1)$ is the subalgebra of the Steenrod algebra generated by $Sq^1$ and $Sq^2$ and $ko$ denotes connective $KO$-theory. Let $S$ denote the $\scrA(1)$-module which is rank one in degree $0$ with trivial operations, and $M$ the $\scrA(1)$-module generated by elements $x$ of degree $0$ and $y$ of degree $2$ with $Sq^2(x)=y$.  Then $H^*(X;\Z/2)$ is a direct sum of shifts of $S$ and $M$ as an $\scrA(1)$-module, and this decomposition completely determines $KO^*(X)$ additively. Concretely, for $F_0$ and $F_1$ one has
  \begin{equation}
  H^*(F_0;\bZ/2) = S \oplus S[2] \oplus S[2] \oplus S[4], \qquad H^*(F_1;\bZ/2) = S \oplus S[2] \oplus M[2]
  \end{equation}
  (the cohomology ring of $F_0$ is even so $Sq^2$ vanishes, whilst for $F_1$ it is non-trivial on the exceptional divisor); these are distinguished by their 2-torsion,  and it follows that $H^*(F_1;KO)$ is not additively split.
  \end{Example}

\section{The argument in outline\label{Sec:outline}}

This section reviews and simplifies the classical argument (for the rational cohomology analogue of Theorem \ref{thm:main}) in a format which will guide the translation from classical cohomology to generalized cohomology.   Throughout the paper, we will use the following

\begin{Notation} 
For a subspace $A \subset Y$, we set $H^*(Y|A;\bK) := H^*(Y, Y\backslash A;\bK)$.
\end{Notation}

 In this section, $\bK$ denotes a ring (in later sections, it will typically be a ring spectrum).

\subsection{Background} Keep the previous notation; a loop $\phi \in \pi_1(\Ham(X,\omega))$ gives rise to a Hamiltonian fibre bundle $P_{\phi} \to S^2$ with fibre $X$. 

\begin{Lemma} \label{Lem:early}
The following are equivalent:
\begin{enumerate}
\item \label{item sweepout} The map $\delta_{\phi}: H_*(X;\bK) \to H_{*+1}(X;\bK)$ vanishes.
\item \label{item homlogy inj} The natural map $H_*(X;\bK) \to H_*(P_\phi;\bK)$ is injective.
\item \label{item ss degen} The cohomology spectral sequence for the fibration $P_\phi \to S^2$ degenerates.
\item \label{item cohomlogy surj} The restriction $H^*(P_{\phi};\bK) \to H^*(X;\bK)$ is surjective.
\end{enumerate} 
If $\bK$ is a field then the above conditions are equivalent to the condition that
$H^*(P_{\phi};\bK)$ splits additively as $H^*(X;\bK) \otimes_{\bK} H^*(S^2;\bK)$. 
\end{Lemma}
\begin{proof} 
The sweepout map $\delta_{\phi}$ arises in the Wang sequence
\begin{equation}
	\begin{tikzcd}
		H_{i}(X;\bK) \arrow[rr, "\delta_{\phi}"] && H_{i+1}(X;\bK) \arrow[rr, "\mathrm{incl}_*"] && H_{i+1}(P_{\phi};\bK) \arrow[r] & H_{i-1}(X;\bK).
	\end{tikzcd}
\end{equation}
It vanishes if and only if $\mathrm{incl}_*$
is injective and hence
conditions
(\ref{item sweepout})
and
(\ref{item homlogy inj})
are equivalent.
The homology Serre spectral sequence for $P_\phi \to S^2$
has a unique possible differential
which is equal to the map $\delta_\phi$ above.
Hence it degenerates if and only
if $\delta_\phi$ vanishes.
By Poincar\'{e}
duality on the total space $P_{\phi}$, 
the homology Serre spectral sequence for $P_\phi \to S^2$
degenerates if and only
if the corresponding cohomology
one does.
Therefore conditions
(\ref{item sweepout})
and (\ref{item ss degen})
are equivalent.
The restriction map
$H^*(P_{\phi};\bK) \to H^*(X;\bK)$
is surjective
if and only if the cohomology spectral sequence degenerates, 
by looking at the cohomology Wang sequence.
Hence (\ref{item ss degen}) and (\ref{item cohomlogy surj}) are equivalent.
Finally if $\bK$ is a field
then the cohomology spectral sequence
degenerates if and only if
$H^*(P_{\phi};\bK) \cong H^*(X;\bK) \otimes_{\bK} H^*(S^2;\bK)$.
 \end{proof}

 The basic structure of the proof in \cite{LMP}, when $\bK = \bQ$, is to use holomorphic sections of $P_{\phi} \to S^2$ to build a map $\widehat{\Phi}_{\phi}: H^*(X;\bK) \to H^*(P_{\phi};\bK)$ (more properly, a map with Novikov coefficients $\Lambda_{\bQ}$)   whose composition with restriction 
 \begin{equation}
\Phi_{\phi}: H^*(X;\bK) \stackrel{\widehat{\Phi}_{\phi}}{\longrightarrow} H^*(P_{\phi};\bK) \stackrel{\mathrm{res}}{\longrightarrow}H^*(X;\bK)
\end{equation} 
satisfies:
\begin{itemize}
\item $\Phi_{\id}$ is an isomorphism;
\item $\Phi_{\phi} \circ \Phi_{\phi^{-1}} = \Phi_{\id}$.
\end{itemize}
The argument above is inspired by the existence of the Seidel homomorphism \cite{Seidel:QH_invertible}
\begin{equation}\label{eqn:Seidel}
\pi_1\Ham(X,\omega) \longrightarrow QH^*(X,\omega;\Lambda_{\bR})^{\times}
\end{equation}
to invertibles in quantum cohomology (in general defined over the universal one-variable Novikov field over $\bR$).  However, to prove the cohomological splitting of Theorem \ref{thm:main} (or its rational cohomology counterpart) one can actually get away with somewhat less.  This will be helpful, since we will later work in a setting in which $\bK$ is not a coefficient field but an extraordinary cohomology theory.

\subsection{The degeneration\label{Sec:degeneration}}

Fix as before a Hamiltonian loop $\phi \in \pi_1(\Ham(X,\omega))$, and let $\bK$ be a commutative ring; we shall eventually be interested in the case $\bK = \bZ$, but working over the field $\bK = \bF_p$ is an interesting test case for our methods which avoids some complications.

Let $\bS$ denote the one-point blow-up of $\bC\bP^1\times \bC\bP^1$; we view $\bS$ as the total space of a ruled surface over $B = \bC\bP^1$ with a unique singular fibre over $0 \in B$, in particular with a smooth fibre at $\infty \in B$.  Let $\pi_B : \bS \to B$ be this ruling. We equip $\bS$ with a K\"ahler form.

\begin{Lemma} \label{lemma fibration existence}
There is a symplectic fibration $\pi_{\bS}: \widetilde{P} \to \bS$ with fibre $(X,\omega)$, with the property that 
\begin{enumerate}
\item on a neighborhood $W_\infty$ of $\infty \in B$, $\widetilde{P}|_{\pi_B^{-1}(W_\infty)}$ is isomorphic to the trivial
symplectic fibration
$(S^2 \times W_\infty \times X, \omega_{S^2} \oplus \omega_{S^2}|_{W_\infty} \oplus \omega_X)$ in such a way that $\pi_B|_{\pi_B^{-1}(W_\infty)}$ corresponds to the projection map $S^2 \times W_\infty \to W_\infty$;

\item $P_0 = P_{\phi} \cup_X P_{\phi^{-1}} \to \bC\bP^1 \vee \bC\bP^1$, where each component of the reducible singular fibre carries the canonical deformation class of symplectic structure.
\end{enumerate}
\end{Lemma}

\begin{proof} See \cite[Section 3]{McDuff}. \end{proof}

\begin{figure}[ht]
\begin{center}
\begin{tikzpicture}[scale=0.9]

\draw[semithick] (-2,0) -- (-2,5);
\draw[semithick] (-1,0) -- (-1,5);

\draw[semithick] (0,5) arc (180:200:10);
\draw[semithick] (0,0) arc (180:160:10);

\draw[semithick] (1,0) -- (1,5);
\draw[semithick,blue] (2,0) -- (2,5);
\draw[semithick] (3,0) -- (3,5);
\draw[semithick,red] (-2.5,4.5) -- (3.5,4.5);

\draw[fill] (2,4.5) circle (0.1);
\draw[fill] (0.0,4.5) circle (0.1);

\draw (1.7,4.8) node {$X_{t}$};
\draw (0.4,4.8) node {$X_0$};
\draw (-0.1,-.15) node {$P_{\phi^{-1}}$};
\draw (-0.1,2.85) node {$P_{\phi}$};
\draw (1.4,5.3) node {\color{blue} $P_{t_0} \cong X\times S^2$};
\draw (-3.5,4.5) node {\color{red} $(S^2 \times X)_{hor}$};

\draw[dashed] (0,-.5) -- (0,-2.2);
\draw[dashed] (2,-.5) -- (2,-2.2);

\draw[thick, arrows = ->] (0.5,-1) -- (0.5,-2);
\draw[semithick] (-2.5,-3) -- (3.5,-3);
\draw (-3.5,-3) node {$B = \mathbb{P}^1$};
\draw[thick, arrows = ->] (-3,2) -- (-3,-2.2);
\draw (-3, 2.5) node {$\mathbb{S}$};

\draw (-5,3.9) node {$\widetilde{P}$};
\draw[thick, arrows = ->] (-4.7,3.7) -- (-3.2, 2.7);

\draw[fill] (2,-3) circle (0.1);
\draw[fill] (0,-3) circle (0.1);
\draw (0.2,-2.7) node {$0$};

\draw (2.2,-2.7) node {$t$};

\draw [dashed,purple] (2.5,5) rectangle (3.5,0);
\node at (2.5,-3) {$($};
\node at (3.5,-3) {$)$};
\node at (3.1,-3.5) {$W_\infty$};
\node at (3.3,-0.4) {\color{purple} $\pi_B^{-1}(W_\infty)$};
\node at (4,-1) {\color{purple} $=S^2 \times W_\infty \times X$};
\draw[fill] (3,-3) circle (0.1);
\node at (3.1,-2.7) {$\infty$};

\draw[fill] (3,4.5) circle (0.1);
\node at (4.1,5.3) {$X_\infty$};
\draw [->](3.9,5) -- (3.2,4.6);
\end{tikzpicture}
\caption{The total space of the degeneration $\widetilde{P} \to \mathbb{S}$\label{Figure:degeneration}}
\end{center}
\end{figure}
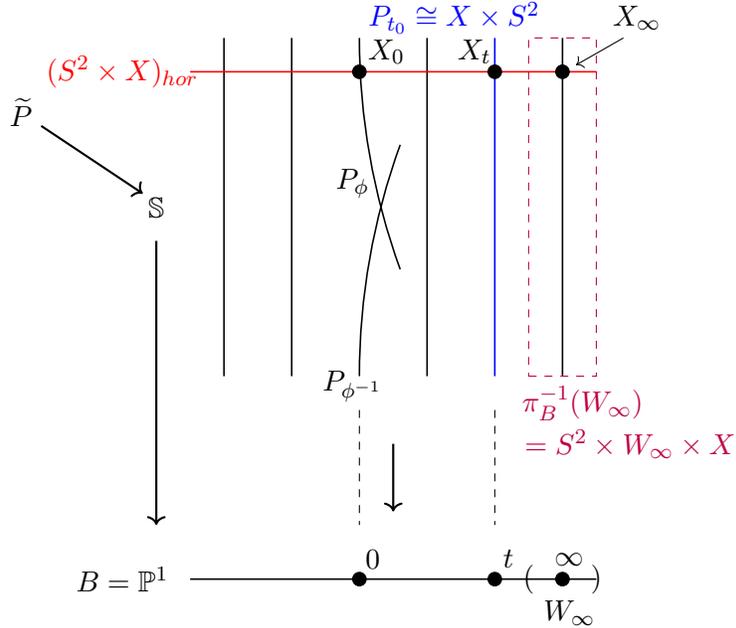

Let $S^2_{hor} \subset \bS$ be the image of a holomorphic section of the fibration $\bS \to B$ passing through a point in $P_\phi \backslash P_{\phi^{-1}}$.
The preimage of $S^2_{hor}$ in $\widetilde{P}$ is a Hamiltonian fibration over $S^2_{hor}$ which is isotopic to a trivial one.
We let $(X \times S^2)_{hor}$ be this preimage.
We let $X_t := \pi_{\bS}^{-1}(\pi_B^{-1}(t))$ be the fiber of $(S^2 \times X)_{hor}$ over $t$ for each $t \in B \cong S^2_{hor}$.
We also define $P_t := \widetilde{P}|_{\pi_B^{-1}(t)}$ for each $t \in B$.
For $t \in B \backslash 0$, $P_t$ is isotopic as a Hamiltonian fibration over $\pi_B^{-1}(t)$ to $(S^2 \times X)$
and is canonically equal to such a product when $t \in W_\infty$.
These objects are pictured in the Figure \ref{Figure:degeneration}.

\begin{Lemma} \label{Lem:Reformulate}
If the natural map $\Theta: H^*(\widetilde{P}|X_t;\bK) \to H^*(\widetilde{P};\bK)$ is injective for some $t \in B$,
then the sweepout map
$\delta_{\phi}: H_*(X;\bK) \to H_{*+1}(X;\bK)$
vanishes.
	\end{Lemma}

\begin{proof}
Let $2n=\dim_{\bR}(X)$. 
The map $\Theta$ does not depend on the choice of fibre $X_t$ of $\tilde{P} \to \bS$, and hence does not depend on $t$; let $t = 0$.
Let $\iota : X_0 \lra{} P_\phi$
and $\iota_\phi : P_\phi \lra{} \widetilde{P}$
be the inclusion maps. 
Noting that all spaces involved are canonically $\bK$-oriented, we have a commutative diagram
\begin{equation}
  \begin{tikzcd}
H_i(X_0) \ar[r,"(\iota_\phi)_* \circ \iota_*"] \ar[d,"PD"] & H_i(\widetilde{P})  \ar[d,"PD"] \\
H^{2n-i}(X_0) \ar[r] \ar[d,"\mathrm{Thom}"]& H^{2n+4-i}(\widetilde{P}) \ar[d,"="] \\
H^{2n+4-i}(\widetilde{P}|X_0)  \ar[r,"\Theta"] & H^{2n+4-i}(\widetilde{P})    
  \end{tikzcd}
\end{equation}
in which the vertical maps are isomorphisms, the horizontal maps at top and bottom are the maps induced by the natural inclusions, and we use the fact that, for the regular fibre $X = X_t \subset \widetilde{P}$ we have $\widetilde{P} / (\widetilde{P} \backslash X) \, \cong \, \Sigma^4(X_+)$. 
Since $\Theta$ is injective,
we have by the diagram above that
$(\iota_\phi)_* \circ \iota_*$
is injective and hence
$\iota_*$ is injective.
The result now follows from Lemma \ref{Lem:early}.
\end{proof}

\subsection{Moduli spaces of $J$-holomorphic curves}
\label{sec:moduli-spaces-j}

Fix a compatible almost complex structure $J$ on $\widetilde{P}$ so that the natural map $\pi_\bS : \widetilde{P} \to \bS$ is $J$-holomorphic and so that
it is a product almost complex structure on $\widetilde{P}|_{\pi_{\bS}^{-1}(W_\infty)} = X \times S^2 \times W_\infty$ (cf. the properties in Lemma \ref{lemma fibration existence}).
We let $\beta \in H_2(\widetilde{P};\bZ)$ denote the homology class of the constant curve $\{pt\} \times S^2 \times \{\infty\}$ in $X \times S^2 \times W_\infty$. Let $\ccMbar = \ccMbar_{0,2}(\beta)$ denote the moduli space of stable genus zero $J$-holomorphic curves $u: \bC\bP^1 \to \widetilde{P}$ which lie in the homology class $\beta$. We have an evaluation map
\begin{equation} \label{eqn evaluationev}
ev: \ccMbar \to \widetilde{P} \times \widetilde{P}
\end{equation}
and we let $\ccMbar_h \subset \ccMbar$ be the inverse image $ev^{-1}(\widetilde{P} \times (S^2 \times X)_{hor})$ (the subscript $h$ is a reminder that one marked point should evaluate into the `horizontal' copy of $S^2 \times X$).

\begin{Lemma} \label{lemma virtual dimension}
$\ccMbar_h$ has virtual dimension $\dim_\bR(\widetilde{P}) = 2n+4$.
\end{Lemma}

\begin{proof} 
	Since $\beta$ is represented by $\{pt\} \times S^2 \times \{\infty\}$ in $X \times S^2 \times W_\infty$,
	the homology class $\beta \in H_2(\widetilde{P})$ satisfies $\langle c_1(T\widetilde{P}), \beta\rangle = 2$. 
	Therefore the space of stable genus $0$ parametrized curves in $\widetilde{P}$ in class $\beta$ has real virtual dimension $\dim(\widetilde{P}) + 4 = 2n+8$, where $2n=\dim_\bR(X)$. We impose a divisorial condition on one evaluation map, and there is a residual reparametrization action of $\bC^*$ since the domains have only two marked points, which brings down the virtual dimension to $2n+4$.
\end{proof}

Recall, by Lemma \ref{lemma fibration existence}, we have an identification
$\widetilde{P}|_{\pi_B^{-1}(W_\infty)} = S^2 \times W_\infty \times X$
for some neighborhood $W_\infty \subset B$ of $\infty \in B = \C P^1$, 
where $S^2 \times W_\infty$ is identified with $\pi_B^{-1}(W_\infty)$ in the natural way so that $\pi_{\bS}$ is the projection map to $S^2 \times W_\infty$.
This gives us a natural identification:
\begin{equation}
	(S^2 \times X)_{hor} \cap (S^2 \times W_\infty \times X) = W_\infty \times X.
\end{equation}
After fixing such identifications, we define
$\ccMbar_{W_\infty} := ev^{-1}((S^2 \times W_\infty \times X) \times (W_\infty \times X))$
and
\begin{equation}
	ev|_{W_\infty} : \ccMbar_{W_\infty} \to (S^2 \times W_\infty \times X) \times (W_\infty \times X), \quad ev|_{W_\infty}(x) := ev(x).
\end{equation}

\begin{lemma} \label{Lem:good_open_chart}
	The moduli space $\ccMbar_{W_\infty}$ is regular
	and
	$ev|_{W_\infty}$ is an orientation preserving diffeomorphism onto the diagonal submanifold
	\begin{equation} \label{equation diagonal map}
		S^2 \times W_\infty \times X \to (S^2 \times W_\infty \times X) \times (W_\infty \times X),
	\end{equation}
$$
(x,y,z) \to ((x,y,z),(y,z)), \ \forall \ (x,y,z) \in (S^2 \times W_\infty \times X).
$$
\end{lemma}
\begin{proof}
The composition of a stable map representing a point of $\ccMbar_{W_\infty}$ with the projection map $\pi_{\bS}$
is a holomorphic curve with image equal to a fiber of $\pi_B$.
Hence any such curve has image contained in $S^2 \times W_\infty \times X$.
Our result now follows from the fact that $J$ is a product in $S^2 \times W_\infty \times X$.
\end{proof}

\smallskip

Consider now the evaluation map 
\begin{equation}
ev: \ccMbar_h \to \widetilde{P} \times (S^2 \times X)_{hor}
\end{equation}
 introduced above.
 Let 
\begin{equation}
f: \widetilde{P} \times (S^2 \times X)_{hor} \to \bS \times X \qquad \mathrm{and} \qquad F: \widetilde{P} \times (S^2 \times X)_{hor} \to \widetilde{P}
\end{equation}
 be respectively the map obtained from the projections $\widetilde{P} \to \bS$ and $X\times S^2 \to X$, for $f$, and the projection to the first factor for $F$.  We also have projections
 \begin{equation}
 Q_\bS: \bS\times X \to \bS \qquad \mathrm{and} \qquad Q_X: \bS\times X \to X.
 \end{equation}
 The maps fit into the following diagram:
 \begin{equation} \label{eqn:not_too_large}
   \begin{tikzcd}
 \ccMbar_h \ar[d,"ev"] & &  \\ 
 \widetilde{P} \times (S^2 \times X)_{hor} \ar[d,"F"] \ar[rr,"f"] &&  \bS\times X \ar[dl,"Q_{\bS}"] \ar[dr,"Q_X"] &  \\
 \widetilde{P} \ar[r,"\pi_\bS"] & \bS && X 
\end{tikzcd} \end{equation}
 from which we see that 
 \begin{equation}
 Q_{\bS}\circ f\circ ev = \pi_\bS \circ F \circ ev.
 \end{equation}
Since we know that
 \begin{equation}
 \vdim(\ccMbar_h) = \dim_{\bR} (\bS\times X) = \dim_{\bR}(\widetilde{P}),
 \end{equation}
in a `regular' situation in which the moduli space $\ccMbar_h$ was everywhere transversely cut out and there was no bubbling $f\circ ev$ would be a map of closed manifolds of the same dimension.
 
\begin{Lemma} \label{Lem:injection_suffices}
If $f\circ ev: \ccMbar_h \to \bS \times X$ induces an injection on cohomology, then the sweepout map
$\delta_{\phi}: H_*(X;\bK) \to H_{*+1}(X;\bK)$
vanishes.
\end{Lemma}

\begin{proof}
Let $t \in \pi_B^{-1}(\infty)$.
The map $H^*(\bS|t) \to H^*(\bS)$ is injective (and defines the cohomological fundamental class of $\bS$), so
\begin{equation}
\iota: H^*(\bS\times X | \{t\}\times X) \to H^*(\bS\times X)
\end{equation}
is injective (this follows from the K\"unneth theorem if $\bK$ is a field, but more generally one can use the fact that the restriction $H^*(S^4) \cong H^*(\bS|t) \to H^*(\bS)$ is split by the choice of orientation class for $\bS$).  Thus, if $(f\circ ev)$ induces a cohomology injection, then $(f\circ ev)^* \circ \iota$ is also injective. 
By Lemma \ref{Lem:good_open_chart},
the maps $f \circ ev$ and $F \circ ev$ agree along $S^2 \times W_\infty \times X$.
Hence
\begin{equation}
(f\circ ev)^* \circ \iota = (F \circ ev)^* \circ \Theta
\end{equation}
where $\Theta: H^*(\widetilde{P}|X_\infty) \to H^*(\widetilde{P})$ comes from Lemma \ref{Lem:Reformulate} and where
we have naturally identified:
\begin{equation}
	H^*(\widetilde{P}|X_\infty) = H^*(S^2 \times W_\infty \times X|\{pt\} \times \{\infty\} \times X) = H^*(\bS \times X| \{t\} \times X).
\end{equation}
Therefore, $(F \circ ev)^* \circ \Theta$ is injective and so $\Theta$ is too.
\end{proof}

The moduli space $\ccMbar_h$ is typically poorly behaved. However if it admits a virtual fundamental class, one can then approach the hypothesis of the previous Lemma in terms of the map $f \circ ev$ being `degree one' in a suitable sense via Lemma \ref{Lem:good_open_chart}. 

\subsection{Splitting\label{Sec:Splitting}}
\label{sec:splitting}

The previous set-up shows that, given a virtual fundamental class for $\ccMbar_h$ of the correct virtual dimension, and a regular open subset of the moduli space as in Lemma \ref{Lem:good_open_chart}, one can satisfy the conditions of Lemma \ref{Lem:injection_suffices} and hence prove that the sweepout map $\delta_\phi$ vanishes.  This implies $H^*(P_{\phi};\bK) \cong H^*(X;\bK) \otimes H^*(S^2;\bK)$ if $\bK$ is a field by Lemma \ref{Lem:early}, but not when $\bK = \bZ$. 

For that we require a stronger condition, which is that the natural map:
\begin{equation} \label{eqn section}
	\rho : H^*(P_\phi;\bK) \to H^*(X;\bK)
\end{equation}
is a split epimorphism, so there is a $\bK$-module map $s : H^*(X;\bK) \to H^*(P_\phi;\bK)$
called the \emph{splitting map}
satisfying $\rho \circ s = \id$.
This, in turn, would give us an isomorphism:
\begin{equation} \label{eqn decomposition}
\begin{tikzcd}
	H^*(S^2;\bK) \otimes_{\bK} H^*(X;\bK) \cong H^*(X;\bK) \oplus H^*(P_\phi|X;\bK) \arrow[r, "{s\oplus e}"] & H^*(P_\phi;\bK)
	\end{tikzcd}
\end{equation}
with $e$ the natural map. 
If $\bK = \Z$ this would prove Theorem \ref{thm:main}.

Consider the moduli space
\begin{equation} \label{equation mbarbullet}
	\ccMbar_\bullet := ev^{-1}(P_\bullet \times (S^2 \times X)_{hor}) \subset \ccMbar_h
\end{equation}
together with the evaluation map
\begin{equation} \label{equation evalutionbullet}
	ev_\bullet : \ccMbar_\bullet \to P_\bullet \times (S^2 \times X)_{hor}
\end{equation}
where $\bullet = \phi$ or $\infty$.
We wish to use the moduli space $\ccMbar_\phi$ to build a map $H^*(X;\bZ) \to H^*(P_{\phi};\bZ)$ which splits the natural restriction map.
The proof of the splitting property will also use the moduli spaces $\ccMbar_h$ and $\ccMbar_\infty$ in a similar manner.

\begin{remark}
Note that the curves in $\ccMbar_\phi$ do not have image contained in $P_\phi$. Instead they have image in $P_\phi \cup_X P_{\phi^{-1}}$
and so they are not the same holomorphic sections that are used in \cite{LMP, McDuff}.
\end{remark}

For the moment, we will work under the unrealistic assumption that the map
\begin{equation} \label{eqn unrealistic}
	\ccMbar_h \lra{ev} \widetilde{P} \times (S^2 \times X)_{hor} \lra{pr} \widetilde{P}
\end{equation}
is a submersion.
This implies that $\ccMbar_h$, $\ccMbar_\phi$ and $\ccMbar_\infty$
are all oriented manifolds with fundamental classes
$[\ccMbar_\bullet] \in H_{vd}(\ccMbar_\bullet;\bK)$, $\bullet = h$, $\phi$ or $\infty$.
Hence we get the following commutative diagram:
\begin{equation} \label{eqn:push_pull}
\begin{tikzcd}
	{H^*((S^2 \times X)_{hor})} & {H^*(\ccMbar_h)} & {H_{vd-*}(\ccMbar_h)} & {H_{vd-*}(\widetilde{P})} & {H^*(\widetilde{P})} \\
	{H^*((S^2 \times X)_{hor})} & {H^*(\ccMbar_\bullet)} & {H_{vd_\bullet-*}(\ccMbar_\bullet)} & {H_{vd_\bullet-*}(P_\bullet)} & {H^*(P_\bullet)}
	\arrow["{ev^*}", from=1-1, to=1-2]
	\arrow["\lambda", from=1-2, to=1-3]
	\arrow["{ev_*}", from=1-3, to=1-4]
	\arrow["D", from=1-4, to=1-5]
	\arrow["\lambda", from=2-2, to=2-3]
	\arrow["{ev_*}", from=2-3, to=2-4]
	\arrow["D", from=2-4, to=2-5]
	\arrow["res", from=1-5, to=2-5]
	\arrow["{\cap Thom_{\ccMbar_\bullet}}", from=1-3, to=2-3]
	\arrow["{\cap Thom_{P_\bullet}}", from=1-4, to=2-4]
	\arrow["res", from=1-2, to=2-2]
	\arrow[Rightarrow, no head, from=2-1, to=1-1]
	\arrow["{ev^*}", from=2-1, to=2-2]
	\arrow["{\Psi_\bullet}"', curve={height=18pt}, from=2-1, to=2-5]
	\arrow["{\Psi_h}", curve={height=-30pt}, from=1-1, to=1-5]
\end{tikzcd}
\end{equation}
defining the push-pull maps $\Psi_h$ and $\Psi_\bullet$ for $\bullet = \phi$ or $\infty$
where $\lambda$ and $D$ are Poincar\'{e} duality maps.
Now consider the following commutative diagram:
\begin{equation} \label{eqn pushpull compatibility}
	\begin{tikzcd}
		{H^*(X)} & {H^*((S^2 \times X)_{hor})} & {H^*(\widetilde{P})} & {H^*(X)} \\
		&& {H^*(P_\bullet)}
		\arrow["{pr^*}"', from=1-1, to=1-2]
		\arrow["{\Psi_h}"', from=1-2, to=1-3]
		\arrow["res"', from=1-3, to=1-4]
		\arrow["{\Psi_\bullet}"', from=1-2, to=2-3]
		\arrow["res"', from=1-3, to=2-3]
		\arrow["res"', from=2-3, to=1-4]
	\end{tikzcd}
\end{equation}
By Lemma \ref{Lem:good_open_chart},
we have that the map $\Psi_\infty$ is equal to the map $pr_X^*$
where $pr_X$ is the projection map from $P_\infty = S^2 \times X$ to $X$.
This implies that $res \circ \Psi_\bullet \circ pr^*$ is the identity map in the commutative diagram above
with $\bullet = \infty$ and hence $res \circ \Psi_h \circ pr^*$ is the identity map,
which in turn implies that $res \circ \Psi_\phi \circ pr^*$ is the identity.
Hence \eqref{eqn section} is a split epimorphism with splitting map $s := \Psi_\phi \circ pr^*$, proving Theorem \ref{thm:main} if $\bK = \Z$.

It remains to drop the unrealistic assumption above.
In Section \ref{Sec:conclusions}, and again in Section \ref{Sec:nullhomotopy}, we will consider push pull maps $\Psi_h$ and $\Psi_\phi$
from \eqref{eqn:push_pull} in which $\ccMbar_h$ is replaced by a global Kuranishi chart for which the necessary smoothness and transversality properties do hold. This will lead to the cohomological splitting of Theorem \ref{thm:main} over $\bZ$, rather than just the degeneration of the Serre spectral sequence.

\section{Global Kuranishi Charts \label{Sec:globalKuranishi}}

A global Kuranishi chart presents a space as the zero set of a vector bundle section
over an orbifold, itself presented as a global Lie group quotient.
We will use such global Kuranishi charts of moduli spaces of genus zero curves 
to equip them with a  virtual fundamental class with respect
to certain generalized cohomology theories.
In order to do this, our global Kuranishi charts need additional structure.
In this section we will define global Kuranishi charts, describe basic operations on them
and give some conditions needed to put additional structures on such charts.

\subsection{Initial definitions} \label{section initial defns}

\begin{Definition}
	A \emph{global Kuranishi chart} is a tuple $(G,\scrT,E,s)$ where
	\begin{enumerate}
		\item  $G$ is a compact Lie group;
		\item $\scrT$ is a  topological manifold admitting a continuous action of $G$ with finite stabilisers;
		\item $E \to \scrT$ is a $G$-vector bundle;  
		\item $s: \scrT \to E$ is a $G$-equivariant section. 
	\end{enumerate}
	A \emph{global Kuranishi chart} for a metric space $M$ is a global Kuranishi chart $(G,\scrT,E,s)$ as above
	together with a `footprint' homeomorphism $M \to s^{-1}(0) / G$.
	The \emph{virtual dimension} of $(G,\scrT,E,s)$ is defined to be $\vdim(G,\scrT,E,s) := \dim(\scrT) - \rank(E) - \dim(G)$.

	A global Kuranishi chart $(G',\scrT',E',s')$ is \emph{homeomorphic} to $(G,\scrT,E,s)$
	if $G$ is isomorphic to $G'$ and there is a $G$-equivariant homeomorphism $\scrT \to \scrT'$
	and a $G$-vector bundle isomorphism $E \to E'$ covering this homeomorphism sending $s$ to $s'$.
	
\end{Definition}

We will refer to $\scrT$ as the \emph{thickening}. 
If it is clear what the `footprint' homeomorphism is, then we will omit it from our notation
and just refer to the quadruple $(G,\scrT,E,s)$ as a Kuranishi chart for $M$.
If it is clear which Kuranishi chart is associated to $M$,
then we will just write $\vdim(M)$ instead of $\vdim(G,\scrT,E,s)$.

\begin{Example}
	If $M$ is a manifold then $(0,M,M,0)$ is a global Kuranishi chart where $0$ is the trivial Lie group
	and $M$ is viewed as the zero dimensional vector bundle over $M$.
\end{Example}

In general moduli spaces of holomorphic curves will be  highly singular.
However, for computational purposes, it is helpful to look at nicer
regions of these moduli spaces as in the definitions below.   In the following definition, we always lift subsets of $s^{-1}(0)/G$ to $s^{-1}(0)$ by the natural quotient map.

\begin{defn} \label{Defn free regular}
Let $(G,\scrT,E,s)$ be a Kuranishi chart.
\begin{itemize}
\item 
	We say that $(G,\scrT,E,s)$ is \emph{free} along $C \subset s^{-1}(0)/G$ if
	the $G$-action is free along the preimage of $C$ in $s^{-1}(0)$. Such a chart is \emph{free} if it is free everywhere (i.e. along $s^{-1}(0)/G$).
	
\item	A global Kuranishi chart $(G,\scrT,E,s)$
	is \emph{smooth}
	if both $\scrT$ and $E$ are smooth and if the corresponding $G$-actions on $\scrT$ and $E$ are smooth. Note that the section $s$ need not be smooth.

\item 
	A smooth global Kuranishi chart $(G,\scrT,E,s)$ is \emph{regular}
	along a subset $C \subset s^{-1}(0)/G$ if $s$ is smooth and transverse to $0$ along the preimage of $C$ in $s^{-1}(0)$.
	It is \emph{regular} if it is regular along $s^{-1}(0)/G$.
	\end{itemize}
\end{defn}

Note that for a smooth global Kuranishi chart,
$s^{-1}(0)/G$ is naturally a smooth orbifold of dimension $\vdim(G,\scrT,E,s)$ along its locus of regular points and a smooth manifold along the set of points which are both regular and free.

\begin{Example}
	If $M$ is a smooth orbifold with a bounded number of stabilizer groups,
	then by \cite[Corollary 1.3]{pardon2019enough} there is a smooth Kuranishi chart
	$(U(m),\scrT,\scrT,0)$ for $M$ where $m \in \bN$.
	In other words, this orbifold is expressed as a global quotient by a Lie group.
\end{Example}

Global Kuranishi charts $(G,\scrT,E,s)$ for a given metric space $M$
are not unique. We now describe operations
on $(G,\scrT,E,s)$
which we will later show preserve the virtual fundamental class.

\begin{enumerate}
	
	\item (Germ equivalence).
	Take a $G$-invariant open neighborhood $U \subset \scrT$ of $s^{-1}(0)$, and replace $(G,\scrT,E,s)$ by $(G,U,E|_U, s|_U)$.
	
	\item (Stabilization).
	 Let $p: W \to \scrT$ be a $G$-equivariant vector bundle, and replace $(G,\scrT,E,s)$ by $(G,W, p^*E \oplus p^*W, p^*s\oplus\Delta)$, where $\Delta$ is the tautological `diagonal' section of $p^*W \to W$. We call this a {\it stabilization} of $(G,\scrT,E,s)$ by $W$.
	
	\item (Group enlargement).
	 Let $G'$ be a compact Lie group,  let $q: P \to \scrT$ be a $G$-equivariant principal $G'$-bundle, and replace $(G,\scrT,E,s)$ by $(G\times G', P, q^*E, q^*s)$.
	
\end{enumerate}

These operations each preserve the property of being free, smooth or regular.
In Gromov-Witten theory, we sometimes wish to pull back submanifolds
under an evaluation map defined on the moduli space.
To make such pullbacks well behaved, it is helpful
to ensure that the corresponding maps from the thickening
are topological submersions.

\begin{Lemma}  \label{Lem:transverse_global_chart}
	Let  $\K = (G,\scrT,E,s)$ be a smooth global Kuranishi chart for a compact metric space $M$
	and suppose we have a continuous (but not necessarily smooth) $G$-equivariant evaluation map $ev: \scrT \to X$ to a smooth manifold $X$.
	Then there is a continuous $G$-equivariant map $\widetilde{ev} : ev^*TX \to X$ and a $C^0$-small
	fiber-preserving $G$-equivariant homeomorphism $h : ev^*TX \to ev^*TX$ so that the following properties hold.
	\begin{enumerate}
		\item $\widetilde{ev}|_\scrT = ev$.
		\item $ev^*TX$ admits a smooth structure with respect to which the map $\widetilde{ev} \circ h$ is a smooth submersion near $\scrT$.
		\item The map $h$ is $C^0$-small and isotopic through $G$-equivariant homeomorphisms to the identity.
		\item If, in addition, $ev$ is already smooth near a compact subset $K \subset \scrT$,
		then we can assume that $h$ is the identity near $K$.
	\end{enumerate}
	In particular the stabilization of $\K$ by $ev^*TX$ is a smooth global Kuranishi chart
	together with a smooth submersion to $X$ extending $ev$.
\end{Lemma}

\begin{proof}
	Choose a complete Riemannian metric on $X$ and let $\exp : TX \to X$ be the exponential map.
	Define
	\begin{equation}
		\widetilde{ev} : ev^*TX \to X, \quad (e,v) \to \exp(v), \ \forall \ e \in \scrT, \ v \in T_{ev(e)}V.
	\end{equation}
	
	By a mollification argument, we can find a continuous family of $G$-equivariant maps
	$ev_t : \scrT \lra{} X$ so that $ev_0 = ev$ and $ev_1$ is smooth and so that all of these maps are $C^0$
	close to each other.
	We can also assume these maps are the identity near our compact subset $K$.
	Now choose a $C^0$ small $G$-equivariant
	continuous family of fiber preserving homeomorphisms $h_t$, $t \in [0,1]$ of $ev^*TX$ whose restriction to each fiber is a translation map and
	so that $h_1$ sends $x \in \scrT$ to the unique point  in $ev^*TX|_x \cap \widetilde{ev}^{-1}(ev_1(x))$.
	\end{proof}

\subsection{Microbundles}
The global Kuranishi charts for moduli spaces of holomorphic curves that we obtain from geometry are not obviously smooth.  We will use the theory of microbundles to show that, by stabilizing, one can equip the moduli space with some smooth global chart. 
A microbundle is a generalization of the concept of a vector bundle,
allowing us to speak of the `tangent bundle' of a  topological manifold.
A prototype statement is that if the tangent microbundle of a topological manifold admits a vector bundle lift, then the product of that manifold with a Euclidean space 
admits a smooth structure.

\begin{Remark} We will later use a number of results of Cheng \cite{Cheng} on $K(n)$-duality for compact Lie group actions on smooth manifolds with finite stabilisers. An alternative  to appealing to smoothing theory would be to observe that Cheng's results also apply \emph{mutatis mutandis} to topological manifolds with locally linear actions under the same hypothesis of the microbundle admitting a vector bundle lift.  Since the technical ingredients in both approaches are largely the same, and having smooth global charts for moduli spaces of genus zero curves may be of wider interest, we have chosen the former approach. \end{Remark}

\begin{Definition} A \emph{microbundle} of rank $n$ over a topological space $X$ is a diagram $\{X \stackrel{s}{\longrightarrow} E \stackrel{p}{\longrightarrow} X\}$ where $E$ is a topological space and $p$ and $s$ are continuous maps satisfying 
	\begin{itemize}
		\item $p\circ s =\id_X$;
		\item for each $x\in X$ there are neighbourhoods $U_x \subset X$ of $x$ and $V_x \subset E$ of $s(x)$ and a homeomorphism $h_x: U_x \times \bR^n \to V_x$ for which $p\circ h_x = \mathrm{pr}_{U_x}$ and $h_x|_{U_x\times\{0\}} = s$. (Here $pr_{U_x}$ denotes the natural projection map to $U_x$.)
	\end{itemize}
A \emph{morphism} $\phi : E_1 \lra{} E_2$ of microbundles
$\{X \stackrel{s_i}{\longrightarrow} E_i \stackrel{p_i}{\longrightarrow} X\}$, $i=1,2$
is the equivalence class of a continuous map 
$\phi : U_1 \lra{} E_2$ defined on a neighborhood $U_1$ of $s_1(X)$ and which commutes with the
maps $s_i$ and $p_i$, $i=1,2$ respectively.
Two such maps  are in the same equivalence class if they agree on a neighborhood of $s_1(X)$.
\end{Definition}

Product neighbourhoods and morphisms are illustrated in the following commutative diagrams:

\begin{center}
\begin{equation}\begin{tikzcd}
	&&&&& {E_1} \\
	X & E & X && X & {U_1} & X \\
	{U_x} & {V_x} & {U_x} && X & {E_2} & X \\
	{} & {U \times \bR^n} & {}
	\arrow["s", from=2-1, to=2-2]
	\arrow[from=3-1, to=3-2]
	\arrow["\cong"', from=3-2, to=4-2]
	\arrow["{\mathrm{pr}_{U_x}}"', from=4-2, to=3-3]
	\arrow[hook, from=3-1, to=2-1]
	\arrow[hook, from=3-2, to=2-2]
	\arrow[hook, from=3-3, to=2-3]
	\arrow["p", from=2-2, to=2-3]
	\arrow["{\id_X}", curve={height=-18pt}, from=2-1, to=2-3]
	\arrow[hook, from=2-6, to=1-6]
	\arrow["{s_1}", from=2-5, to=2-6]
	\arrow["{s_2}", from=3-5, to=3-6]
	\arrow["{p_1}"', from=1-6, to=2-7]
	\arrow[Rightarrow, no head, from=2-5, to=3-5]
	\arrow[Rightarrow, no head, from=2-7, to=3-7]
	\arrow["\phi", from=2-6, to=3-6]
	\arrow["{p_1}", from=3-6, to=3-7]
\end{tikzcd}\end{equation}
\end{center}

When the context is clear, we will write $E$ instead of the microbundle $X \lra{s} E \lra{p} X$.
If $V \to X$ is a rank $n$ vector bundle then $X \lra{0} V \lra{} X$ is its \emph{associated microbundle}.
We will sometimes write $V_\mu$ for the associated microbundle or, if the context is clear, we will just write $V$ for this microbundle.
A \emph{vector bundle lift} of a microbundle $E$
is a vector bundle $V$ together with an isomorphism from $V_\mu$ to $E$.
If $X$ is paracompact, then
a  theorem of Kister \cite{Kister} says that equivalence classes of microbundles over $X$ are in bijection with isomorphism classes of fiber bundles whose fiber is homeomorphic to $\bR^n$ and whose structure group consists of homeomorphisms fixing $0$.

\begin{Definition} 
	The \emph{tangent microbundle} $T_\mu X$ of a topological manifold\footnote{The tangent microbundle is also written as $\tau X$ in the literature.} $X$
	is the microbundle
	$\{X \stackrel{\Delta}{\longrightarrow} X\times X \stackrel{p_1}{\longrightarrow} X\}$
	where $\Delta$ is the diagonal inclusion map and $p_1$ is the projection to the first factor of $X \times X$. 
\end{Definition}

If $X$ is smooth, then the `usual' tangent bundle $TX$  is a vector bundle lift of $T_\mu X$.
This lift is constructed, after choosing a complete metric on $X$, using the exponential map from the normal bundle of the diagonal in $X \times X$.

\begin{Definition} \label{defn direct sum}
	Let $X \lra{s_i} E_i \lra{p_i} X$, $i=1,2$ be two microbundles and let $\Delta$ be the diagonal inclusion map as above.
	Let $E_{12} := (p_1 \times p_2)^{-1}(\Delta(X)) \subset E_1 \times E_2$ where $\Delta$ is the diagonal inclusion map as above.
	The \emph{direct sum} $E_1 \oplus E_2$ is the microbundle
	\begin{equation}
	\begin{tikzcd}
		X \arrow[rr, "(s_1 \times s_2) \circ \Delta"] && E_{12} \arrow[rr, "p_1 \times p_2"] && \Delta(X) \, =\, X.
		\end{tikzcd}
	\end{equation}
\end{Definition}

If $X \lra{s} E \lra{p}  X$ is a microbundle and $\phi_i : U \lra{} E_i$, $i=1,2$ are
microbundle morphisms with $U \subset E$ a neighborhood of $s(X) \subset E$,
then we get another microbundle morphism
$\phi_1 \oplus \phi_2 : E \lra{} E_1 \oplus E_2$ sending $x \in U$ to $(\phi_1(x),\phi_2(x))$.
The theorem of Kister \cite{Kister} mentioned above gives an alternative description of the direct sum
of a rank $n$ and $m$ bundle respectively as a fiber product of the corresponding $\bR^n$ and $\bR^m$ bundles.

\begin{Definition}
	If $X \lra{s} E \lra{p} X$ is a microbundle and $f : \check{X} \lra{} X$ a continuous map
	then we define the \emph{pullback microbundle}  to be
	$\check{X} \lra{f^*s} f^*E \lra{f^*p} X$
	where
	\begin{equation}
		f^*E = \{(e,x) \in E \times \check{X} \ | \ p(e) = x \} \subset E \times X
	\end{equation}
	and $f^*s$ sends $x$ to $(s(f(x)),x)$ and $f^*p$ sends $(e,x)$ to $x$.
	If $\check{X}$ is a subset of $X$ and $f$ the corresponding inclusion map then we define the \emph{restriction}
	$E|_{\check{X}}$ to be the corresponding pullback $f^*E$.
\end{Definition}

There are $G$-equivariant versions of these notions for a topological group $G$.
\begin{defn}
	Let $X$ be a $G$-space.
	A \emph{$G$-microbundle} on $X$ is a microbundle $X \lra{s} E \lra{p} X$ where $E$ is a $G$-space such that 
	$g \cdot s(x) = s(g \cdot x)$ and $p(g \cdot x) = g \cdot p(x)$.
	\end{defn}
	
Morphisms and direct sums are defined in the obvious way. 
		A \emph{morphism} between two $G$-microbundles $\{X \stackrel{s_i}{\longrightarrow} E_i \stackrel{p_i}{\longrightarrow} X\}$, for $i=1,2$
	is a $G$-equivariant morphism between the corresponding microbundles $E_1$, $E_2$, 
	where $G$ acts on the space of such morphisms by pre and post composition with an element of $G$ and its inverse respectively.
	The \emph{direct sum} of two $G$-microbundles $E_1$ and $E_2$ is the direct sum as in Definition
	\ref{defn direct sum},  where $G$ acts diagonally on $E_1 \times E_2$.

Given a $G$-vector bundle $V \to X$ over a $G$-space $X$, there is an \emph{associated $G$-microbundle}
$X \lra{0} V \twoheadrightarrow X$.
A \emph{$G$-vector bundle lift} of a $G$-microbundle $X \lra{s} E \lra{p} X$
is a $G$-vector bundle together with an equivalence between the associated $G$-microbundle of $V$ to $E$.

\begin{Example} \label{example natural smooth lift}
	If $M$ is a smooth manifold with a smooth $G$-action for some compact Lie group $G$,
	then there is a natural $G$-equivariant lift $TM \to T_\mu M$ defined as follows:
	Choose a complete $G$-equivariant metric on $M$ and let $\exp : TM \to M$ be the exponential map.
	Then the lift sends $v \in T_p M$ to $(p,\exp(v)) \in T_\mu M \subset M \times M$.
	The space of such lifts is contractible due to the fact that the space of complete
	$G$-equivariant metrics is contractible.
\end{Example}

\subsection{Smoothing theory}

We will use the following result from \cite[Theorem 6.8]{Lashof}:

\begin{Proposition}[Lashof] \label{prop lashof}
	Let $M$ be a topological manifold equipped with a continuous action of a compact Lie group $G$.  Assume that
	\begin{enumerate}
		\item There are finitely many orbit types;
		\item The microbundle $T_{\mu}M$ admits a $G$-vector bundle lift $l : E \to T_\mu M$.
	\end{enumerate}
	Then there is a $G$-representation $V$ for which the product $V \times M$ admits a $G$-equivariant smooth structure.
\end{Proposition}

\begin{Addendum}\label{add lashof}
	In  Proposition \ref{prop lashof}, we can furthermore ensure that
	the  lift $T(M \times V) \to T_\mu (M \times V)$ constructed using the exponential map as in Example \ref{example natural smooth lift}, when  restricted to $M \times 0$, is isotopic through $G$-equivariant lifts to the induced lift $l \oplus \id_V : E \oplus V \to T_\mu M \oplus V$ (see \cite[Lemma 3.5]{Lashof}).
	\end{Addendum}
	
Lashof's result is somewhat stronger: it gives a bijection between stable isotopy classes of $G$-smoothings and homotopy classes of $G$-vector bundle lift of $T_{\mu}M$. 

\begin{Remark}
Later results \cite{KirbySiebenmann} show that, provided $\dim_{\bR}(M)\neq 4$, under the same hypotheses as in Proposition \ref{prop lashof} one can equip $M$ itself with a smooth structure, i.e. taking $V=\{0\}$. However, stabilisations of Kuranishi charts by $G$-representations appear naturally in our work anyway. 
\end{Remark}

\begin{Corollary} \label{Cor:smooth_fattening}
	Let $\scrK = (G,\scrT,E,s)$ be a global Kuranishi chart so that the $G$-action on $\scrT$ admits only finitely many orbit types
	and so that $T_\mu \scrT$ admits a $G$-vector bundle lift $l : V \to T_\mu \scrT$.
	Then there is a $G$-representation $V$, and a stabilization $\K'$ of $(G,\scrT,E,s)$ with thickening $\scrT\times V$, which is equivalent to a smooth Kuranishi chart.
	\end{Corollary}

\begin{proof}
By the proposition above, there is a $G$-representation $V$ so that the $G$-manifold
$V \times \scrT$ admits a smooth structure.
Also every topological $G$-vector bundle over $V \times \scrT$ is isomorphic to a smooth $G$-vector bundle.
Hence stabilizing $(G,\scrT,E,s)$ to $V \times \scrT$ gives us our result.
\end{proof}

\begin{Remark} \label{rmk:smooth_fattening}
Looking ahead to our discussion of orientation in Section \ref{Section Orientations of KC} below, we observe at this stage that, in the situation of Corollary \ref{Cor:smooth_fattening}, there is a natural 1-1 correspondence between $\bK$-orientations of $\scrK$ and $\scrK'$. Indeed,  the natural lift $T(\scrT \times V) \to T_\mu (\scrT \times V)$ obtained from the exponential map as in Example \ref{example natural smooth lift},  restricted to $\scrT \times 0$, is isotopic through $G$-equivariant lifts to the induced lift $l \oplus \id_V : E \oplus V \to T_\mu \scrT \oplus V$, by Addendum \ref{add lashof}. This 
gives a natural bijective correspondence between $\bK$-orientations of $(G,\scrT,E,s)$ and of the stabilized global Kuranishi chart $\scrK'$.
\end{Remark}

\subsection{Topological submersions}

To apply Proposition \ref{prop lashof}, we need a nice geometric condition ensuring that a given
topological manifold with $G$-action admits a vector bundle lift of its tangent $G$-microbundle.
The following section explores when such conditions hold for topological submersions.
Throughout this section, let us fix a compact Lie group $G$,
topological manifolds $M$ and $B$ of dimension $n$ and $k$
respectively together with $G$-actions on both $M$ and $B$
and a $G$-equivariant continuous map  $\pi : M \to B$.
For each subset $W \subset M$, define $W|_b := \pi^{-1}(b) \cap W$
for each $b \in B$.
For a $G$-space $X$ and  $x \in X$,
we let $G_x \subset G$ be the corresponding stabilizer subgroup.

\begin{Definition} \label{defn product neighborhood}
	Let $p \in M$ and let $b = \pi(p) \in B$.
	A \emph{product neighborhood (of $p$)}
	is a homeomorphism
	$\iota : W \to W|_b \times \pi(W)$
	where $W \subset M$ is a neighborhood of $p$
	satisfying:
	\begin{itemize}
		\item $\pi \circ \iota^{-1}$ is the projection map to $\pi(W)$
		\item and $\iota|_{(W|_b)} : W|_b \to W|_b \times \{\pi(p)\}$ is the identity map.
	\end{itemize}
	We say that $\pi$ is a \emph{topological submersion} if every point $p \in M$
	admits a product neighborhood. In that case, the
\emph{vertical tangent microbundle} $T^{vt}_\mu(\pi)$ of $\pi$
is the microbundle:
\begin{equation}
	M \lra{\Delta} M \times_B M \lra{\pi \times \pi} M
\end{equation}
where $M \times_B M$ is the fiber product of $\pi$ with itself,
$\Delta$ is the diagonal map and $\pi \times \pi$ sends $(p,q)$ to $\pi(p) = \pi(q)$.
	\end{Definition}

	Each $g \in G$ sends a product neighborhood $\iota$ of $p \in  M$ to a new product neighborhood $g_*\iota$ of $g \cdot p \in M$
satisfying
\begin{equation}
	g_*\iota : g(W) \to g(W|_b) \times g(\pi(W)), \quad g_*(\iota)(w) = g \iota(g^{-1} \cdot w).
\end{equation}
where $g$ acts diagonally on $W|_b \times \pi(W)$.

\begin{Definition}
	A \emph{$G_p$-invariant product neighborhood} is a product neighborhood of $p \in M$
	which is $G_p$-invariant under the action above.
	
	We say that $\pi$ is a \emph{$G$-equivariant topological submersion}
	if every point $p \in M$ admits a $G_p$-invariant product neighborhood.
\end{Definition}

\begin{Definition}
	A $G$-action on a topological manifold is \emph{locally linear}
	if for each point $p$ in this manifold, there is a chart centered at $p$ so that
	$G_p$ acts linearly on this chart.
	We call this a \emph{linear chart at $p$}.
	We say that $\pi$ is \emph{fiberwise locally linear}
	if it is a topological submersion and
	if for each $b \in B$, the group $G_b$
	is locally linear on the manifold $\pi^{-1}(b)$.
\end{Definition}

\begin{lemma} \label{lemma product chart}
	Suppose that $B$ is locally linear and that $\pi$ is fiberwise locally linear and a $G$-equivariant toplogical submersion.
	Then for each $p \in M$, there exists a chart $(x_1,\cdots,x_n) : U \lra{} \bR^n$
	centered at $p$ satisfying the following properties:
	\begin{itemize}
		\item There is a continuous map $\nu : \bR^k \lra{} B$ which is a homeomorphism onto its image
		so that the composition $\nu \circ (x_{n-k+1},\cdots,x_n)$ is equal to $\pi$. In other words, $\pi$ is essentially the projection to the last $k$ coordinates.
		\item For each $g \in G_p$, $g \cdot U = U$.
		\item For each $g \in G_p$ there are $(n-k) \times (n-k)$ and $k \times k$ matrices $A_g$ and $B_g$
		respectively so that $g \cdot (x_1,\cdots,x_n) = (A_g(x_1,\cdots,x_{n-k}),B_g(x_{n-k+1},\cdots,x_n))$.
		In other words, $g$ acts linearly and by a block diagonal matrix.
	\end{itemize}
\end{lemma}
\begin{proof}
	Take a $G_p$-equivariant product neighborhood $\iota : W \lra{} W|_b \times \pi(W)$, $b = \pi(p)$,
	centered at $p$. Choose local coordinate charts
	$y_1,\cdots, y_{n-k}$ on $W|_b$ centered at $p$
	and $y_{n-k+1},\cdots,y_n$ on $B$ centered at $b$
	so that $G_p$ acts linearly on both of these coordinate systems.
	We let $B_1 \times B_2 \subset W|_b \times \pi(W)$ be a product of open balls centered at the origin
	with respect to these coordinates
	and we let $U = \iota^{-1}(B_1 \times B_2)$. These balls have to be small enough so that their closure is compact in $W|_b$ and $\pi(W)$ respectively.
	We let $x_1,\cdots,x_n$ be the pullback of the coordinates $y_1,\cdots,y_{n-k}$ and $y_{n-k+1},\cdots,y_n$
	via the composition of $\iota$ with the respective projection maps to $W|_b$ and $\pi(W)$.
\end{proof}

We wish to construct a natural isomorphism of microbundles $T_\mu M \cong T_\mu^{vt}(M) \oplus \pi^* T_\mu B$.
This isomorphism will be a direct sum of a map $P : T_\mu M \lra{} T_\mu^{vt}(M)$ 
and the natural map $\tau : T_\mu M \lra{} \pi^*T_\mu B$ induced by $\pi$ (see the paragraph after Definition \ref{defn direct sum}).
For such a direct sum to induce an isomorphism of bundles, we want the map
$P$ to look like a fiberwise toplogical submersion.
By definition, the bundle map $P$ sends points $(q,q')$ in a neighborhood of the diagonal in $M \times M$
to a point $(q,\phi(q,q'))$ for some map $\phi$. The definition below gives conditions
on such a map $\phi$ to ensure that the direct sum of $P$ and $\tau$ induces a $G$-equivariant isomorphism
of microbundles over an open subset $W \subset M$.

\begin{defn} \label{defn microbundle submersion}
	Let $W \subset M$ be a $G$-invariant open subset of $M$.
	A \emph{$G$-equivariant fiber submersion along $W$}
	is a continuous map $\phi : \widetilde{W} \lra{} M$
	where $\widetilde{W} \subset W \times W$ is an open neighborhood of the diagonal
	satisfying the following properties:
	\begin{itemize}
		\item  $\phi_{q,b'} := \phi|_{\widetilde{W} \cap (q \times \pi^{-1}(b'))}$ is a homeomorphism onto
		an open subset of $\pi^{-1}(\pi(q))$ for each $b' \in B$, $q \in W$,
		\item $\phi_{q,\pi(q)}$ sends each point $(q,q')$ in $\widetilde{W} \cap (q \times \pi^{-1}(\pi(p)))$ to $q' \in M$ (in other words, it is the identity map),
		\item $\widetilde{W}$ is $G$-equivariant where $G$ acts diagonally and $\phi$ is a $G$-equivariant map.
	\end{itemize}

\end{defn}

In order to construct $G$-equivariant fiber submersions, we need a slice theorem.

\begin{defn} \label{defn slice}
	For each subset $Q \subset M$, and each subgroup $G' \subset G$, define $G' \cdot Q \subset M$
	to be the set of points equal to $g \cdot q$ for each $g \in G'$ and $q \in Q$.
	
	A \emph{slice} through a point $p \in M$ is a topological subspace $S_p \subset M$ satisfying the following properties:
	\begin{itemize}
		\item $p \in S_p$,
		\item $G_p \cdot S_p = S_p$,
		\item the natural map $G \times_{G_p} S_p \lra{} G \cdot S_p$ sending $(g,s)$ to $g \cdot s$ is a homeomorphism and
		\item $G \cdot S$ is open in $M$.
	\end{itemize}
\end{defn}

We have the following theorem:
\begin{theorem} \label{theorem slice} \cite[Theorem 2.1]{mostowembeddings}
	There exists a slice through every point in $M$.
\end{theorem}

The following lemma is an extension result for $G$-equivariant fiber submersions.

\begin{lemma} \label{lemma microbundle extension}
	Let $(x_1,\cdots,x_n) : U \lra{} \bR^n$ be a coordinate chart centered at a point $p \in M$ as described in Lemma
	\ref{lemma product chart} whose image in $\bR^n$ is convex.
	Suppose that there is a slice $S_p$ through $p$ so that $G \cdot S_p$ contains $U$.
	Let $K, W \subset G \cdot U$ be $G$-invariant subsets satisfying $K \subset W$, with $W$ open and $K$ closed in $G \cdot U$.
	Let $\phi : \widetilde{W} \lra{} M$ be a $G$-equivariant fiber submersion along $W$.
	Then there exists a $G$-equivariant fiber submersion $\psi : \widetilde{U} \lra{} M$
	along $G \cdot U$ with the property that $\psi$ is equal to $\phi$ is a small neighborhood
	of the diagonal of $K \times K$ in $W \times W$.
\end{lemma}

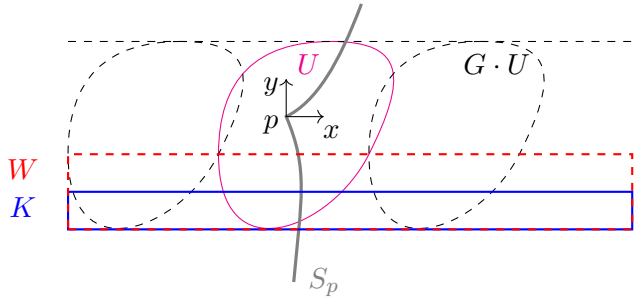
\begin{center}
\begin{tikzpicture}
	
	\draw[magenta] (-0.5,1.5) .. controls (-1.5,1.5) and (-2,1) .. (-2,0) .. controls (-2,-1.6) and (-0.5,-1) .. (0,0) .. controls (0.5,1) and (0.5,1.5) .. (-0.5,1.5);
	\draw[dashed] (-0.5+2,1.5) .. controls (-1.5+2,1.5) and (-2+2,1) .. (-2+2,0) .. controls (-2+2,-1.6) and (-0.5+2,-1) .. (0+2,0) .. controls (0.5+2,1) and (0.5+2,1.5) .. (-0.5+2,1.5);
	\draw[dashed] (-0.5-2,1.5) .. controls (-1.5-2,1.5) and (-2-2,1) .. (-2-2,0) .. controls (-2-2,-1.6) and (-0.5-2,-1) .. (0-2,0) .. controls (0.5-2,1) and (0.5-2,1.5) .. (-0.5-2,1.5);
	\draw [dashed](-4,1.5) -- (3.5,1.5);
	\draw [dashed](-4,-1) -- (3.5,-1) -- cycle;
	\draw [very thick,gray](-0.1,2) .. controls (-0.4,1.2) and (-0.7,0.7) .. (-1.1,0.5) .. controls (-0.8,-0.2) and (-0.9,-0.6) .. (-1,-1.7);
	\node at (-1.3,0.4) {$p$};
	\node at (-0.6,-1.7) {\color{gray} $S_p$};
	\node at (-0.8,1.2) {\color{magenta} $U$};
	\draw [blue,thick]  (-4,-0.5) rectangle (3.5,-1);\draw [thick, dashed,red] (-4,0) rectangle (3.5,-1);
	\node at (-4.6,-0.7) {\color{blue} $K$};
	\node at (-4.6,-0.2) {\color{red} $W$};
	\node at (1.7,1.2) {$G \cdot U$};
	\draw [<->](-1.1,1) -- (-1.1,0.5) -- (-0.6,0.5);
	\node at (-0.5,0.3) {$x$};
	\node at (-1.3,0.9) {$y$};
\end{tikzpicture}
\end{center}

\begin{proof} [Proof of Lemma \ref{lemma microbundle extension}]
	The key idea of the proof is to use the linear structure of the chart to linearly interpolate between
	the map $\phi$ and the natural projection map to the first $n-k$ coordinates of our coordinate system.
	A slice is needed to ensure that everything is done $G$-equivariantly.
	
	Define $S'_p := S_p \cap U$.
	Let $\rho : S'_p \lra{} [0,1]$ be a continuous function which is equal to $1$ along $K$
	and $0$ outside $W$ and which is $G_p$-equivariant (such a function can be constructed by an averaging argument).
	We will write our coordinates on $U$ in shorthand as $(x,y)$
	where $x = (x_1,\cdots,x_{n-k})$ and $y = (x_{n-k+1},\cdots,x_n)$.
	For each $q \in S'_p$ with coordinates $(x_q,y_q)$ in $U$ and each $t \in [0,1]$,
	define:
	\begin{equation}
		\eta_q : U \lra{} U, \quad \eta_q(x,y) = (x,y_q + \rho(q)(y-y_q)).
	\end{equation}
	Define $\check{V}_p \subset S'_p \times U$ to be the union of $(S'_p-W) \times U$ and $\widetilde{W} \cap (S'_p \times U)$.
	Let
	\begin{equation}
		V_p := \{(q,q') \in S'_p \times U \ : \ (q,\eta_q(q')) \in \check{V}_p \}
	\end{equation}
	and $\widetilde{U} := G \cdot V_p \subset (G \cdot U) \times (G \cdot U)$ where $G$ acts diagonally.
	The set $\widetilde{U}$ is an open neighborhood of the diagonal in $(G \cdot U) \times (G \cdot U)$
	since $S'_p \times U$ is a slice for this diagonal action.
	We define
	\begin{equation}
		\psi_p : V_p \lra{} U, \quad \psi_p(q,q') := \left\{\begin{array}{ll}
		\phi(q,\eta_q(q'))	& \textnormal{if} \ q \in W \\
		\eta_q(q')	& \textnormal{otherwise.}
		\end{array}\right.
	\end{equation}
We define $\psi : \widetilde{U} \lra{} M$ to be the unique $G$-equivariant map whose restriction to $V_p$
is $\psi_p$.
\end{proof}

\begin{prop}
	Suppose that $B$ is locally linear and that $\pi$ is fiberwise locally linear and a $G$-equivariant toplogical submersion.
	Then there exists a $G$-equivariant fiber submersion over $M$.
\end{prop}
\begin{proof}
	This follows from the fact that there exists a slice through every point by Theorem \ref{theorem slice}
	combined with an induction argument using the extension Lemma \ref{lemma microbundle extension}
	on a countable covering family of charts as described in this lemma.
\end{proof}

\begin{corollary} \label{corollary product isomorphism}
	Suppose that $B$ is locally linear and that $\pi$ is fiberwise locally linear and a $G$-equivariant toplogical submersion.
	Then there is a natural morphism of microbundles $P : T_\mu M \lra{} T_\mu^{vt}(M)$
	whose restriction to $T_\mu^{vt}(M)$ is the identity map.
	This map also induces a natural $G$-equivariant isomorphism
	\begin{equation} \label{eqn microbundle iso}
		P \oplus \tau : T_\mu(M) \to T^{vt}_\mu(M) \oplus \pi^*(T_\mu B)
	\end{equation}
	where $\tau : T_\mu(M) \to \pi^*(T_\mu(B))$ is the map
	\begin{equation}
		M \times M \supset T_{\mu}M \ni (p,w) \, \mapsto \, (p,(\pi(p),\pi(w))) \in \pi^*(T_\mu(B)) \subset M \times B \times B
	\end{equation}
	induced by the projection map $\pi$.
\end{corollary}
\begin{proof}
	By the previous proposition,
	there is a $G$-equivariant fiber submersion $\psi : \widetilde{U} \lra{} M$ over $M$,
	where $\widetilde{U} \subset M \times M$ is a neighborhood of the diagonal.
	We define
	\begin{equation}
		P : \widetilde{U} \lra{} M \times_B M, \quad P(q,q') := (q,\psi(q,q')).
	\end{equation}
	The map \eqref{eqn microbundle iso} is an isomorphism due to the fact that the map
	sending each point $q'$ in a neighborhood $\widetilde{U} \cap \{q \times M\} \subset q \times M=M$ of $q$ to $(\psi(q,q'),\pi(q')) \in \pi^{-1}(\pi(q)) \times B$
	is a homeomorphism onto its image for each $q \in M$.
\end{proof}

\subsection{Fiberwise smooth structures}

Throughout this section, we will fix a Lie group
$G$ and a
$G$-equivariant continuous map
$\pi : M \to B$ between $G$-manifolds $M$ and $B$ of dimensions $n$ and $k$ respectively. We will further assume that the fibres of $\pi$ are equipped with smooth structures.  For the next definition, we recall that the $C^1_{loc}$-topology on the space of $C^1$ functions on a smooth manifold is the topology associated to the collection of $C^1$ semi-norms obtained by restriction to compact subsets.

\begin{defn} \label{defn fiberwise smooth}
	Two product neighborhoods
	$\iota_i : W_i \to W_i|_{b_i} \times \pi(W_i)$, $i=1,2$ of $\pi$
	are \emph{$C^1_{loc}$-compatible}
	if for each $p \in W_1 \cap W_2$, there exists
	a product neighborhood $\iota : W \to W|_b \times \pi(W)$, $b = \pi(p)$
	so that the family of maps:
	\begin{equation} 
		\eta_v : W|_b \lra{} W_i|_{b_i}, \quad w \to \Pi_i(\iota((\iota|_{W|_v})^{-1}(w))), \quad v \in \pi(W)
	\end{equation}
	are all smooth and
	vary continuously with respect to the $C^1_{loc}$-topology,
	where $\Pi_i$ is the projection map to $W_i$ for $i=1,2$.
	\end{defn}
	
	A \emph{fiberwise smooth $C^1_{loc}$ $G$-bundle}
	is a $G$-equivariant map
	$\pi : M \to B$
	together with a collection $(\iota_i)_{i \in I}$ of
	$C^1_{loc}$-compatible $G_p$-invariant product neighborhoods around a collection of points $(p_i)_{i \in I}$
	whose domains cover $M$.
	Any $G_p$-invariant product neighborhood of $p \in M$ which is $C^1_{loc}$-compatible with each $(\iota_i)_{i \in I}$ will be called a \emph{chart} of $\pi$.
		
	\begin{defn}
The \emph{vertical tangent bundle} $T^{vt}M$ of a fibrewise smooth $C^1_{loc}$ $G$-bundle is the $G$-vector bundle
whose restriction to the domain $W$ of a chart $\iota : W \to W|_b \times \pi(W)$
of $\pi$ is the pullback of $T(W|_b)$ to $W$ via the projection map to $\pi(W)$.
\end{defn}

Note that we need the $C^1_{loc}$-compatibility condition for such a vertical tangent bundle to
make sense. If we replaced `$C^1_{loc}$' with `$C^0_{loc}$' then we would be unable to define
a vertical tangent bundle.

\begin{lemma} \label{lemma lift of vertical microbundle}
	Let $\pi : M \to B$ be a fiberwise smooth $C^1_{loc}$ $G$-bundle whose base $B$
	is a smooth $G$-manifold.
	Then there is a natural $G$-equivariant lift of the
	vertical tangent microbundle $T_\mu^{vt} M$ to $T^{vt} M$.
\end{lemma}
\begin{proof}
	Using a fiberwise smooth partition of unity, we can put a metric on $T^{vt}(M)$
	so that the restriction to each fiber $M|_b = \pi^{-1}(b)$ is a smooth Riemannian metric $g_b$ for each $b \in B$.
	Let $\exp : T^{vt} M \to M$ be the map whose restriction to each fiber $T^{vt}M|_{(M|_b)}$
	is the exponential map $T^{vt}M|_{(M|_b)} \to M|_b$ with respect to $g_b$ for each $b \in B$.
	Our lift is then the map
	\begin{equation}
		T^{vt} M \to T^{vt}_\mu M \subset M \times M, \quad v \to (p,\exp(v)), \quad \forall \ p \in M, \ v \in T^{vt}_p M.
	\end{equation}
	It is straightforward to check this has the required properties.
\end{proof}

\begin{prop} \label{prop Kuranishi smoothing}
	Let $\K := (G,\scrT,E,s)$ be a global Kuranishi chart
	for which $\scrT$ is the total space of a fiberwise smooth $C^1_{loc}$ $G$-bundle
	whose base $B$ is a smooth $G$-manifold.
	Also, suppose that the $G$-action on $\scrT$ has only finitely many orbit types.
	Then there is a stabilization $\K'$ of $\K$ which admits a smooth structure.
	The thickening  of $\K'$ is $\scrT \times V$ for some $G$-representation $V$
	and the restriction to $\scrT \times 0$ of the natural lift $T(\scrT \times V) \to T_\mu (\scrT \times V)$, constructed using the exponential map as in Example \ref{example natural smooth lift},  is isotopic through $G$-equivariant lifts to the induced lift $l \oplus \id_V : E \oplus V \to T_\mu \scrT \oplus V$.
\end{prop}
\begin{proof}
	This follows from Lemma \ref{lemma lift of vertical microbundle} combined with
	Corollary \ref{Cor:smooth_fattening} and Remark \ref{rmk:smooth_fattening}.
\end{proof}

To apply Proposition \ref{prop Kuranishi smoothing}
to the moduli space of genus zero curves, we will show that
it admits a natural global Kuranishi chart with a fiberwise smooth $C^1_{loc}$ $G$-bundle structure.

\subsection{Pullbacks} \label{section pullbacks of kuranishi charts}

Given a Kuranishi chart describing a moduli space $M$ of holomorphic curves, one would like to  construct a Kuranishi chart for a moduli space $M'$ of curves in $M$ with additional structure
(such as adding marked points) in such a way that the forgetful map $M' \to M$
behaves well with respect to the Kuranishi charts.
In the sequel, this will be done by first mapping to another simpler moduli space,
adding the additional structure there, and pulling back the additional structure.
Here we will give the basic pullback operation for Kuranishi structures we require; it is applied to moduli spaces of curves
in Section \ref{Sec:marked_points}.

\begin{defn} \label{DEFN pullback of kuranishi charts}
	Let $\K=(G,\scrT,E,s)$ be a global Kuranishi chart
	and let $\Pi : \scrT \to B$ be a $G$-equivariant topological submersion as in Definition
	\ref{defn product neighborhood}.
	Let $f : B' \to B$ be a $G$-equivariant continuous map.
	We define the \emph{pullback} of $(G,\scrT,E,s)$ by $f$
	to be the global Kuranishi chart
	$f^*\K = (G,f^*\scrT,f^*E,f^*s)$ where
	\begin{equation}
		f^*\scrT = \left\{(x,y) \in \scrT \times B' \ : \ \Pi(x)=f(y) \right\} \subset \scrT \times B'
	\end{equation}
	is the pullback of $\scrT$ by $f$,
	$f^*E$ and $f^* s$ are the pullbacks of $E$ and $s$ via the natural projection map to the first factor
	$f^*\scrT \to \scrT$.
\end{defn}

Note that pullback commutes with germ equivalence, stabilization
and group enlargement.

\section{Morava virtual class\label{Sec:Morava_virtual_class}}

The goal of this section is to describe the virtual fundamental class in Morava $K$-theory for a global Kuranishi chart satisfying a suitable orientation hypothesis. We will also see that the virtual class has various desirable properties:
for instance, it will coincide with the usual fundamental class of a smooth manifold if the
corresponding global Kuranishi chart is smooth, free and regular.

\subsection{A rational homology warm-up}

Before we go into the details of constructing the virtual fundamental class in general,
we will briefly describe how to do this over $\bQ$.
Recall that for any spaces $X \subset Y$, we set $H^*(X|Y;\bQ) := H^*(X,X\backslash Y;\bQ)$; 
similarly for \v{C}ech cohomology.
By a \emph{homology $\bQ$-manifold} of dimension $n$,
we will mean a locally compact Hausdorff space $Q$ for which, for every $x \in Q$,  
$H_*(Q|x;\bQ)$ is equal to $\bQ$ and concentrated in degree $n$.
An \emph{orientation} for $Q$ is an element in Borel-Moore homology
$H_n^{lf}(Q;\bQ)$
whose restriction to $H_*(Q|x;\bQ)$ is non-zero for each $x \in Q$.

Now suppose that a compact Lie group $G$
acts on a topological manifold $Y$ with finite stabilizers.
Then by \cite[Theorem 2.1]{mostowembeddings},
we have that for each $x \in Y$,
there is a $G$-equivariant neighborhood of the $G$-orbit through $x$
isomorphic to $G \times_{G_x} S_x$ for some
$G_x$ space $S_x$.
Since $G \times_{G_x} S_x$ is a manifold for each $x \in Y$,
we get that $Y/G$ is a homology $\bQ$-manifold.

Now let $(G,\scrT,E,s)$ be a global Kuranishi chart for a compact metric space $M$
of virtual dimension $d$ and
suppose that the homology $\bQ$-manifold $\scrT/G$ is oriented and the bundle $E$
admits a $G$-equivariant orientation.
Let $m$ be the rank of $E$, $n$ the dimension of $G$.
Define $Z = s^{-1}(0)$.
Then the \emph{$\bQ$-virtual fundamental class} of $M$ is the composition:
\begin{equation}
\begin{split}
	\vfc_M : \check{H}^d(M;\bQ) \lra{D} \check{H}_m(\scrT/G|M;\bQ) \stackrel{\cong}{\longleftarrow} H_m^G(\scrT|Z;\bQ) \lra{s_*}   \\
	H^G_m(E|\scrT;\bQ) \lra{\small{Thom}} H_0^G(\scrT;\bQ) \lra{}
	H_0^G(pt;\bQ) \cong \bQ
\end{split}
\end{equation}
where $D$ is Alexander duality for $\bQ$-homology manifolds (see \cite[Proposition 1.3]{Skljarenko}).
This definition of the virtual fundamental class does not
depend on the choice of global Kuranishi chart up to germ equivalence, stabilization
or group enlargement.
Note that it is also sufficient to require that the virtual microbundle $T_\mu \scrT - E - \frg$
be oriented with a $G$-equivariant orientation, since one can stabilize
the Kuranishi chart until $\scrT$ and $E$ are both $G$-equivariantly oriented (See Lemma \ref{lemma stronger orientation} below).

This formulation of the virtual fundamental class is modelled on that of Pardon \cite{pardon2016algebraic}, who describes it locally in a version of \v{C}ech cochains; since we are working globally, we do not need any chain-level considerations. We wish to use the same general template to construct a virtual fundamental class
$H^*(M;\bK) \to \bK_*$ for certain global Kuranishi charts, but where $\bK$ is a ring spectrum.

\subsection{Generalised cohomology theories}

For the moment we will mostly work with generalized (co)homology theories
rather than the underlying spectra; constructions at the level of spectra (required for Theorem \ref{thm:main_generalised}) are deferred to Sections \ref{Sec:review} and \ref{Sec:chromatic}.
 We next provide a very brief summary of the notation and properties
that we will need from these (co)homology theories.

Recall that a ring spectrum $\bK$ determines a multiplicative generalised (co)homology theory; we will write $H_*(X;\bK)$ respectively $H^*(X;\bK)$ for the (co)homology groups\footnote{When $X$ has the homotopy type of a CW complex these groups are unambiguous, but for general $X$ they depend on the point-set model for the spectrum $\bK$ (i.e. they need  not be  homotopy invariant in $\bK$). The choice of that point-set model matters, for instance, in Lemma \ref{lemma Cech property}.} $\bK_*(X)$ respectively $\bK^*(X)$.   We denote the coefficients $H_*(pt;\bK)$ by $\bK_*$. We will largely need the case in which $\bK$ is one of the (periodic) Morava $K$-theories, see e.g. \cite{Rudyak, Lurie:lectures} for a discussion and background on these theories.  Compatible with our notation, there is a natural pairing
\begin{equation}
H_*(X;\bK) \otimes_{\bK_*} H^*(X;\bK) \to \bK_*
\end{equation}
and the groups $H_*(X;\bK)$ and $H^*(X;\bK)$ are (both left and right) $\bK_*$-modules.
\begin{rem}
	In the case where $H^*(X;\bK)$ is an ordinary cohomology theory, i.e. $\bK = HR$ is an Eilenberg-Maclane spectrum, this group agrees on compact Hausdorff spaces with \v{C}ech cohomology \cite{Huber}, which provides the link with the construction of \cite{pardon2016algebraic}.
\end{rem}

The following demonstrates one way in which $H^*(-;\bK)$ behaves like \v{C}ech cohomology.

\begin{lemma} \label{lemma Cech property} 
	Let $Z$ be a compact subset of a metric space $X$.
	Then
	\begin{equation}
	H^*(Z;\bK) = \varinjlim_{U \supset Z} H^*(U;\bK)\end{equation}
	where the direct limit is taken over open neighborhoods of $Z$ in $X$.
\end{lemma}

\begin{proof} For this result, it matters that we pick a model for the spectrum $\bK$ in which each space has the homotopy type of a CW complex; this is possible since the loop space of a countable CW complex is homotopy equivalent to a countable CW complex by \cite{Milnor:CW}.  The result then follows from \cite[Corollary 2.8]{LeeRaymond}, which is in turn based on a version of the Tietze extension theorem.  They show that for a paracompact Hausdorff closed pair $(X,A)$, and a nested system $(X_n, A_n)$ for that pair comprising neighbourhoods in a fixed space $Y$,  there is a bijection of homotopy classes
\begin{equation}
\varinjlim [(X_n, A_n), (E,E’)] \ = \  [(X,A), (E,E’)]
\end{equation}
for any CW pair $(E,E')$. We apply this with $(E,E') = (K,\{pt\})$ where $K$ is one of the underlying spaces of the spectrum $\bK$ and $\{pt\}$ its basepoint.  See also the related \cite[Proposition 6.6]{AbouzaidBlumberg2021}.
\end{proof}

If $Y \subset X$ are topological spaces then we define
$X|Y$ to be the cone of the inclusion map $X \backslash Y \to X$, viewed as a pointed space, which is compatible with the notation $\widetilde{H}^*(X|Y;\bK) = H^*(X,X\backslash Y;\bK)$. We also record the 
\emph{Atiyah-Hirzebruch spectral sequence}:

\begin{thm}
For a CW-complex $X$, there is a spectral sequence computing $H^*(X;\bK)$
whose $E_2$-page is $H^p(X;\bK_q)$. The spectral sequence is functorial under continuous maps between CW-complexes.
\end{thm}

For a space $X$, we write $X^c$ for its one point compactification.
We write $H^*_c(X;\bK) = \widetilde{H}^*(X^c;\bK)$ for compactly supported cohomology and $H_*^c(X;\bK)=\widetilde{H}_*(X^c;\bK)$ for locally finite homology.
If $X$ is a $G$-space, then we have the equivariant analogues of these groups:
\begin{equation}
	H_*^G(X;\bK) = H_*(EG \times_G X;\bK), \ 	H_*^{G,c}(X;\bK) = H_*(EG \times_G X^c, BG;\bK)
\end{equation}
\begin{equation}
	 H^*_G(X;\bK) = H^*(EG \times_G X;\bK), \ 	H^*_{G,c}(X;\bK) = H^*(EG \times_G X^c, BG;\bK).
\end{equation}
The same argument as for Lemma \ref{lemma Cech property} shows:

\begin{lemma} \label{lemma equivariant cech} 
	The $G$-equivariant version of Lemma \ref{lemma Cech property} holds.
\end{lemma}

\subsection{Orientations of (virtual) bundles} \label{Section orientations of bundles}

If $\xi : V \to X$ is a vector bundle over a space $X$,
 the \emph{Thom space} $X^V = X^\xi$
is the pointed space given by the quotient $D(V) / S(V)$ where $S(V)$ and $D(V)$
are the unit sphere and disk bundles with respect to some choice of metric.
Equivalently, we can think of it as the space $V|X$; one can then similarly define the Thom space of a microbundle $X \lra{s} E \lra{p} X$ to be $X^E := E|(s(X))$.

\begin{defn} \label{DEFN oriented vector bundle}
	A \emph{$\bK$-orientation} of $V$ is a class $\fro \in H^*(X^V;\bK)=H^*(V|X;\bK)$ lying in degree $\rk(V)$ 
	whose restriction to $H^*((V|_x)\left|x\right.;\bK) \cong H^*(pt;\bK)$ is a unit for each $x \in X$.
	$V$ is \emph{$\bK$-orientable} if it admits a $\bK$-orientation.

A \emph{ $\bK$-orientation} for a microbundle $X \lra{s} E \lra{p} X$ is a class in $H^*(E|s(X);\bK)$ of degree $\rk(E)$ whose restriction to $H^*((E|_x)|x;\bK)$ is a unit for each $x \in X$.
An \emph{oriented manifold} is a manifold with an orientation of its tangent microbundle.
\end{defn}

If a microbundle admits a vector bundle lift, then both notions of orientation coincide in a
natural way and we will not distinguish between them.

A trivial bundle $\bR^k \times X$ is always $\bK$-oriented
and the set of its $\bK$-orientations bijects with the set of units
in $H^*(\bR^k|0;\bK)=H^{*-k}(pt;\bK)$ (via pullback to $\bR^k \times X$ under the projection map).
The chosen trivialization of $\bR^k \times X$ determines a canonical orientation given by pulling back the identity element in $H^*(pt;\bK)$.

\begin{Remark} 
A topological manifold may be $H\bZ$-orientable but not be $\bK$-orientable for a Morava $K$-theory $\bK$ at the prime $p=2$. See \cite[Chapter IX, Section 7]{Rudyak}.
\end{Remark}

We shall not require any sophisticated notion of equivariant orientation in this paper: we will obtain our $\bK$-orientations from a choice of geometric data (e.g. stable almost complex structure on $V$) which we can take to be $G$-equivariant. Such a choice  will induce a $\bK$-orientation of the induced  $G$-bundle over the Borel construction, which is all that we will require. Such a datum is weaker than the notion of equivariant orientation in the sense of \cite{CMW}.

The following result is classical, and is reviewed at the level of spectra in Section \ref{Subsec:thom_review}. Similar results hold for microbundles and also for compactly supported cohomology.

\begin{theorem} (Thom isomorphism theorem). \label{thm:thom_iso}
	Let $X$ be homotopic to a CW complex (for instance, a topological manifold \cite{Milnor:CW}). 	If $\xi: V\to X$ is a $\bK$-oriented rank $d$ vector bundle with orientation $\fro$,
	then there are Thom isomorphisms
	\begin{equation} \label{eqn:thom}
	\begin{tikzcd}
		H^*(X;\bK) \arrow[rr, "{\fro \cup \xi^*(-)}"] && H^{*+d}(X^V;\bK), \\ 
		H_{*+d}(X^V;\bK)  \arrow[rr, "{\fro \cap \xi^*(-)}"] && H_*(X;\bK).
		\end{tikzcd}
	\end{equation}
\end{theorem}

Suppose that $\check{V}$ is another vector bundle over $X$. For the virtual vector bundle $V-\check{V}$, we define 
\begin{equation}
H^*(X^{V - \check{V}};\bK) = H^{*+N}(X^{V \oplus \check{V}^\perp};\bK)
\end{equation}
where $\check{V}^\perp$ is any vector bundle satisfying
$\check{V} \oplus \check{V}^\perp \cong \bR^N$. We set $H^*(X^{-\check{V}};\bK) = H^*(X^{0 - \check{V}};\bK)$. Again, similar definitions hold for microbundles.  (One can make sense of $X^{V - \check{V}}$ in the category of spectra, as we will recall in Section \ref{Subsec:thom_review}.)

Suppose that $\check{\fro}$ is a $\bK$-orientation on $\check{V}$.
Then the Theorem \ref{thm:thom_iso} gives us a natural bijection between $\bK$-orientations of $V$
and $\bK$-orientations of $V \oplus \check{V}$ by the following
map:
\begin{equation} \label{equation orientation cup}
\begin{tikzcd}
	H^*(V|X;\bK) \arrow[rr, "{\pi^*_{\check{V}}(\check{\fro}) \cup \pi_V^*(-)}"] && H^{*+d}(\check{V} \oplus V|X;\bK) \end{tikzcd}
\end{equation}
where $\pi_V$ and $\pi_{\check{V}}$ are the natural projection maps from $V \oplus \check{V}$ to $V$ and $\check{V}$ respectively.
There is a similar identification for microbundles.

We conclude that, when discussing orientations,  one does not need to distinguish between $V$ and $V \oplus \bR^k$; thus we can also talk about orientations of virtual vector bundles.

\begin{defn} \label{DEFN oriented difference}
	A \emph{$\bK$-orientation} on the formal difference $V - \check{V}$ is a $\bK$-orientation on $V \oplus \check{V}^\perp$
	where we have chosen a vector bundle $\check{V}^\perp$ with $\check{V} \oplus \check{V}^\perp \cong \bR^N$.
\end{defn}
	
It is straightforward to apply this notion to the Borel equivariant setting.

It follows from  Equation \eqref{equation orientation cup} that 
one can say $V- \check{V}$ is {\it $\bK$-oriented} without referring to a specific choice of $\check{V}^{\perp}$. 
The discussion above also ensures that if $\check{V}$ is the trivial vector bundle $0$,
then both notions of $\bK$-orientation for $V = V - \check{V}$ coincide (Definitions \ref{DEFN oriented vector bundle} and \ref{DEFN oriented difference}).
Again, there is a similar story for microbundles.

The following result is classical\footnote{This also holds for topological manifolds on replacing the tangent bundle with the tangent microbundle. We do not know a canonical reference, but it is proved and extensively discussed in \cite{AbouzaidBlumberg2021}.}:
\begin{theorem} (Atiyah Duality)
	If $W$ is a smooth manifold and $Z \subset W$ is a closed subset, 
	\begin{equation} \label{eqn:atiyah_duality}
		H_*(W|Z;\bK) \cong H^{-*}(Z^{-TW|_Z};\bK).
	\end{equation}
      \end{theorem}
      \begin{rem}
One can formulate Atiyah Duality as a statement about spectra, i.e. as the fact that $W|Z$ and $Z^{-TW|_Z}$ are dual. In this form, there is a straightforward equivariant generalisation to the setting in which a group $G$ acts on $W$, preserving $Z$. However, one does not conclude from that an generalisation of Equation \eqref{eqn:atiyah_duality} involving (Borel) equivariant homology and cohomology, as one can evidently see by noting that the equivariant homology and cohomology groups of the Borel construction are not in general isomorphic. The notion of ambidextrous cohomology theories, discussed in Section \ref{Sec:Cheng} below, provides a context in which such an isomorphism holds. 
      \end{rem}

 We recall that the Thom isomorphism and Atiyah duality are compatible in the following sense: if $W$ is a $\bK$-oriented compact manifold and $\scrW \to W$ is a $\bK$-oriented rank $r$ vector bundle, then the total space of $\scrW$ is also oriented as a non-compact manifold, and there is a commuting diagram
 \begin{equation} \label{eqn:atiyah_meets_thom}
 \xymatrix{
 H^*(W;\bK) \ar[rr]^-{D_W}\ar[d]_-{\cdot u_{\scrW}}  && H_{n-*}(W;\bK) \ar[d]_{i_*} \\
 H^{*+r}_c(\scrW;\bK) \ar[rr]^-{D_\scrW} && H_{n-*}(\scrW;\bK)
 }
 \end{equation}
 where $i: W \to \scrW$ is the natural inclusion and on the left we have the composition of cup product with the Thom class and the natural map $H^*(\scrW|0;\bK) \to H^*_c(\scrW;\bK)$.

\begin{Remark}
	If $Z\subset W$ is not a neighbourhood deformation retract,  \eqref{eqn:atiyah_duality} 
	uses a continuity argument based on Urysohn's Lemma, cf. Lemma \ref{lemma Cech property} and \cite{Huber}.
\end{Remark}

\subsection{Orientations of global Kuranishi charts} \label{Section Orientations of KC}

In order to define virtual fundamental classes, we need a notion of orientation for a global Kuranishi chart.
Let us fix a ring spectrum $\bK$ as above.
Fix a compact Lie group $G$. Let $\frg$ be the Lie algebra of $G$.
By abuse of notation, for a $G$-space $X$ we will sometimes write
$\frg$
instead of the $G$-vector bundle $\frg \times X$
where $G$ acts diagonally.

\begin{defn} \label{DEFN oriented Kuranishi}	
	A (Borel equivariant) \emph{$\bK$-orientation} on a global Kuranishi chart $(G,\scrT,E,s)$ is a $\bK$-orientation
	on the virtual bundle over the Borel construction of $\scrT$ associated $T_\mu \scrT - E - \frg$.
	It is \emph{$\bK$-oriented} if it is equipped with a $\bK$-orientation.
\end{defn}

If it is a smooth chart then the microbundle $T_\mu \scrT$ can be replaced by $T\scrT$ (Example \ref{example natural smooth lift}). 
Note that a germ equivalence, stabilization or group enlargement $(G',\scrT',E',s')$ of a global
Kuranishi chart $(G,\scrT,E,s)$ induces a natural bijection between $\bK$-orientations of
$(G,\scrT,E,s)$ and $(G',\scrT',E',s')$ (Equation \eqref{equation orientation cup}).

The virtual fundamental class will be built from various classes
on $\scrT$ and the Thom space of $E$.  The following Lemma makes it possible to work with charts which satisfy a superficially stronger orientation hypothesis.

\begin{lemma} \label{lemma stronger orientation}
	Let $(G,\scrT,E,s)$ be a $\bK$-oriented smooth global Kuranishi chart.
	Then there is a smooth $G$-vector bundle $W$ over $\scrT$
	so that the stabilization $(G,\scrT',E',s')$ of $(G,\scrT,E,s)$ by $W$ 
	has the property that $T\scrT' - \frg$ and $E'$ both admit Borel equivariant $\bK$-orientations, whose difference recovers the $\bK$-orientation on $T\scrT' - E' - \frg$ obtained from Definition \ref{DEFN oriented difference}
	and Equation \eqref{equation orientation cup}.
\end{lemma}
\begin{proof}
	Take the vector bundle $W$ to be given by $E'$.
	We can then choose a $\bK$-orientation on $TW - \frg$ and hence on $E'$.
\end{proof}

\begin{lemma} \label{lemma pullback orientation}
	Let $\K=(G,\scrT,E,s)$ be a global Kuranishi chart
	and let $\Pi : \scrT \to B$ be a $G$-equivariant topological submersion as in Definition
	\ref{defn product neighborhood}.
	Let $f : B' \to B$ be a $G$-equivariant continuous map between $\bK$-oriented
	manifolds $B'$ and $B$.
	If $\K$ admits a $\bK$-orientation then its pullback $f^*\bK = (G,f^*\scrT,f^*E,f^*s)$
	admits a natural \emph{pullback orientation}.
\end{lemma}
\begin{proof}
	By Corollary \ref{corollary product isomorphism}, we have a natural
	$G$-equivariant splitting $T_\mu(\scrT) \cong T^{vt}_\mu(\scrT) \oplus \pi^*(T_\mu B)$
	where $T^{vt}_\mu(\scrT)$ is the vertical tangent microbundle of $\Pi$.
	Using this splitting, we see that $T^{vt}_\mu(\scrT) - E - \frg$ is naturally $\bK$-oriented.
	Hence $T^{vt}_\mu(f^*\scrT) - f^*E - \frg$ is also $\bK$-oriented where $T^{vt}_\mu(f^*\scrT)$
	is the vertical tangent microbundle of the natural map $f^*\scrT \to B'$.
	Since $T_\mu B'$ is also $\bK$-oriented,  by the $G$-equivariant
	splitting isomorphism $T_\mu(f^*\scrT) \cong T^{vt}_\mu(f^*\scrT) \oplus \pi^*(T_\mu B')$,
	we get a $\bK$-orientation for $f^*\bK$.
\end{proof}

\subsection{Morava $K$-theories}

For background on Morava $K$-theories, the reader can consult \cite{Lurie:lectures,Wurgler}. This paper will not define them, but only use a number of properties which we collect here.  

By a \emph{stable complex structure} on a vector bundle $V$, we mean
a complex structure on $V \oplus \bR^N$ for some $N \in \bN$.

\begin{Lemma}
Let $\bK = K_p(n)$ denote the Morava $K$-theory associated to a prime $p$ and an integer $n>0$. 
\begin{enumerate}
\item The coefficient ring is $H^*(pt;\bK) = \bF_p[v^{\pm 1}]$ with  $|v| = 2(p^n-1)$.
\item If $X$ and $Y$ are homotopy equivalent to CW complexes then the K\"unneth theorem holds: $H^*(X\times Y;\bK) \cong H^*(X;\bK)\otimes_{\bK_*} H^*(Y;\bK)$.
\item Any vector bundle with a stable complex structure is $\bK$-oriented.
Hence any stably almost complex smooth manifold $X$ is $\bK$-oriented.
\item If $p > 2$, then any oriented vector bundle is $\bK$-oriented.
Hence if $p>2$, any oriented smooth manifold is $\bK$-oriented.
\end{enumerate}
\end{Lemma}

\begin{proof} The coefficients form a graded field, i.e. a graded ring with the property that all non-zero homogeneous elements are invertible, and all graded modules are projective. The K\"unneth theorem then follows essentially formally. (Note that this involves tensoring left and right $\bK_*$-modules; it is fact that the left and right module structures agree, cf. Lemma \ref{Lem:K-linear} below.) For a full discussion and for orientations, see \cite{Rudyak,Lurie:lectures}. \end{proof}

In fact, we shall use a slightly more general notion of Morava $K$-theories, whose construction we discuss in Section \ref{Sec:more_coefficients}:
\begin{Proposition} \label{Prop:coefficients-text}
For a prime $p$ and natural numbers $n, k>0$ there are cohomology theories $K_{p^k}(n)$ with the following properties:
\begin{itemize}
\item the coefficients are given by $K_{p^k}(n)_* = \bZ/(p^k)[v_n^{\pm}]$ with $|v_n| = 2(p^n-1)$;
\item any vector bundle with a stable complex structure is $K_{p^k}(n)$-oriented.
Hence any stably almost complex smooth manifold $X$ is $K_{p^k}(n)$-oriented.
\item the spectrum $K_{p^k}(n)$ is a $k$-fold iterated extension of $K_{p}(n)$.
\end{itemize}
\end{Proposition}

Having theories based upon rings $\bZ/(p^k)$ and not just $\bZ/p$ gives access to integral cohomology.
We will use the following lemma to show that the map \eqref{Sec:Splitting} with $\bK=\Z$ is a split epimorphism.

\begin{lemma} \label{lemma morava split}
	Given a map of finite-dimensional compact manifolds  $f : X\to Y$,  if $H^*(Y;K_{p^k}(n)) \to H^*(X;K_{p^k}(n))$ is a split epimorphism for all Morava $K$-theories, then the map $H^*(Y,\bZ) \to H^*(X, \bZ)$ is also a split epimorphism.
\end{lemma}
\begin{proof}
	For $2(p^n-1) > \dim_{\bR}(Y)$, the Atiyah-Hirzebruch spectral sequence for Morava $K$-theory degenerates,
	and 
	\begin{equation}
		\begin{split}
	H^*(X;K_{p^k}(n)) \cong H^*(X;\bZ/p^k) \otimes_{\bZ/p^k} \bZ/p^k [v^{\pm}], \\ 
	H^*(Y;K_{p^k}(n)) \cong H^*(Y;\bZ/p^k) \otimes_{\bZ/p^k} \bZ/p^k [v^{\pm}].
			\end{split}
	\end{equation}
	Hence
	\begin{equation} \label{eqn pk split}
		f_* : H^*(Y;\bZ/p^k) \lra{} H^*(X;\bZ/p^k)
	\end{equation}
	is a split epimorphism for each prime $p$ and each $k \in \bN$.
	By the Chinese remainder theorem combined with the universal coefficient theorem,
	the map
	\begin{equation} \label{eqn fZ}
		f_* : H^*(Y;\bZ) \lra{} H^*(X;\bZ)
	\end{equation}
	splits when tensored with $\Z/m$ for any integer $m$.
	This is enough to show that \eqref{eqn fZ} splits.
\end{proof}

\subsection{Cheng's theorems\label{Sec:Cheng}}
\label{sec:chengs-theorems}

We would like to lift Atiyah duality  \eqref{eqn:atiyah_duality} to a statement in the $G$-equivariant setting. 
The relevant results are due to Cheng \cite{Cheng}, who works in the language of equivariant homotopy theory as established in \cite{LMSM}. Here we formulate his results at the cohomological level.  For completeness, we briefly recall the underlying spectrum-level construction in Section \ref{Sec:review}. (That review indicates that the results should also hold for locally linear actions on topological manifolds, but we will stick to Cheng's smooth setting for simplicity.) 
 
Throughout this section, $\bK$ denotes a ring spectrum which is complex-oriented and $K(n)$-local for some Morava $K$-theory $K(n)$.  The paper \cite{Cheng} considers only the case in which $\bK = K(n)$, but the arguments apply equally well in the more general setting, as follows from our r\'esum\'e of Cheng's work in Section \ref{Sec:Cheng_review}, and in particular from Lemma \ref{lem:norm_map_equivalence_K-local}.

\begin{Theorem}[Cheng] \label{thm:Cheng-homological}
Let $M$ be a smooth manifold on which a compact Lie group $G$ acts smoothly with finite stabilisers. If $M$ is closed, there is a distinguished equivalence

\begin{equation} \label{eqn:Cheng_equivalence}
\lambda_{G,M}: H_*^G(M;\bK) \stackrel{\sim}{\longrightarrow} \widetilde{H}^{-*}_G(M^{-TM\oplus \frg};\bK).
\end{equation}
If $M$ is open, and $M^c$ denotes the one-point compactification, then we have 
\begin{equation}\label{eqn:relative_lambda_homol}
\lambda_{G,M}: H_*^G(M^c ; \bK)  \stackrel{\sim}{\longrightarrow} H^{-*}_G(M^{-TM\oplus \frg};\bK).
\end{equation}

\end{Theorem}
\vspace{1em}

\begin{Notation} If $U$ is a compact smooth manifold with boundary $\partial U$, we abbreviate 
\begin{equation}
U^c := \mathrm{int}(U)^c = U / \partial U.
\end{equation}
With that convention, \eqref{eqn:relative_lambda_homol} also applies to manifolds with boundary.
\end{Notation}

If the virtual bundle $-TM\oplus\frg$ is $\bK$-oriented, then the cohomology groups appearing in Theorem \ref{thm:Cheng-homological} are just cohomology groups of $EG\times_G M$ with a shift in degree via the discussion on virtual bundles in Section \ref{Section orientations of bundles}.

There are two important addenda to Cheng's result.  The first is a compatibility with open inclusions: 

\begin{Addendum} \label{Add:relative_Cheng-homological}
If $U \subset M$ is a $G$-invariant open submanifold there is a commutative diagram
\begin{equation}
  \begin{tikzcd}
H_*^G(M^c;\bK) \ar[d] \ar[r,"\lambda_{G,M}"] &  H^{-*}_G(M^{-TM\oplus \frg};\bK) \ar[d] \\ 
H_*^G(U^c;\bK) \ar[r,"\lambda_{G,U}"] & H^{-*}_G(U^{-TU\oplus \frg}; \bK)    
  \end{tikzcd}
\end{equation}
where the vertical arrows are respectively induced by the collapse map $M^c \to U^c$ and by restriction.
\end{Addendum}

The second addendum is a compatibility with `change-of-group' i.e. with the choice of presentation of a given orbifold as a global quotient.

\begin{Addendum} \label{Add:change_of_group-homological}
Suppose there is an equivalence of smooth orbifolds $M/G = N/H$.  There is then a commuting diagram
\begin{equation}\label{eqn:group_enlargement_general}
\xymatrix{
H^{G}_{*}(M^c;\bK) \ar[rr]^-{\lambda_{G,M}} \ar@{-}[d]_{\simeq} & &  H^{-*}_{G}(M^{-TM\oplus \frg};\bK) \\
H^{H}_{*}(N^c;\bK)\ar[rr]^-{\lambda_{H,P}} & & H^{-*}_{H}(N^{-TN\oplus\frh};\bK)  \ar@{-}[u]_{\simeq}
}
\end{equation}
where the vertical maps are induced by canonical weak equivalences $EG \times _G M \simeq EH \times_H N$ and $(EG\times_G M)^{-TM + \frak{g}} \simeq (EH\times_H N)^{-TN +\frak{h}}$.
\end{Addendum}

Cheng introduces a manifold $P$ (the fibre product of $M$ and $N$ over the common orbifold quotient) with $P/(G\times H)$ isomorphic as an orbifold to $M/G \cong N/H$, and the vertical equivalences are obtained by factoring through this. 

As a special case, if $M$ is a smooth free $G$-manifold, and hence $M/G = N$ is a smooth manifold, we get a diagram
\begin{equation}
  \begin{tikzcd}
H_*^G(M^c; \bK) \ar[r,"\lambda_{G,M}"] \ar[d] & H^{-*}_G(M^{-TM\oplus \frg};\bK) \ar[d] \\
H_*(N^c;\bK) \ar[r,"\lambda_{N}"]& H^{-*}(N^{-TN};\bK)    
  \end{tikzcd}
\end{equation}
where $\lambda_N$ is the equivalence induced by Atiyah duality.

\begin{Remark} Cheng states Addendum \ref{Add:change_of_group-homological} in the case where $M$ and $N$ are closed, but his proof also yields the case in which these have non-empty boundary, upon replacing $N$ and $M$ by $N^c$ and $M^c$ on the left half of the diagram as above.
\end{Remark}

Since we have assumed that $\bK$ represents a multiplicative cohomology theory, for any $G$-space $U$ there is a natural element
\begin{equation}
1 \in H^0_G(U;\bK) = H^0(EG \times_G U;\bK)
\end{equation}
obtained from the unit in $\bK_*$ and the module structure over the coefficients. 

\begin{Definition} Suppose that $U$ is open (or compact with boundary) and $-TU \oplus \frg$ is $\bK$-oriented.  Then the image $[U^c]_{\bK}$ of the unit $1$ under the isomorphism
\begin{equation}
(\lambda_{G,U})^{-1}: H^0_G(U;\bK) \to H^{G,c}_{\dim(U)-\dim(G)}(U^c; \bK)
\end{equation}
is called the ($G$-equivariant) \emph{fundamental class with compact supports} of $U$.
\end{Definition}

There is a natural pairing
\begin{equation}
H^*_{G,c}(U;\bK) \otimes H_*^{G,c}(U;\bK) \to \bK_*
\end{equation}
and in particular, when $-TU \oplus \frg$ is $\bK$-oriented,  there is a natural map 
\begin{equation} \label{eqn:evaluate_against_fundamental_class}
\langle \bullet, [U^c]_{\bK} \rangle: \ H^{\dim(U)-\dim(G)}_{G,c}(U;\bK) \longrightarrow \bK_*
\end{equation}
given by evaluation against the compactly supported fundamental class.  If $Z\subset \mathrm{int}(U)$ is a compact $G$-invariant subset of the interior of $U$, then the base-point-preserving collapse map $U/\partial U \to U/(U\backslash Z)$ induces a map 
\begin{equation}
H^*_{G}(U|Z;\bK) \to H^*_{G,c}(U;\bK)
\end{equation}
and hence a map $H^*_G(U|Z;\bK) \to \bK_*$ by composing with \eqref{eqn:evaluate_against_fundamental_class}.

The following result generalises the compatibility between classical Atiyah duality and the Thom isomorphism, diagram \eqref{eqn:atiyah_meets_thom},  to the orbifold situation.

\begin{Lemma} \label{Lem:compatible_virtual_classes}
	Let $\scrT$ be a smooth $G$-manifold with boundary and $q: \scrW \to \scrT$ a $G$-equivariant vector bundle over $\scrT$. Then there is a commuting diagram
	\begin{equation}
		\xymatrix{
			H^*_{G}(\scrT;\bK) \ar[rr]^-{\lambda_{G,\scrT}} \ar[d]_{q^*} & & H^{G,c}_{\dim(\scrT) - \dim(G)-*}(\scrT;\bK) \ar[d]^{\mathrm{Thom}_{\scrW}}  \\
			H^*_{G}(\scrW;\bK) \ar[rr]^-{\lambda_{G,\scrW}} & & H^{G,c}_{\dim(\scrT) + \rk(\scrW)- \dim(G)-*}(\scrW;\bK) 
		}
	\end{equation}
	where the right vertical map is induced by the Thom isomorphism for $\scrW$.
\end{Lemma}

\begin{proof}[Sketch] This follows from the construction of the ambidexterity maps in \cite{Cheng} and reviewed in Section \ref{Sec:Cheng}.  The Atiyah dual of the Thom spectrum $\scrT^{\scrW}$ is $\scrT^{-T\scrT-\scrW}$, using that $T\scrW$ and $q^*T\scrT \oplus q^*\scrW$ are stably isomorphic, and $q$ is homotopic to the identity. Before passing to $G$-fixed points and using contractibility of the cofibre, the map $\lambda_{G,M} \wedge \bK$ is obtained from a map $\alpha_{G,M}$ which is parametrized Atiyah duality.
\end{proof}

In order to take advantage of Cheng's results, we will define a virtual class for smooth Kuranishi charts, and show the result is independent of our operations (stabilisation, germ equivalence, group enlargement).  Since we know we can apply smoothing theory to topological Kuranishi charts that have appropriate fibrewise smooth structures, this will suffice for our applications.

\subsection{Pushforward maps and the virtual fundamental class}\label{sec:vfc}

We continue to work as in Section \ref{sec:chengs-theorems} with a fixed cohomology theory $\bK$ and a compact Lie group $G$ with Lie algebra $\frg$.

\begin{defn} \label{vfc on thickened space}
	Let $\K := (G,\scrT,E,s)$ be a smooth Kuranishi chart and let $Z = s^{-1}(0)$.
	Let $f : \scrT \to X$ be a continuous $G$-equivariant map to a $\bK$-oriented compact smooth manifold $X$.
	Suppose $T\scrT - \frg$ and $E$ are both $\bK$-oriented.
	Let $d = \vdim(\K)$, $m = \dim(E)$, $k = \dim(\frg)$ and $l = \dim(X)$.
	Then the \emph{pushforward map}
	is defined to be the composition:
\begin{equation} \label{eqn pushforward}
	\begin{split}
		f_*^\K : H^*_G(\scrT;\bK) \lra{\lambda_{G,\scrT}^{-1}} H_{d+m-*}^G(\scrT^c;\bK) \lra{res} H_{d+m-*}^G(\scrT|Z;\bK) \lra{s_*} \dots
		\\ 
		\dots H_{d+m-*}^G(E|\scrT;\bK)
		\lra{\small{Thom}} H_{d - *}^G(\scrT;\bK) \lra{f_*} H_{d-*}(X;\bK)
		\lra{\lambda_{0,X}} H^{*-d+l}(X;\bK)
\end{split}
\end{equation}
where $\lambda_{G,\scrT}$
and $\lambda_{0,X}$
are Poincar\'{e} duality maps as in Theorem \ref{thm:Cheng-homological} composed with the corresponding Thom isomorphism maps 
\begin{equation}
H^{-*}_G(\scrT^{-T\scrT \oplus \frg};\bK) \cong H_{d+m-*}^G(\scrT;\bK), \quad
H^{*-d}(X^{-TX};\bK) \cong H^{*-d+l}(X;\bK).
\end{equation}
\end{defn}

\begin{remark} \label{rmk:swapped} One can use  Lemma \ref{Lem:compatible_virtual_classes} to swap the order of applying Thom isomorphism and ambidexterity, and describe the pushforward $H^d_G(\scrT;\bK) \to \bK_*$ via the  pairing \eqref{eqn:evaluate_against_fundamental_class}. 
\end{remark}

First of all, we will show that the push-forward  is well defined with respect to germ equivalence.

\begin{lemma} \label{lemma germ vfc well behaved}
	Let $\K := (G,\scrT,E,s)$ be a smooth Kuranishi chart with $T\scrT - \frg$ and $E$ both $\bK$-oriented and
	let $f : \scrT \to X$ be a continuous $G$-equivariant map to a $\bK$-oriented smooth manifold $X$.
	Let $U \subset \scrT$ be an open neighborhood of $\scrT$
	and let $\K|_U := (G,U,E|_U,s|_U)$ be the restriction of the Kuranish chart to $U$.
	Then the restriction of $f_*^\K$ to $U$ is  the pushforward map $f_*^{\K|_U}$ associated to $\K|_U$.
\end{lemma}
\begin{proof}
	This follows from Addendum \ref{Add:relative_Cheng-homological} and the fact that Thom classes are functorial with respect to pullbacks.
\end{proof}

The lemma above enables us to construct pushforward maps from the domain of the footprint map as follows.

\begin{defn} \label{vfc on zero set}
	Let $\K := (G,\scrT,E,s)$ be a smooth Kuranishi chart for a compact metric space $M$
	with $T\scrT - \frg$ and $E$ both $\bK$-oriented.
	Let $f : M \to X$ be a continuous map to a $\bK$-oriented smooth manifold $X$.
	Let $\widetilde{f} : \scrT \lra{} X$ be any $G$-equivariant continuous map whose restriction to $s^{-1}(0)$
	is the natural composition $s^{-1}(0) \lra{pr} M \lra{f} X$.

	Then the \emph{pushforward map} 
	is defined to be the map:
	\begin{equation} \label{eqn vfc on zero set}
	\begin{tikzcd}
	f^M : H^*(M;\bK) \leftarrow H^*_G(s^{-1}(0);\bK) \stackrel{\cong}{\longleftarrow} \\ \qquad \qquad  \varinjlim_{U \supset s^{-1}(0)} H^*_G(U;\bK) \arrow[rr, "\widetilde{f}_*^{\K|_U}"] && H^{*+\vdim(M)-\dim(X)}(X;\bK)
	\end{tikzcd}
	\end{equation}
where the maps come from Lemma \ref{lemma equivariant cech} and  Lemma \ref{lemma germ vfc well behaved} respectively.
\end{defn}

By Lemma \ref{lemma Cech property}, this does not depend on the choice of extension $\widetilde{f}$ of $f$.
We will now show that the push-forward behaves well with respect to stabilization.

\begin{lemma} \label{lemma stabilization}
	Let $\K := (G,\scrT,E,s)$ be a smooth Kuranishi chart
	with $T\scrT - \frg$ and $E$ both $\bK$-oriented.
	Let $f : \scrT \to X$ be a continuous $G$-equivariant map to a $\bK$-oriented smooth manifold $X$. Let $\pi : W \to \scrT$ be a smooth $\bK$-oriented $G$-vector bundle over $\scrT$ and
	let $\K_W$ be the stabilization of $\K$ by $W$.
	Then 
	\begin{equation}
	f_*^{\K_W} \circ \pi^* = f_*^\K.\end{equation} In particular,
	$f_*^\K$ is independent of the stabilization operation.
\end{lemma}
\begin{proof}
This follows from Lemma \ref{Lem:compatible_virtual_classes}.
\end{proof}

\begin{corollary} \label{corollary orientation dependence}
	Let $\K := (G,\scrT,E,s)$ be a smooth oriented Kuranishi chart for a compact metric space $M$ (Definition \ref{DEFN oriented Kuranishi}).
	Let $f : \scrT \to X$ be a continuous $G$-equivariant map to a $\bK$-oriented smooth manifold $X$. The push-forward map $f_*^\K$ depends only on the orientation on the Kuranishi chart in the following sense: 
	
	Let $\fro_\scrT$, $\fre_\scrT$ be $\bK$-orientations on $T\scrT - \frg$ and
	$\fro_E$, $\fre_E$ $\bK$-orientations on $E$
	so that the `difference' orientations on $T\scrT - E - \frg$
	induced by $\{\fro_\scrT, \fro_E\} $ and $\{\fre_\scrT, \fre_E\}$ both agree with the given orientation on $\scrK$. 
	Then the pushforward map $f_*^\K$ defined using $\{\fro_T, \fro_E\}$ agrees 
	with that defined using $\{\fre_T, \fre_E\}$.
\end{corollary}
\begin{proof}
	Stabilize $\scrT$ by $E$ so that the orientation on $TE|_\scrT - \frg = T\scrT \oplus E - \frg$
	is $\fro_\scrT \oplus \fre_E = \fre_\scrT \oplus \fro_E$, and use the previous lemma.
\end{proof}

The corollary above enables us to define the pushforward for oriented (but not necessarily smooth)  Kuranishi charts.

\begin{defn} \label{defn pushforward on oriented chart}
	Let $\K := (G,\scrT,E,s)$ be an oriented global Kuranishi chart for a compact metric space $M$
	and suppose that $\scrT$ is the total space of a fiberwise smooth $C^1_{loc}$ $G$-bundle
	whose base $B$ is a smooth $G$-manifold (Definition \ref{defn fiberwise smooth}). 
	Suppose that the $G$-action on $\scrT$ only has finitely many orbit types.
	Let $f : M \to X$ be a continuous map to a $\bK$-oriented smooth manifold $X$.
	
	Then the \emph{pushforward map} 
	\begin{equation}
	f_*^M : H^*(M;\bK) \to H^{*-\vdim(M)+\dim(X)}(X; \bK)
	\end{equation}
	is constructed as follows:
	By Proposition \ref{prop Kuranishi smoothing} and Lemma \ref{lemma stronger orientation}, replacing  $\K$ with a stabilization if necessary, we can assume that it is a smooth Kuranishi chart
	equipped with $G$-equivariant $\bK$-orientations on $T\scrT - \frg$ and $E$ whose difference agrees with the  given orientation.
	We now define $f_*^M$ as in Definition \ref{vfc on zero set}.
	\end{defn}
	
	\begin{defn}\label{defn:vfc} In the setting of Definition \ref{defn pushforward on oriented chart}, 
	the \emph{$\bK$-virtual fundamental class} of $M$  is the map
	$\vfc_M : H^*(M;\bK) \to \bK_*$ given by $f_*^M$ where $f : M \to \{pt\}$ is the map to a point.
\end{defn}

The pushforward $f_*^M$ (and hence the virtual class) does not depend on the choice of stabilization, by Lemma
\ref{lemma stabilization}, nor on the choices of orientation on $T\scrT - \frg$ and $E$ by Corollary \ref{corollary orientation dependence}.
It is also does not depend on $\K$ up to germ equivalence
by combining Lemma \ref{lemma germ vfc well behaved} with  the fact that germ equivalence commutes with stabilization.

\begin{Remark}
It should be possible  to generalize the pushforward
map to the  `relatively $\bK$-oriented' case, where only $T\scrT - \frg - E - \widetilde{f}^*TX$ is $\bK$-oriented
where $\widetilde{f}$ is as in Definition \ref{vfc on zero set}. We will not make use of such a generalization in this paper.
\end{Remark}

\begin{lemma} \label{lemma group enlargement}
	Let $\K := (G,\scrT,E,s)$ be a $\bK$-oriented global Kuranishi chart
	and suppose that $\scrT$ is the total space of a fiberwise smooth $C^1_{loc}$ $G$-bundle
	whose base $B$ is a smooth $G$-manifold.
	Suppose that the $G$-action on $\scrT$ only has finitely many orbit types.
	Let $f : \scrT \to X$ be a continuous $G$-equivariant map to a $\bK$-oriented smooth manifold $X$.

	Let $\K' := (G \times G',P, q^*E, q^*s)$ be a group enlargement of $\K$
	where $q: P \to \scrT$ is a $G$-equivariant principal $G'$-bundle for some other Lie group $G'$.
	Then the pushforward maps $H^*(M;\bK) \to H^*(X;\bK)$ constructed using $\K$ and $\K'$ coincide.
\end{lemma}
\begin{proof}
	Since group enlargement commutes with stabilization,
	we can stabilize $\K$ so that it is smooth and that we have choices of orientation
	for $T\scrT - \frg$ and $E$ whose difference recovers the given orientation (see Proposition \ref{prop Kuranishi smoothing} and Lemma \ref{lemma stronger orientation}).
	We can also make sure that $P$ is a smooth principal bundle with $G'$ acting smoothly.
	Since $P/(G\times G') = \scrT/G$, the result follows from Addendum \ref{Add:change_of_group-homological}.
	\end{proof}

\begin{corollary} \label{corollary usual fundamental class}
	Let $\K = (G,\scrT,E,s)$ be an oriented smooth, regular and free Kuranishi chart
	for a compact metric space (manifold!) $M$.
	Let $f : M \to X$ be a smooth map to a $\bK$-oriented manifold.
	Then the pushforward map $f_*^M$ coincides
	with the composition:
	\begin{equation} \label{eqn pushforard smooth}
		H^*(M;\bK) \lra{\lambda_{0,M}^{-1}} H_{\dim(M)-*}(M;\bK) \lra{f_*}
		H_{\dim(M)-*}(X;\bK) \lra{\lambda_{0,X}} H^{*-\dim(M)+\dim(X)}(X;\bK)
	\end{equation}
	where $\lambda_{0,M}$ and $\lambda_{0,X}$
	are Poincar\'{e} duality maps associated to the $\bK$-oriented manifolds $M$
	and $X$ from Theorem \ref{thm:Cheng-homological}.
\end{corollary}
\begin{proof}
	The map (\ref{eqn pushforard smooth}) is equal to the map $f_*^{\K_M}$
	where $\K_M = (0,M,M,0)$.
	Our result now follows from the fact that $\K$
	is obtained from $\K_M$ via group enlargement, stabilization and germ equivalence.
\end{proof}

In this paper, we deal with moduli spaces that are preimages
of submanifolds under an evaluation map.
See, for instance, the moduli space $\ccMbar_h$ in Section \ref{sec:moduli-spaces-j}.
We need a lemma to help us compute what the corresponding pushforward maps
for these moduli spaces are.

\begin{lemma} \label{lemma restricting to submanifold}
	Let $\K = (G,\scrT,E,s)$ be a smooth oriented Kuranishi chart.
Let $f : \scrT \lra{} X$ be a smooth  $G$-equivariant
	submersion.
	Let $S \subset X$ be a smooth $\bK$-oriented submanifold of $X$. 
	Let $\widetilde{S} = \widetilde{f}^{-1}(S)$ and
	consider the Kuranishi chart $\K' = (G,\widetilde{S},E|_{\widetilde{S}},s|_{\widetilde{S}})$ with the evaluation map
	\begin{equation}
		f_S : \widetilde{S} \to S, \quad f_S(x) = f(x).
	\end{equation}
	Then $(f_S)_*^{\K'}(\alpha|_{\widetilde{S}}) = f_*^{\K}(\alpha)|_S$ for each $\alpha \in H^*_G(\scrT;\bK)$.
\end{lemma}
\begin{proof}
	Now let $U$ be  $G$-equivariant tubular open neighborhood of $\widetilde{S}$.
	Our lemma now follows from Addendum \ref{Add:relative_Cheng-homological} applied to $U$
	combined with Lemma \ref{Lem:compatible_virtual_classes} applied to $U$, viewed as the normal bundle of $\widetilde{S}$.
\end{proof}

\subsection{Upshot for moduli of holomorphic curves} \label{section moduli kuranishi}

Fix a closed or geometrically bounded symplectic manifold $(X,\omega)$ with taming almost complex structure $J$.  In Section \ref{Sec:curves}, we will prove that moduli spaces $\ccM_n := \ccMbar_{0,n}(X,\beta;J)$ of stable $J$-holomorphic spheres with $n$ marked points in a fixed  class $\beta \in H_2(X;\bZ)$ admit global Kuranishi charts whose thickening admits a $C^1_{loc}$ $G$-bundle structure with (smooth, complex) Deligne-Mumford type spaces of domains as a base, and which have finite orbit type.
Together with a natural orientation, this gives us our virtual fundamental class by Definition \ref{defn pushforward on oriented chart} above.
  We briefly recall the  construction of Section \ref{Sec:curves}. 
  
Let us first give a sketch of the construction of the global Kuranishi chart
for $\ccM_0 := \ccMbar_{0,0}(X,\beta;J)$ before we add marked points (see Section \ref{Sec:description}).
Let $\Omega \in H^2(X;\bZ)$ be an integral symplectic form taming $J$; if $\omega$ is integral one can take $\Omega=\omega$, in general take a large multiple of a small rational perturbation of $\omega$ (the choice of $\Omega$ is allowed to depend on the choice of homology class $\beta$). 
A \emph{framing} $F$ of a curve $u: \Sigma\to X$ is a basis of the space $H^0(L_u)$ of holomorphic sections of the Hermitian line bundle $L_u$ over $\Sigma$  with curvature $u^*\Omega$. A framing determines a map $\iota_{\calF}: \Sigma \to X\times\calC$ into the universal curve over a smooth complex moduli space of domains, see Section \ref{Sec:description}. The thickening $\scrT$ of the global Kuranishi chart is a space of tuples $(u,\Sigma,F,\eta)$ where $(u,\Sigma,F)$ is a framed curve and $\eta \in \scrE_u$ is a certain auxiliary field for which the equation
\begin{equation}
\overline{\partial}_J(u) + \langle \eta \rangle \circ d\iota_{F} = 0
\end{equation}
makes sense and is transversally cut out. There is a $G$-action, with $G=U(N)$ for $N$ the rank of $H^0(L_u)$, which reparametrizes the framing.  There is a bundle $\scrE   \to \scrT$ with fibre the space of auxiliary fields $\eta$ and a trivial bundle with fibre $\frg$, and a canonical section 
\begin{equation}
\frak{s}: \scrT \to \scrE \oplus\frg \qquad (u,\Sigma,F,\eta) \mapsto (\eta,\scrH(u,\Sigma,F))
\end{equation}
where $\exp\,\scrH(u,\Sigma,F)$ gives a matrix representation of the $L^2$-inner product on $H^0(L_u)$ with respect to the given framing $F$.  (Strictly, $\scrT$ is only a manifold in some open neighbourhood of $\frak{s}^{-1}(0)$, so we are working with the germ of the chart along that locus.)  The moduli space of unmarked curves under consideration is $\frak{s}^{-1}(0)/G$.
Hence $\K_0 = (U(N),\scrT,\scrE,\frs)$
is a Kuranishi chart for $\ccM_0$.

We then construct a global Kuranishi chart for  $\ccM_n = \ccMbar_{0,n}(X,\beta;J)$ as a pullback of the Kuranishi chart for $\ccMbar_{0,0}(X,\beta;J)$; see Section \ref{Sec:marked_points} for more details. We let $d = \beta(\Omega)$ and $\scrF_n(d)$ be the space of stable genus zero degree $d$ curves
in $\C P^d$ with $n$ marked points  whose image is not contained in a hyperplane. (The last condition ensures that $\scrF_n(d)$ is smooth and not just an orbifold.)
There is a natural map $\Pi : \scrT \to \scrF_n(d)$ given by sending $(u,\Sigma,F,\eta)$ to the map $\Sigma \to \C P^d$ induced by the framing (see Equation \eqref{equation F map}).
There is also a natural forgetful map $f_n : \scrF_n(d) \to \scrF_0(d)=: \scrF(d)$.
Then the Kuranishi chart for $\ccM_n$ is $\K_n = f_d^*\K_0$.

\begin{Proposition} \label{Prop:vfcexistence}
	Let $f : \ccM_n \to X$ be a continuous map to a $\bK$-oriented smooth manifold $X$.
	Then there is a natural smooth $\bK$-oriented global Kuranishi chart $\widetilde{\K}_n = (G,\widetilde{\scrT},\widetilde{E},\widetilde{s})$, obtained from $\K_n = (G,\scrT,E,s)$
	via stabilization and germ equivalence, together with a smooth
	$G$-equivariant submersion $\widetilde{f} : \widetilde{\scrT} \to X$ whose restriction to $s^{-1}(0)$
	is the natural composition
	\begin{equation}
		s^{-1}(0) \lra{pr} s^{-1}(0)/G \cong \ccM_n \lra{f} X.
	\end{equation}
	If, in addition, there is a closed subset $K \subset \ccM_n$ consisting
	of regular elements in the sense that the linearized Cauchy-Riemann
	operator is surjective, then we can assume that $K$ is regular in
	$\widetilde{\K}_n$ in the sense of Definition
	\ref{Defn free regular} and that the two smooth structures on $K$
	coming from the two notions of regularity coincide.
	Furthermore, the automorphism group of each element of $\ccM_n$ coincides with the stabilizer
	group of the corresponding $G$-orbit in $\widetilde{\scrT}$.
\end{Proposition}
\begin{proof}
	By Corollary \ref{cor:fibrewise_smooth} and \ref{cor:C1_loc},
	we have that $\Pi$ is a fiberwise smooth $C^1_{loc}$ $G$-bundle as in Definition \ref{defn fiberwise smooth}
	whose vertical tangent bundle $T^{vt}(\scrT)$ admits a stable almost complex structure.
	Also the base $\scrF(d)$ and the fiber $\scrE$ admits a complex structure.
	Hence by Lemma \ref{lemma lift of vertical microbundle}
	and Corollary \ref{corollary product isomorphism}
	we have that $T_\mu \scrT$ is $\bK$-oriented.
	And so $T_\mu \scrT - \scrE \oplus \frg - \frg = T_\mu \scrT - \scrE$
	is $\bK$-oriented making $\K_0$ $\bK$-oriented by Proposition \ref{Prop:coefficients-text}.
	By Lemma \ref{lemma pullback orientation} we get that $\K_n$ is also $\bK$-oriented.
	We also have that the natural map $f_d^*\scrT \to \scrF(d)$ is a fiberwise
	smooth $C^1_{loc}$ $G$-bundle and the $G$-action on $f_d^*\scrT$ has finitely many orbit
	types.
	Hence we have a natural push-forward map and virtual fundamental class associated to $\K_n$ by Definitions \ref{defn pushforward on oriented chart} and \ref{defn:vfc}.
	Our Kuranishi chart $\widetilde{\K}_n$ is constructed using Proposition \ref{prop Kuranishi smoothing} combined with Lemma \ref{Lem:transverse_global_chart}.
\end{proof}

\begin{corollary} \label{corollary pullback}
	With the same notation as in Proposition \ref{Prop:vfcexistence},
	we have that $f^{-1}(S)$ has a natural Kuranishi structure
	$\widetilde{\K}_n|_S := (G,\widetilde{\scrT}|_S,\widetilde{E}|_S,s|_S)$ for any smooth $\bK$-oriented submanifold $S \subset X$.
	The associated pushforward map satisfies
	\begin{equation} \label{eqn moduli restriction}
			(f|_{f^{-1}(S)})^{\widetilde{\K}_n|_S}_*(\alpha|_S) = f^{\widetilde{\K}_n}_*(\alpha)|_S, \quad \forall \ \alpha \in H^*(\ccM_n;\bK)
	\end{equation}
\end{corollary}

\begin{proof} This follows from combining Proposition \ref{Prop:vfcexistence} with Lemma \ref{lemma restricting to submanifold}. \end{proof}

\subsection{Conclusion\label{Sec:conclusions}}

We will use the methods and notation from Section
\ref{Sec:degeneration}.
The global Kuranishi charts $\K_\bullet$
for $\ccMbar_\bullet$, $\bullet = h$ or $\phi$ or $\infty$
are constructed from $\ccMbar_{0,2}(\tilde{P},\beta)$
using Corollary \ref{corollary pullback} and the evaluation map \eqref{eqn evaluationev}.
Let
\begin{equation}
	ev_{h,1} : \ccMbar_h \to (S^2 \times X)_{hor}, \ ev_{h,2} : \ccMbar_h \to \widetilde{P},
\end{equation}
\begin{equation}
	ev_{\bullet,1} : \ccMbar_h \to (S^2 \times X)_{hor}, \ ev_{h,2} : \ccMbar_h \to P_\bullet,
\end{equation}
where $\bullet = \phi$ or $\infty$ be the natural evaluation maps corresponding to each marked point.
We define $\Psi_h$ to be the composition:
\begin{equation}
\begin{tikzcd}
	H^*((S^2 \times X)_{hor};\bK) \arrow[r,"ev_{h,1}^*"] &  H^*(\ccMbar_h;\bK) \arrow[rr,"(ev_{h,2})_*^{\K_h}"] && H^{*-\vdim(\ccMbar_h)+\dim(\tilde{P})}(\tilde{P};\bK)
	\end{tikzcd}
\end{equation}
and we define $\Psi_\bullet$ to be the composition:
\begin{equation}
\begin{tikzcd}
	H^*((S^2 \times X)_{hor};\bK) \arrow[r,"ev_{\bullet,1}^*"] &  H^*(\ccMbar_\bullet;\bK) \arrow[rr,"(ev_{\bullet,2})_*^{\K_\bullet}"] &&  H^{*-\vdim(\ccMbar_\bullet)+\dim(P_\bullet)}(P_\bullet;\bK).
	\end{tikzcd}
\end{equation}

Then by using the diagram \eqref{eqn pushpull compatibility} (which follows from Lemma \ref{lemma restricting to submanifold})
and the same reasoning as given after this diagram, 
we get that \eqref{eqn section}
is a split epimorphism for our fixed cohomology theory $\bK$, which we assumed to be complex oriented and $K(n)$ local. 
Hence by Proposition \ref{Prop:coefficients-text} and Lemma \ref{lemma morava split},
the map \eqref{eqn section} is a split epimorphism over $\Z$,
proving
Theorem \ref{thm:main}.

\subsection{A homological virtual class} \label{Sec:vfc_homological}

We have constructed the virtual fundamental class for a metric space $M$ with a global Kuranishi presentation as an element in the dual of cohomology $H^*(M;\bK) \to \bK_*$, cf. Definition \ref{defn:vfc}.  This fits with the set-up adopted in \cite{pardon2016algebraic}, with the usual picture in which moduli spaces of holomorphic curves act by correspondences on cohomology \cite{Manin}, and circumvents difficulties arising from the potential pathologies of the space $M$ (which would necessitate discussing generalised Steenrod homology \cite{EdwardsHastings}). Nonetheless, for technical reasons it will be useful in Section \ref{Sec:chromatic} to formulate the fundamental class as an element of a homology group, which for us will be the homology of the thickening; since the thickening is a manifold, its generalised homology groups are defined using standard methods of algebraic topology. We briefly indicate the set-up here.
\medskip

Let $\bE$ be a ring spectrum. The associated generalised cohomology theory has cap products 
\begin{equation}\label{eqn:cap}
\begin{tikzcd}
H^i(X,A;\bE) \otimes H_j(X,A;\bE) \arrow[rr, "\frown"] && H_{j-i}(X;\bE)
\end{tikzcd}
\end{equation}
which on representatives $(\alpha: X/A \to \Sigma^i\bE) \in H^i(X,A;\bE)$ and $(\beta: S^j \to X/A \wedge \bE) \in H_j(X,A;\bE)$ is defined by the composite
\begin{equation}
\begin{tikzcd}
\alpha \frown \beta: S^j \arrow[r, "\beta"] & X/A \wedge \bE \arrow[r, "\Delta"] & X/A \wedge X \wedge \bE \arrow[r, "\alpha"] & X \wedge \Sigma^i\bE \wedge \bE \arrow[r, "\mu_\bE"] & X/A \wedge \Sigma^i\bE
\end{tikzcd}
\end{equation}
where $\Delta$ is induced from the diagonal map of $X$ and $\mu_\bE$ is multiplication.  
\medskip

Suppose now that $\bK$ is a complex-oriented $K(n)$-local ring spectrum and $(G,\scrT,E,s)$ is a smooth $\bK$-oriented global Kuranishi chart.  
Up to stabilization and germ equivalence, we can assume that $T_{\mu}\scrT - \frak{g}$ and $E$ are $\bK$-oriented in a way compatible with the given $\bK$-orientation (Lemma \ref{lemma stronger orientation}), and that $\scrT$ is a compact smooth manifold with boundary $\partial \scrT$.  We have the equivariant fundamental class
\begin{equation}
H^0_G(\scrT; \bK) \ni 1 \mapsto [\scrT,\partial \scrT] \in H_{\dim(\scrT)-\dim(G)}^G(\scrT,\partial \scrT;\bK)
\end{equation}
and the orientation $\fro_E \in H^{\rk(E)}_G(E|\scrT;\bK)$. Since by hypothesis $E$ admits a $G$-equivariant section $s$ for which the zero-locus $s^{-1}(0) \subset \mathrm{int}(\scrT)$ is supported in the interior, we have a lift $e_G(E) := s^* \fro_E \in H^*_G(\scrT,\partial \scrT;\bK)$ of the Euler class of $E$, and therefore the cap-product defines an \emph{absolute} homology class,
\begin{equation} \label{eqn:homological_vfc}
e_G(E) \frown [\scrT,\partial \scrT] \, := \, [\vfc] \in H^G_{vd}(\scrT;\bK)
\end{equation}
in degree the virtual dimension $vd = \dim(\scrT) - \dim(G) -\rk(E)$. The virtual classes for different Kuranishi presentations of $M$ are entwined by the operations on global charts in the obvious sense. \medskip

From this starting point, one produces a \emph{homological} splitting of the inclusion map $H_*(X;\bK) \to H_*(P_{\phi};\bK)$, rather than a cohomological splitting of the dual restriction map, as formulated in Theorem \ref{thm:main}.  The argument proceeds as follows: we have as before a global chart $\scrK$ for a space of curves $\ccMbar_\bullet$ with two marked points lying on $\tilde{P}$ and $X\times S^2$; recall that the virtual dimension of the moduli space we consider is $d= 2n+4 = \dim(\tilde{P})$. We consider the composite
\begin{multline} \label{eqn:split by cap product}
  H_i(\tilde{P};\bK) \stackrel{\mathrm{Dual}}{\longrightarrow}   H^{d-i}(\tilde{P};\bK) \stackrel{\mathrm{pull}}{\longrightarrow} H^{d-i}(\tilde{P} \times (X\times S^2);\bK) \\
 \stackrel{{\frown [\vfc]} }{\longrightarrow}  H_i(\tilde{P}\times(X\times S^2);\bK)  \stackrel{\mathrm{project} }{\longrightarrow} H_i(X\times S^2;\bK).
\end{multline}
Dualising the argument from Section \ref{Sec:conclusions}, we see that this splits the inclusion map.

\section{Moduli spaces of genus zero curves\label{Sec:curves}}

Let $(X,\omega)$ be a symplectic manifold and $\beta \in H_2(X;\bZ)$. Fix a compatible almost complex structure $J$ on $X$.  Let $\ccMbar_{0,n}(X;\beta)$ denote the moduli space of stable genus zero holomorphic maps in class $\beta$ with $n$ marked points. The aim of this section is to prove the following two results, used in the construction of the Morava virtual fundamental class.

\begin{Proposition} \label{Prop:global_chart}
For suitable $N, d'$ the space $\ccMbar_{0,n}(X;\beta)$ admits a global Kuranishi chart 
\begin{equation}
\ccMbar_{0,n}(X;\beta) \cong Z/U(N); \qquad Z = \frak{s}^{-1}(0), \ \frak{s}: \scrT \to \scrE
\end{equation}
for which there is a fibrewise smooth topological submersion $\scrT \to \ccMbar_{0,d'}$ to the moduli space of $d'$-pointed rational curves.
\end{Proposition}

If $[\omega]$ has an integral lift, if $d$ denotes the degree $\langle [\omega],\beta\rangle $, we can take $N = d +1$ and  $d' = n + 3 + \dim(\ccMbar_{0,0}(\bC\bP^{d },d))$. The construction of the global chart shares features with an old idea of Siebert, see \cite[Section 3]{Siebert}. 

\begin{Proposition}\label{Prop:global_chart_finite_orbit_type}
For the global Kuranishi chart of Proposition \ref{Prop:global_chart}, the action of $U(N)$ on $\scrT$ has only finitely many orbits.
\end{Proposition}

We will treat the case of curves without marked points, $n=0$, in detail, and briefly explain the adaptations to $n>0$ in Section \ref{Sec:marked_points}.

\subsection{Standing conventions} Until explicitly indicated otherwise, we fix:

\begin{enumerate}
	\item A closed symplectic manifold $(X,\omega)$.
	\item An $\omega$-compatible almost complex structure $J$.
          \item A closed non-degenerate $2$-form $\Omega$ admitting an integral lift, which is tamed by $J$.
	\item A class $\beta \in H_2(X;\Z)$; define $d := \beta(\Omega)$.
\end{enumerate}

\subsection{Description of the chart\label{Sec:description}}

By a {\it smooth function} on a nodal curve, we mean a continuous function  whose pullback to the  normalization is smooth. With this convention,  we can define {\it smooth maps},  {\it smooth 	fiber bundles},{ \it differential forms} etc, on such curves.
Similarly we can define (for instance) $W^{1,p}$ functions on such curves, where $p>2$, or more generally any Sobolev-type class as long as the functions are continuous.
The {\it genus} of a nodal curve will mean its arithmetic genus.

\begin{Definition} \label{defn universal curve}
	Let $\overline{\ccM}_{0,0}(\C \bP^d,d)$
	be the space of stable genus zero connected nodal curves
	$u : \Sigma \lra{} \C \bP^d$
	of degree $d$.
	Let $\scrF := \scrF(d) \subset \overline{\ccM}_{0,0}(\C \bP^d,d)$
	be the Zariski open subset 
	consisting of those curves whose image is not contained in any linear subspace.
	We write \begin{equation}\univ : \scrC \lra{} \scrF\end{equation} for the universal curve over $\scrF$.
\end{Definition}

An irreducible degree $d$ curve in $\bC\bP^d$ is a rational normal curve, and all such are equivalent under the action of $P GL(d+1)$. Reducible limits of rational normal curves which have multiply covered components are always contained in hyperplanes:

\begin{Lemma}\label{Lem:no_automorphisms}
Both $\scrF$ and $\scrC$ are smooth quasi-projective varieties.
\end{Lemma}

\begin{proof} 
According to \cite[Theorem 2]{FultonPandharipande}, the moduli space  $\ccMbar_{0,0}(\bC\bP^d,d)$ is a projective orbifold whose dimension agrees with its virtual dimension (this holds because projective space is `convex').
This theorem also states that the points on $\ccMbar_{0,0}(\bC\bP^d,d)$
corresponding to curves with trivial automorphism groups
have trivial stabilizer groups.
Therefore it suffices to prove that $\scrF$ includes into the locus of stable maps with trivial automorphism group. 

The following reformulation of $\scrF$ will be helpful: fixing the usual homogeneous co-ordinates on $\bC\bP^d$, and sections $s_i$ vanishing on the co-ordinate hyperplanes, note that $\scrF$ can equivalently be viewed as the Zariski open in $\ccMbar_{0,0}(\bC\bP^d,d)$ of nodal curves $u: \Sigma \to \bC\bP^d$ for which $\{u^*s_0,\ldots,u^*s_d\}$ yields a basis of $H^0(u^*\mathcal{O}(1))$.  More abstractly, let $L$ be a holomorphic line bundle over $\Sigma$ of total degree $d$, and for which the degree of $L$ on each unstable component of $\Sigma$ is strictly positive. Let $F = (f_0,f_1,\cdots,f_d)$ be a basis of $H^0(L)$; 
	we define
	\begin{equation} \label{equation F map}
		\phi_F : \Sigma \lra{} \C \bP^d, \quad \phi_F(\sigma) := [f_0(\sigma):f_1(\sigma):\cdots:f_d(\sigma)].\end{equation}
	Since no unstable component is contracted,  this yields a stable map to projective space, which is manifestly not contained in a hyperplane; thus $\phi_F$ defines a point of  $\scrF$. 
	
The group of automorphisms of a curve $u$ is isomorphic to the group of automorphisms of the corresponding triple $(\Sigma,L,F)$, whose elements are given by an automorphism of $\Sigma$ which lifts to $L$ and preserves the basis $F$, hence acts on $H^0(L)$ trivially. The dual graph of $\Sigma$ is a tree; the terminal vertices (leaves) of that tree are unstable components, over which $L$ has strictly positive degree. Over any such component there is  a non-zero section of $L$ which vanishes on all other components, so an automorphism cannot permute leaves. It follows that it fixes each component setwise, and must fix the unstable components pointwise since it acts trivially on the vector space of sections of an ample line bundle over each such component. This completes our proof that the locus $\scrF$ consists of stable maps with trivial automorphisms, hence is a smooth manifold.
\end{proof}

For later use, we state the next result asserting that,  locally,
$\scrF$ looks like the moduli space  $\overline{\ccM}_{0,d'}$
of genus zero curves
with $d' := \dim(\scrF)+3$ marked points:

\begin{prop} \label{prop local coordinate}
	Let $w : \Sigma \lra{} \bC \bP^d$ be an element of $\scrF$.
	Then there are linear hypersurfaces
	$L_1,\cdots,L_{d'}$ in $\bC \bP^d$
	which are transverse to $w$
	so that the following property holds:
	
	Let $\scrF(L_\bullet) \subset \scrF$ be the subset of elements which are transverse to the hypersurfaces $L_\bullet := (L_i)_{i=1}^{d'}$.
	Then the map
	\begin{equation}\mu_{L_{\bullet}} : \scrF(L_\bullet) \lra{} \overline{\ccM}_{0,d'}\end{equation}
	sending $v'$ to the unique marked genus zero curve $(\Sigma,(x_i)_{i=1}^{d'})$
	satisfying $w'(x_i) \in L_i$, $i=1,\cdots,d'$
	is a biholomorphism onto its image.
\end{prop}

\begin{proof} Classical. \end{proof}

\begin{defn} \label{defn domain}
	Let $\Sigma$ be a genus zero nodal curve.
	A  \emph{domain map}
	is an inclusion map
	$\iota : \Sigma \hookrightarrow \scrC
	$
	so that $\iota$ is an isomorphism onto a fibre of the universal curve over $\scrF$.
	\end{defn}
	
	\begin{Example}
	For a triple $(\Sigma, L, F)$ as in the proof of Lemma \ref{Lem:no_automorphisms}, we write 
	\begin{equation}\iota_F : \Sigma \stackrel{\cong}{\longrightarrow} \univ^{-1}(\phi_F) \subset  \scrC\end{equation}
	for the corresponding domain map associated to $F$.
	\end{Example}

\begin{lemma} \label{lemma hermitian bundle}
	Let $\Sigma$ be a genus zero nodal curve and let $\Omega \in \Omega^2(\Sigma)$ be a closed $2$-form representing an integral cohomology class.
	Then there exists a unique Hermitian line bundle on $\Sigma$ up to isomorphism whose curvature is   $-2\pi i\Omega$.
\end{lemma}

\begin{proof}
For existence, let $\ccL$ be any Hermitian line bundle satisfying $c_1(\ccL)=[\Omega]$ and let $|\cdot|$ be the corresponding Hermitian norm.
Let $-2\pi i\Omega_\ccL \in \Omega^2(\Sigma;i\bR)$ be the curvature of $\ccL$ and let $\theta \in \Omega^{0,1}(\Sigma;\bC)$ satisfy $\partial\theta = -2\pi i(\Omega - \Omega_\ccL)$.
Now since $\overline{\partial} \theta = 0$
and since $\Sigma$ is simply connected, we have that $\theta = \overline{\partial} h$ for some function $h : \Sigma \to \bC$.
Let $f : \Sigma \to \bR$ be equal to $i \Im(h)$.
Since $\Omega$ and $\Omega_\ccL$ are elements of $\Omega^2(\Sigma;\bR)$,
$\partial \overline{\partial} f = -2\pi i(\Omega - \Omega_\ccL)$.
The Hermitian bundle with Hermitian norm $e^{-f/2} |\cdot|$
on $\ccL$ has curvature $\Omega$.

For uniqueness, suppose we have two Hermitian bundles $L_0$, $L_1$
both with curvature $\Omega$.
Then $L := L_0 \otimes L_1^{-1}$
is a flat Hermitian bundle.
Let $|\cdot|$ be the corresponding Hermitian norm on $L$.
Since $\Sigma$
is of genus zero, we have that $L$ is a trivial holomorphic bundle
and so admits a holomorphic section $s$ which is nowhere zero.
Hence $\overline{\partial} \partial \log(|s|) = 0$ and so $\log(|s|)$ is a harmonic function.
By the maximum principle, this harmonic function has to be constant.
Hence $L$ is isomorphic as a Hermitian bundle to a trivial bundle with the trivial metric,
and therefore $L_1$ is isomorphic to $L_2$.
\end{proof}

\begin{defn} \label{DEFN unique bundle}
	For any nodal connected genus zero curve $\Sigma$ and for any $2$-form $\Omega$ on $\Sigma$ admitting an integral lift, we define $L_\Omega$ to be a Hermitian line bundle whose curvature is $-2\pi i \Omega$. We write
	$\left<\_ ,\_ \right>_\Omega$ for its associated Hermitian metric.
\end{defn}

We now return to consider a symplectic manifold $(X,\omega)$, with an almost complex structure $J$, and an integral symplectic form $\Omega$ which is tamed by $J$. 

\begin{defn}\label{defn:framed_curve}
	A \emph{framed genus $0$ curve in $X$}
	is a tuple $(u,\Sigma,F)$ where
	\begin{enumerate}
		\item $\Sigma$ is a genus zero nodal curve,
		\item $u : \Sigma \lra{} X$ is a smooth map representing $\beta$ so that the degree of $L_{u^*\Omega}$ is strictly positive on each unstable component of $\Sigma$ and
		\item $F = (f_0,\cdots,f_d)$
		is a basis for $H^0(L_{u^*\Omega})$
		so that the Hermitian matrix
		\begin{equation} \label{eqn mx}
		\scrH(u,\Sigma,F) := \left(\int_\Sigma \left< f_i,f_j \right>_{u^*\Omega} u^*\Omega \right)_{i,j=0,\ldots,d}
		\end{equation}
		has positive eigenvalues.
	\end{enumerate}
\end{defn}

Note that, at this stage, we do not impose any holomorphicity assumption on the map $u : \Sigma \to X$. We may nonetheless define an equivalence relation on framed curves in the usual way, by considering biholomorphisms between the domains which intertwine all the data. The set of equivalence classes of framed curves can be identified with the set consisting of a point in $\scrF$ and a map from the corresponding fibre of the universal curve $\scrC$ to $X$, and is equipped with a natural topology from the Hausdorff metric on the space of graphs of the maps, considered as closed subsets  of the product $\scrC \times X$. This topology is separated because the map from each fibre to $\scrC$ is a closed embedding, so that no component of the domain is collapsed in the graph.
Also note that we could have required the framing $F$
to be orthonormal instead.
However, the condition \eqref{eqn mx} is useful when we describe our moduli spaces using the `Gromov trick' (Section \ref{Sec:Gromov_trick_1}).
\medskip

\noindent We next choose:
\begin{enumerate}
	\item A Hermitian line bundle $\scrL \lra{} \scrC$ whose restriction to each fiber is ample and
	with the property that the $U(d+1)$ action on $\scrC$ induced by the corresponding linear action on $\C P^d$ lifts to a $U(d+1)$ action on the total space of $\scrL$ preserving the Hermitian structure.
	\item A large integer $k \gg 1$.
\end{enumerate}

\begin{defn} \label{defn large space}
	The \emph{thickened moduli space} $\scrT$
	is the moduli space of tuples
	$(u,\Sigma,F,\eta)$
	where
	\begin{enumerate}
		\item $(u,\Sigma,F)$ is a framed curve and
		\item $\eta \in H^0(\overline{\Hom}(\iota_F^* T\scrC, u^*TX) \otimes \iota_F^*\scrL^k) \otimes_\C \overline{H^0(\iota_F^*\scrL^k)}$
	\end{enumerate}
	satisfying
	\begin{equation} \label{eqn afwaa}
	\overline{\partial}_J u + \langle \eta \rangle \circ d\iota_F = 0
	\end{equation}
	where $\langle \eta \rangle$ is the image of $\eta$ under the natural Hermitian pairing
	\begin{equation} \label{eqn hermitian product map}
		H^0(\overline{\Hom}(\iota_F^* T\scrC, u^*TX) \otimes \iota_F^*\scrL^k) \otimes_\C \overline{H^0(\iota_F^*\scrL^k)} \to C^\infty(\overline{\Hom}(\iota_F^* T\scrC, u^*TX)).
	\end{equation}
\end{defn}

\begin{rem}
To be completely precise, the space $\scrT$ defined above will turn out to only be a manifold in a neighbourhood of the zero locus $\frs^{-1}(0)$ of a section introduced below, and it will be equipped with a fibrewise smooth structure over $\scrF$ after further shrinking this neighbourhood,  as discussed following Corollary \ref{cor:fibrewise_smooth}. Since all our arguments are local to the zero locus, this point does not affect any construction.
\end{rem}

Let $\scrH_n$ be the space of $n \times n$ Hermitian matrices and $\scrH_n^+ \subset \scrH_n$ those Hermitian matrices with positive eigenvalues.
	Let $\exp : \scrH_n \lra{\cong} \scrH_n^+$  be the exponential map and $\exp^{-1}$ its inverse.

\begin{defn} \label{defn moduli space kuranishi}
	We define $\scrE$ to be the bundle over $\scrT$ whose fiber over $(u,\Sigma,F,\eta)$
	is
	\begin{equation}H^0(\overline{\Hom}(\iota_F^* T\scrC, u^*TX) \otimes \iota_F^*\scrL^k) \otimes_\C \overline{H^0(\iota_F^*\scrL^k)} \, \oplus \scrH_{d+1}.\end{equation}
	This has a canonical section $\frs : \scrT \lra{} \scrE$
	sending $(u,\Sigma,F,\eta)$ to $(\eta,\exp^{-1}(\scrH(u,\Sigma,F)))$.
\end{defn}

Note that $\scrE \to \scrT$ is a bundle since it is the pullback
of a bundle over $X \times \scrC$
(see Section \ref{section gro trick}).
There are natural $U(d+1)$ actions on $\scrT$ and $\scrE$, given by changing the basis $F$, with respect to which the section $\frs$ is equivariant.
The moduli space $\overline{\ccM}_{0,0}(X,\beta)$
is equal to $\frs^{-1}(0) / U(d+1)$.

\begin{Lemma} The induced topology on $\ccMbar_{0,0}(X,\beta)$ agrees with the usual Gromov topology.
\end{Lemma}

\begin{proof}
Consider a sequence of genus zero holomorphic maps $u_i$ to an almost complex Riemannian manifold $(M,J,g)$, whose images converge with respect to the Hausdorff distance on closed subsets of $M$  to an injective holomorphic map $u_\infty$ and whose energies converge to the energy of $u_\infty$. After passing to a subsequence, the maps $u_i$ Gromov converge to a stable map $v_\infty$. The image of $v_\infty$ coincides with the image of $u_\infty$, and $u_\infty$ and $v_\infty$ also have equal energy.  Hence $v_\infty$  is injective: multiply covered components are ruled out by energy considerations, whilst ghost bubbles would be unstable due to the injectivity of $u_\infty$. It follows that $v_\infty$ is equivalent in the Gromov topology to $u_\infty$.  We apply this to the graphs of domain maps in $M=\scrC \times X$.
\end{proof}

\begin{Proposition}\label{Prop:finite_orbit_type} The $U(d+1)$ action on $\scrT$ has finite stabilisers.  Moreover, the stabilisers have bounded size, so the action has finitely many orbit types.
\end{Proposition}

\begin{proof}
The fact that $L_{u^*\Omega} \to \Sigma$ is ample for each unstable component of the domain curve $\Sigma$ implies finiteness of the automorphism group of the pair $(u,F)$. The bound on the size of the stabilisers follows from Gromov compactness combined with the fact that $\eta$ is small and finiteness of the number of domain types for curves in the moduli space (see Lemma \ref{lem:gromov_revisited}). 
\end{proof}

\begin{rem} When $J$ is integrable, the thickening $\scrT$ and obstruction space $\scrE$ are both naturally analytic spaces, but the section $\frak{s}$ is far from being analytic.   Siebert \cite{Siebert} also considers Kuranishi charts involving sections of powers of a bundle which is fibrewise ample on the universal curve.  Honda and Bao \cite{Honda-Bao} introduce a thickening of a moduli space of punctured curves in which the auxiliary fields are controlled by finite-dimensional sums of eigenspaces of an asymptotic Laplace-type operator associated to periodic orbits at the punctures.
\end{rem}

\subsection{The `Gromov trick' and shearing almost complex structures\label{Sec:Gromov_trick_1}}

In order to take advantage of existing gluing and transversality results, it will be helpful to recast the moduli spaces of perturbed holomorphic curves appearing in the global Kuranishi chart in terms of unperturbed holomorphic curves.  The `Gromov graph trick' is a standard such reduction. The next three subsections discuss general theory, which we then bring to bear on our particular case at the end of the section.

\medskip
\noindent \textbf{Set-up:} Let $\pi : \check{E} \lra{} \check{X}$ be a Hermitian vector bundle over an almost complex manifold $(\check{X},\check{J})$ with Hermitian connection $\nabla$.
	Let $T^v\check{E} := \ker \,d\pi$  be the vertical tangent bundle of $\check{E}$.
	If $C^\infty(\check{E})$ denotes the space of smooth sections of $\check{E}$, set
	\begin{equation}T^h\check{E} := \{ds(v) \ : \ v \in T\check{X}, \ s \in C^\infty(\check{E}), \ \nabla_v s= 0 \}\end{equation}
	to be the corresponding \emph{horizontal tangent bundle}
	(i.e. the Ehresmann connection corresponding to the Hermitian structure).

There is a natural lift $\widetilde{J}$ of $\check{J}$ to the total space of $\check{E}$, characterised by the following three properties:
\begin{itemize}
	\item $T^v\check{E}$ and $T^h\check{E}$ are $\widetilde{J}$-holomorphic subbundles of $(T\check{E},\widetilde{J})$.
	\item $\pi$ is $\widetilde{J}$-holomorphic.
	\item The restriction of $\widetilde{J}$ to the fibres of $\check{E}$ coincides with their natural complex structure.
\end{itemize}

We will be interested in certain `shears' of this natural lift.

\begin{defn} \label{defn shear}
		Given a real linear bundle map  $\check{\Psi} : \check{E} \oplus T\check{X} \lra{} T\check{X}$  over $\check{X}$,
		we define the
 \emph{$\check{\Psi}$-shear}
of $\widetilde{J}$
to be the almost complex structure
\begin{equation}\widetilde{J}_{\check{\Psi}} := \Phi \circ \widetilde{J} \circ \Phi^{-1}\end{equation}
where, using the natural identification $T^h\check{E} = \pi^*T\check{X}$,
$\Phi$ is given at a point $e \in \check{E} $ as follows:
\begin{align}
  \Phi : T^v\check{E} \oplus T^h\check{E} & \lra{} T^v\check{E} \oplus T^h\check{E} \\
  \Phi(v,h) & := (v,h+\check{\Psi}(e,h)). 
\end{align}
\end{defn}

If $\check{\Psi}$ satisfies
\begin{equation} \label{eqn nilpotency condition}
	\check{J}(\check{\Psi}(e,h)) = -\check{\Psi}(e,\check{J}(h)), \textrm{ and }
	\check{\Psi}(e,\check{\Psi}(e,h)) = 0, 
\end{equation}
then we have the following  explicit formula for $\widetilde{J}_{\check{\Psi}}$:
\begin{equation} \label{equation shear eqn}
\widetilde{J}_{\check{\Psi}}(v,h)
= (\widetilde{J}(v),\widetilde{J}(h)+ 2\check{\Psi}(e,\widetilde{J}(h))). 
\end{equation}

\medskip
The importance of the shear is explained by the following Lemma. Let $\Sigma$ be a nodal curve with complex structure $j$.
Suppose that we have a smooth map
$u : \Sigma \lra{} \check{E}.$
 View the map $u$ as a pair  $u = (\upsilon,s)$
where $\upsilon : \Sigma \lra{} \check{X}$
is a smooth map and $s \in C^\infty(\upsilon^*\check{E})$
is a smooth section.

\begin{lemma} \label{lemma shear} {\bf (Gromov trick.)}
	If $\check{\Psi}$ satisfies Equation (\ref{eqn nilpotency condition}) then
	the map
	$u$ is $\widetilde{J}_{\check{\Psi}}$-holomorphic if and only if $s \in H^0(\upsilon^*\check{E})$ is holomorphic, 
	and
\begin{equation} \label{eqn aef}
	\overline{\partial}_{\check{J}} \upsilon - \check{\Psi}(s,\check{J} \circ d\upsilon \circ j) = 0.
\end{equation}
\end{lemma}
\begin{proof}
For $\xi \in T_\sigma \Sigma$, $\sigma \in \Sigma$,
we have
\begin{equation}du(\xi) = \nabla_\xi s \oplus
(d\upsilon(\xi))\end{equation}
after identifying $T_{\upsilon(\sigma)}\check{X}$
with $T^h_{u(\sigma)}\check{E}$
via the projection map.
Hence
\begin{equation}\overline{\partial}_{\widetilde{J}}u = \nabla^{0,1}_\xi s \oplus \frac{1}{2}\left(d\upsilon(\xi) + \check{J}(d\upsilon(j(\xi)))\right).
\end{equation}
Therefore by Equation (\ref{equation shear eqn}),
\begin{equation}\overline{\partial}_{\widetilde{J}_{\check{\Psi}}}u
= \nabla^{0,1}_\xi s \oplus
\frac{1}{2}\left(
dv(\xi) + \check{J}\left(d\upsilon(j(\xi)) + 2\check{\Psi}(s(\sigma),d\upsilon(j(\xi))\right)
\right)
\end{equation}
giving us the result.
\end{proof}

\subsection{Linearization of $\overline{\partial}$ for sheared almost complex structures. \label{Sec:linearization}}

We continue with the general set-up from Section \ref{Sec:Gromov_trick_1}. That is, 
$\pi : \check{E} \lra{} \check{X}$ is a Hermitian vector bundle over an almost complex manifold $(\check{X},\check{J})$ with Hermitian connection $\nabla$. 
Let $\check{\Psi} : \check{E} \oplus T\check{X} \lra{} T\check{X}$
be a real linear bundle map over $\check{X}$ satisfying Equation (\ref{eqn nilpotency condition}) and let $\widetilde{J}_{\check{\Psi}}$
be the $\check{\Psi}$-shear introduced in Definition 
\ref{defn shear}.

Suppose $u : \Sigma \lra{} \check{E}$ is a $\widetilde{J}_{\check{\Psi}}$-holomorphic curve whose image is contained in $\check{X}$. We may then write  $u = \iota_{\check{X}} \circ \upsilon$
where $\upsilon : \Sigma \lra{} \check{X}$ and where $\iota_{\check{X}} : \check{X} \lra{} \check{E}$ is the zero section.

We wish to describe the linearization of the $\overline{\partial}_{\widetilde{J}_{\check{\Psi}}}$ operator of this map.
We let $\widetilde{\Sigma}$ be the normalization of $\Sigma$ and let $\widetilde{u} : \widetilde{\Sigma} \lra{} \check{E}$ 
be the composition of $u$ with the natural map $\widetilde{\Sigma} \lra{} \Sigma$.
Similarly we define
$\widetilde{\upsilon} : \widetilde{\Sigma} \lra{} \check{X}$
to be the unique map satisfying $\iota_{\check{X}} \circ \widetilde{\upsilon} = \widetilde{u}$.

Let \begin{equation}\nabla^{0,1}  : C^\infty(\upsilon^*\check{E}) \lra{}
C^\infty(\overline{\Hom}(T\widetilde{\Sigma} \otimes \widetilde{\upsilon}^*\check{E}))
\end{equation}
be the corresponding $(0,1)$ part of $\nabla$.
We will use the same symbol $\nabla$ for the Levi-Civita connection on $T\check{X}$ and $T\check{E}$.

\begin{lemma} \label{lemma linearization}
	The linearization
	\begin{equation}D \overline{\partial}_{\widetilde{J}_{\check{\Psi}}} : C^\infty(u^*T\check{E}) \lra{} C^\infty(\overline{\Hom}(T\widetilde{\Sigma},\widetilde{u}^*T\check{E}))\end{equation}
	satisfies
	\begin{equation}D \overline{\partial}_{\widetilde{J}_{\check{\Psi}}}(\frs)
	= \nabla^{0,1}\frs \oplus 0 + 0 \oplus \check{\Psi}(\frs,d\upsilon \circ j)\end{equation}
	for each smooth section $\frs$ of $\upsilon^*\check{E} \subset u^*T\check{E}$.
\end{lemma}
\begin{proof}
From \cite{McDuff-Salamon, Wendl},  we have
\begin{multline} D\overline{\partial}_{\widetilde{J}_\Psi}(\frs) = \frac{1}{2} \left(\nabla \frs \oplus 0 + \widetilde{J}_{\check{\Psi}} \circ ((\nabla \frs \circ j) \oplus 0) + (\nabla_\frs \widetilde{J}_{\check{\Psi}}) \circ du \circ j\right) \\
  = \nabla^{0,1}\frs \oplus 0 +  \frac{1}{2} \left((\nabla_\frs \widetilde{J}_{\check{\Psi}}) \circ du \circ j\right).\end{multline}
Hence by Equation (\ref{equation shear eqn}):
\begin{equation}D\overline{\partial}_{\widetilde{J}_\Psi}(\frs) \  =  \  \nabla^{0,1}\frs \oplus 0 + 0 \oplus \check{\Psi}(\frs, du \circ j).\end{equation}
\end{proof}

Using the fact that the higher cohomology groups of any ample vector bundle on a genus zero nodal curve vanish, the above result implies:
\begin{corollary} \label{corollary surject}
	The linearization $D\overline{\partial}_{\widetilde{J}_{\check{\Psi}}}$
	is surjective if
	$\upsilon^*\check{E}$ is an ample vector bundle over $\Sigma$
	and if
	for each nontrivial element $\fre$ in the kernel of the formal adjoint operator
	\begin{equation}D^* \overline{\partial}_J : C^\infty(\upsilon^*T\check{X}) \lra{} C^\infty(\overline{\Hom}(T\widetilde{\Sigma},\widetilde{\upsilon}^*T\check{X}))
	\end{equation}
	associated to $J$,
	we can find $\xi \in C^\infty(T\Sigma)$ and
		$\frs \in H^0(\check{E})$
	so that 
	\begin{equation}\left< \check{\Psi}(\frs,d\upsilon( j(\xi))),\fre \right>_{L^2} \neq 0.\end{equation}
	Here we have identified $\fre$ with its pullback to a section of $\widetilde{\upsilon}^*T\check{X}$. \qed
\end{corollary}

\subsection{H\"{o}rmanders theorem and peak sections \label{Sec:Hormander}} We next construct holomorphic sections of bundles over nodal curves which vanish at the nodes (and also at any collection of marked points). This is helpful for achieving transversality of the $\cdbar$-operator at the nodes and marked points.

For the general discussion, we will fix
\begin{itemize}
	\item a compact K\"{a}hler manifold
	 $(Y,J_Y,\widehat{\omega})$ without boundary of dimension $m$,
	 \item a Hermitian vector bundle $\widehat{E}$ and
	 \item an ample line bundle $L$ over $Y$.
\end{itemize}
(We are primarily interested in the case when $Y = \bC\bP^1$, but the results seem relevant at least for  moduli problems involving higher genus Riemann surfaces.)

Let $\Herm(\widehat{E},\widehat{E})$
be the bundle of Hermitian endomorphisms of $\widehat{E}$.
For each $p,q \in \bN$, we define
\begin{equation}
	\Lambda_{\widehat{\omega}} :\bigwedge^{p,q}T^*Y \otimes \Herm(\widehat{E},\widehat{E}) \lra{} \bigwedge^{p-1,q-1}T^*Y \otimes \Herm(\widehat{E},\widehat{E})
\end{equation}
to be the metric adjoint of the map $\widehat{\omega} \wedge \cdot$.
In other words, the vector bundle morphism satisfying:
\begin{equation}
	\langle \alpha, \Lambda_{\widehat{\omega}}(\beta) \rangle = \langle \widehat{\omega} \wedge \alpha, \beta \rangle
\end{equation}
for each $\alpha \in \bigwedge^{p-1,q-1}T^*Y \otimes \Herm(\widehat{E},\widehat{E})|_y$ and
each $\beta \in \bigwedge^{p,q}T^*Y \otimes \Herm(\widehat{E},\widehat{E})|_y$, $y \in Y$.
Let $R_{\widehat{E}}$ be the associated curvature operator, where
\begin{equation}iR_{\widehat{E}} \in C^\infty(\bigwedge^2T^*Y \otimes \Herm(\widehat{E},\widehat{E})).\end{equation}
We will think of this as a map:
\begin{align}
	& iR_{\widehat{E}} : \bigwedge^{p,q}T^*Y \otimes \Herm(\widehat{E},\widehat{E}) \lra{} \bigwedge^{p+1,q+1}T^*Y \otimes \Herm(\widehat{E},\widehat{E}), \\
	& iR_{\widehat{E}}(\alpha) = iR_{\widehat{E}} \wedge \alpha
\end{align}
for each $p,q \in \bN$.

\begin{Theorem} {\bf (H\"{o}rmander), \cite[Theorem 8.4]{demailly1996l2}} \label{Thm:Hormander}
	Let $p,q \in \bN$ and suppose that the commutator:
	\begin{equation}
	A := [i R_{\widehat{E}},\Lambda_{\widehat{\omega}}] : \bigwedge^{p,q}T^*Y \otimes \Herm(\widehat{E},\widehat{E}) \lra{} \bigwedge^{p,q}T^*Y \otimes \Herm(\widehat{E},\widehat{E})
	\end{equation}
	is positive definite on each fiber.
	Let $c_{\widehat{E}}$
	be the $L^\infty$ norm of $A$ and define $c := c_{\widehat{E}}^{-1}$.
	Then for each $g \in L^2(\Lambda^{p,q}T^*Y \otimes \widehat{E})$
	satisfying $\overline{\partial} g = 0$,
	there exists $f \in L^2(\Lambda^{p,q-1}T^*Y \otimes \widehat{E})$ satisfying $\|f\|_{L^2} \leq  c\|g\|_{L^2}$
	and $\overline{\partial}f = g$.
\end{Theorem}

We now fix in addition  an effective divisor $D$ on $Y$, and a Hermitian metric on ${\mathcal O}_{Y}(D)$ with respect to which
the commutator $[i R_{\widehat{E} \otimes {\mathcal O}_{Y}(D)},\Lambda_{\widehat{\omega}}]$ on $(p,q)$-forms
is positive definite. 

\begin{corollary} \label{corollary vanishing on D}
	If $g = 0$ in a small neighborhood of $D$, and we set $c = c_{{\widehat{E} \otimes {\mathcal O}_{Y}(D)}}^{-1}$, 
	then we can find $f$ in Theorem \ref{Thm:Hormander} with $f|_D = 0$.
	\end{corollary}
\begin{proof}
This follows directly from the fact that smooth sections of $\Lambda^{p',q'}T^*Y \otimes \widehat{E} \otimes {\mathcal O}_{Y}(D)$ naturally map to smooth sections of $\Lambda^{p',q'}T^*Y \otimes \widehat{E}$ which vanish along $D$ for each $p',q' \in \bN$.
\end{proof}

\begin{Remark} \label{rmk:linetensor}  Eventually we will apply Corollary \ref{corollary vanishing on D} with $\widehat{E}$ replaced by 
\begin{equation} \label{eqn E prime}
\widehat{E}'_k := (\Lambda^{m,0} T^*Y)^{-1} \otimes \widehat{E} \otimes L^k
\end{equation}
 for some large $k$ where $L$ is an ample Hermitian line bundle on $Y$ with curvature $-2 \pi i \widehat{\omega}$. In this case, $[iR_{\widehat{E}'_k},\Lambda_{\widehat{\omega}}]$ is positive on $\bigwedge^{m,1} T^*Y \wedge \Herm(\widehat{E},\widehat{E})$ 
  for all sufficiently large $k$ since $iR_{\widehat{E}'_k}/k$ converges to a positive multiple of $\widehat{\omega}$
  and $[\widehat{\omega},\Lambda_{\widehat{\omega}}]$ is the identity on $\bigwedge^{m,1} T^*Y \wedge \Herm(\widehat{E},\widehat{E})$ via a local calculation in an orthonormal frame.
  The same calculation also ensures that the constant $c$ in  Theorem \ref{Thm:Hormander} tends to zero as $k \to \infty$. \end{Remark}

\begin{lemma} \label{lemma peak sections}
Fix an effective divisor $D\subset Y$ and a point $x\in Y \backslash D$. 
	Let $e \in \widehat{E}|_x$, and let $\delta_e$ be the Dirac delta section at $x$ with value $e$.
	There are holomorphic sections $s_k$ and $\check{s}_k$  of $\widehat{E} \otimes L^k$
	and $L^k$ respectively, for $k\in \bN$, 
	so that
	\begin{itemize}
	\item 	$\langle s_k, \check{s}_k \rangle \to \delta_e$
	in the sense of distributions 
	as $k \to \infty$; 
	\item $s_k$ and $\check{s}_k$ vanish along $D$. 
		\end{itemize}
\end{lemma}
\begin{proof}[Proof.]
By \cite[Lemma 1.2]{tian1990set}
we can construct a sequence of holomorphic sections $(\check{s}_k)_{k \in \bN}$ of $L^k$ whose norm converges to the Dirac delta function at $x$.
Now let $\sigma \in C^\infty(\widehat{E})$ satisfy
$\sigma(x) = e$ and choose $\sigma$ so that it is holomorphic in a small neighborhood of $x$.
We define $s'_k := \sigma \otimes \check{s}_k$.
By Theorem \ref{Thm:Hormander} and Remark \ref{rmk:linetensor} we can find sections $(g_k)_{k \in \bN}$ of
\begin{equation}
(\widehat{E} \otimes L^k)_{k \in \bN}=(\Lambda^{m,0}T^*Y \otimes \widehat{E}'_k)_{k \in \bN}
\end{equation}
(Equation \eqref{eqn E prime})
whose $L^2$ norms go to zero as $k \to \infty$ and which satisfy $\overline{\partial}g_k = \overline{\partial} s'_k$
in $\Lambda^{m,1}T^*Y \otimes \widehat{E}'_k$ for each $k \in \bN$.
We define $s_k := s'_k - g_k$ for each $k$.
Then $\langle s_k,\check{s}_k \rangle$, $k \in \bN$ converges in the sense of distributions to the limit of $\langle s'_k, \check{s}'_k \rangle$ as $k \to \infty$, which is $\delta_e$.
By replacing Theorem \ref{Thm:Hormander}
with Corollary \ref{corollary vanishing on D} in this proof and making sure that $\sigma = 0$ near $D$, we can ensure that $s_k$ and $\check{s}_k$ vanish along $D$ for each $k$.
\end{proof}

\subsection{The `Gromov trick' for the thickening.} \label{section gro trick}

We now return to the proof that the thickening $\scrT$, which we constructed in Section \ref{Sec:description}, is smooth near the zero locus of the defining section $\frs$ of the moduli space of genus $0$ holomorphic curves. To this end, we will use Lemma \ref{lemma shear} for a specific Hermitian vector bundle $E$ to describe $\scrT$ in terms of holomorphic maps into $E$.
Thus, we aim to set things up so that Equation (\ref{eqn aef})
is identified with Equation (\ref{eqn afwaa}).

To start, we consider the natural projection maps
\begin{equation}
\xymatrix{
& X\times \scrC \ar[dl]_{p_X}  \ar[dr]^{p_{\scrC}} & \\ X & & \scrC
}.
\end{equation}

We introduce two Hermitian bundles
\begin{align}
  E_0&  := \overline{\Hom}\left(p_\scrC^*T\scrC, p_X^* TX\right) \otimes p_\scrC^* \scrL^k \\
   E_1 & := p_\scrC^* \univ^* \univ_* \scrL^k
\end{align}
over the product $ X \times \scrC$ (Definition \ref{defn universal curve}), and define $E$ to be the tensor product
\begin{equation}
    E = E_0 \otimes_{\bC} \overline{E_1}
\end{equation}
where $\overline{E_1}$ is the complex conjugate bundle.
We consider the shearing map
	\begin{align}\Psi : E \oplus p_X^*TX \oplus p_\scrC^*T\scrC & \lra{} p_X^*TX \oplus p_\scrC^*T\scrC \\
(\eta,v_1,v_2) & \mapsto
	\langle \eta \rangle(v_2) \oplus 0 \end{align}
	where $\langle \eta \rangle$ is the image of $\eta$
	under the natural Hermitian product pairing
	\begin{equation} 
		E_0 \otimes \overline{E_1} \lra{} E_0 \otimes \overline{p_\scrC^*\scrL^k} \lra{} \overline{\Hom}\left(p_\scrC^*T\scrC, p_X^* TX\right).
	\end{equation}

	Let $\widetilde{J}_\Psi$ be the $\Psi$-shear of the $E$-lift of $J = J_{X} \oplus J_{\scrC}$ in the sense of Definition \ref{defn shear}.  
	
	It may help to summarise the dictionary relating these definitions to the notation appearing in our previous general considerations:
\begin{itemize}
\item When comparing to Section \ref{Sec:linearization}, the Hermitian bundle $\check{E} \to \check{X}$ which carries the `sheared' almost complex structure will now be taken to be: $\check{X} = X\times \scrC$; $\check{E} = E = E_0 \otimes_{\bC} \bar{E_1}$;
\item When  comparing to Section \ref{Sec:Hormander},  we take $Y$ to be the normalization $\tilde{\Sigma}$ of the domain $\Sigma$ of a holomorphic curve $u$; the ample bundle $L$ to be the pullback of $p_{\scrC}^*\scrL$ (restricted to $\Sigma$) to $\tilde{\Sigma}$; and $\hat{E}$ the bundle for which $E_0 = \hat{E} \otimes L^k$  on $\tilde{\Sigma}$. 
\end{itemize}

For the next statement, we let $\overline{\ccM}_{0,0}(\widetilde{J}_\Psi)$
	denote the space of stable $\widetilde{J}_\Psi$-holomorphic genus zero connected nodal curves whose projection to $X$ represents $\beta$ and whose projection to $\scrC$ is a domain map as in Definition \ref{defn domain}.
\begin{lemma} {\bf (Gromov trick revisited)} \label{lem:gromov_revisited}
	
	The map sending an element 
	$(u,\Sigma,F,\eta)$
	of $\scrT$,
	to the map
	\begin{equation}(u \times \iota_F, \eta) : \Sigma \lra{} E
	\end{equation}
	is an open embedding 
	$\scrT \hookrightarrow \overline{\ccM}_{0,0}(\widetilde{J}_\Psi)$.
\end{lemma}
\begin{proof}
The map $\Psi$ satisfies Equation (\ref{eqn nilpotency condition})
where $\check{\Psi}$, $\check{E}$ and $\check{X}$ are replaced by $\Psi$, $E$ and $X \times \scrC$ respectively.
The result then follows from Lemma \ref{lemma shear}.
\end{proof}

\subsection{Surjectivity of $\overline{\partial}_{J_\Psi}$\label{Subsec:surjectivity}}

We continue with the notation of Section \ref{section gro trick}.
We wish to use Corollary \ref{corollary surject}
and Lemma \ref{lemma peak sections}
to show that $D\overline{\partial}_{\widetilde{J}_\Psi}$ is surjective, near the zero locus of $\frs$, for sufficiently large $k$.

To this end, we consider a $J$-holomorphic map $u : \Sigma \lra{} X$, and let
$\iota_F : \Sigma \lra{} \scrC$ be the domain map associated to some $L^2$-orthogonal basis $F$ of $H^0(L_{u^*\Omega})$ as in Definition \ref{defn domain}. Composing with the inclusion of $X$ as the zero section of $E$, we obtain a $\widetilde{J}_\Psi$-holomorphic map
\begin{equation}\widehat{u} : \Sigma \lra{} E, \quad \widehat{u}(\sigma) := (u(\sigma),\iota_F(\sigma)) \in X \times \scrC \subset E.\end{equation}
Recall that we denote by $\widetilde{\Sigma}$ the normalization of $\Sigma$ and by
	$\widetilde{u} : \widetilde{\Sigma} \lra{} E$
	 composition of the normalization map with $u$.
\begin{prop} \label{prop surj}
		The linearized operator
	\begin{equation}D \overline{\partial}_{\widetilde{J}_\Psi} : C^\infty(u^* T(X \times \scrC)) \lra{} C^\infty(\overline{\Hom}(T\widetilde{\Sigma},\widetilde{u}^*T(X \times \scrC)))\end{equation}
	is surjective for $k$ sufficiently large.
\end{prop}
\begin{proof}[Proof sketch.]
Let $\fre \in C^\infty(\widetilde{u}^*T(X \times \scrC))$ be a non-trivial element in the kernel of the formal adjoint operator.
The regularity of the moduli space of genus $0$ curves in projective space implies that this element of the kernel lies in the subspace $\fre \in C^\infty(\widetilde{u}^*TX)$.

Now let $\sigma \in \Sigma$ be a smooth point of $\Sigma$,
$w \in T_\sigma \Sigma$
and let
$f \in \overline{\Hom}(T_{\iota_F(\sigma)}\scrC,T_{u(\sigma)}X)$
satisfy
\begin{equation} \label{eqn pairing}
\langle f(\iota_F(j(w))),\fre(\sigma) \rangle \neq 0.
\end{equation}
Choose a smooth section $\xi \in C^\infty(T\Sigma)$
satisfying $\xi(\sigma) = w$.
By Lemma \ref{lemma peak sections},
we can find holomorphic sections $s_k$ and $\check{s}_k$ of $E_0$ and $p_\scrC^* \scrL^k$ respectively which vanish at the nodes of $\Sigma$ and for which $\langle s_k,\check{s}_k \rangle$
converges to the Dirac delta section
$\delta_f$ of $\overline{\Hom}(\iota_F^*T\scrC,u^*TX)$ as $k \to \infty$.
Hence
$\left<\langle s_k \otimes \overline{\check{s}_k} \rangle(d\iota_F(j(\xi))),\fre \right>_{L^2}$
converges to $\langle f(d\iota_F(j(w))),\fre(\sigma) \rangle$
as $k \to \infty$
where $s_k \otimes \overline{\check{s}_k} \in E_0 \otimes \overline{E_1}$.
By Equation (\ref{eqn pairing}), we conclude that
\begin{equation}
\left< \Psi(s_k \otimes \overline{\check{s}_k},du(j(\xi)), d\iota_F(j(\xi))),\fre \right>_{L_2} =
\left<\langle s_k \otimes \overline{\check{s}_k} \rangle(d\iota_F(j(\xi))),\fre \right>_{L^2} \neq 0\end{equation}
for $k$ large enough.
This implies surjectivity by Corollary \ref{corollary surject}.
\end{proof}

\begin{cor} \label{cor:thickening_regular}
  For $k$ sufficiently large, there is an open neighbourhood of $\frs^{-1}(0) \subset \scrT$ in which $D\overline{\partial}_{\widetilde{J}_\Psi}$ is surjective. \qed
\end{cor}
\subsection{Gluing and the topological submersion}

There is a forgetful map
\begin{equation} \label{equation map to scrF}
	\Pi  : \scrT \lra{} \scrF, \qquad (u,\Sigma,F,\eta) \mapsto [\phi_F]\end{equation}
sending
an element of $\scrT$ to the underlying framed curve $\phi_F : \Sigma  \to \bC \bP^d$ (Equation  (\ref{equation F map})).
 As discussed in Lemma \ref{Lem:no_automorphisms}, $\scrF$ is a smooth manifold.  We will now use standard gluing results, as formulated in \cite{pardon2016algebraic}, to show that $\Pi$ gives $\scrT$ the structure of a fibrewise smooth topological submersion over $\scrF$.

As in Section \ref{Subsec:surjectivity}, consider $\widehat{u} : \Sigma \lra{} E$ a $\widetilde{J}_{\check{\Psi}}$-holomorphic curve whose image is contained in the zero-section
 $X\times \scrC$ and which is in the image of the open embedding
	$\scrT \hookrightarrow \overline{\ccM}_{0,0}(\widetilde{J}_\Psi)$ from Lemma \ref{lem:gromov_revisited}.
	Let $u \in \scrT$ be the corresponding element in $\scrT$. 
We now also assume that $\widehat{u}$ is regular:  the linearized operator $D_{\widetilde{J}_\Psi}$ is surjective. We will construct a neighborhood of $u \in \scrT$ compatible with the forgetful map $\Pi: \scrT \to \scrF$.

Let $w := \Pi(u) \in \scrF$ and let
$d'$, $L_\bullet$, $\scrF(L_\bullet)$ and $\mu_{L_\bullet}$ be as in Proposition \ref{prop local coordinate}.
Let
$D_i \subset E$ be the preimage
of the linear hyperplane $L_i \subset \bC\bP^{d}$ under the natural projection map $E \lra{} \scrC$ composed with the evaluation map to $\C P^d$ for each $i=1,\ldots,d'$.
Let $\scrT'$ be the space of stable $\widetilde{J}_\Psi$ holomorphic maps from genus zero nodal curves to $E$ 
with $d'$ marked points
so that, for each $i$, the $i$th marked point evaluates to $D_i$ and the map is transverse to $D_i$ at that point. Let
\begin{equation}\mu_{\scrT'} : \scrT' \lra{} \overline{\ccM}_{0,d'}\end{equation}
be the map sending each element of $\scrT'$ to its domain curve. A neighborhood of $u$ in $\scrT$ is canonically homeomorphic to a neighborhood of $\widehat{u}$ in $\scrT'$.
As a result, it is sufficient for us to describe a neighborhood of $\widehat{u}$ in $\scrT'$. The map $\mu_{\scrT'}$ is equal to $\mu_{L_\bullet} \circ \Pi$ under this identification.  
Usual elliptic regularity, and the fact that $\cdbar$-operators are compact perturbations of complex-linear operators, shows that the subspace $\mu_{\scrT'}^{-1}(\mu_{\scrT'}(u))$ is canonically endowed with a smooth structure and an orientation.

\begin{theorem} \cite[Appendix B]{pardon2016algebraic}
For any sufficiently small neighborhood $K$ of $u$ in the fibre $\mu_{\scrT'}^{-1}(\mu_{\scrT'}(u))$,
and any small neighborhood
$j: U \subset \overline{\ccM}_{0,d'}$ of
$\mu_{L_\bullet}(w)$,
there exists a continuous map $g$ fitting into a diagram
\begin{equation}
\xymatrix{ 
 U \times K \ar[r]^-g  \ar[d]_{\mathrm{proj}_U} &  \scrT' \ar[d]^{\mu_{L_\bullet} \circ \Pi} \\
 U \ar[r]_j &  \overline{\ccM}_{0,d'} 
 }
\end{equation}
and which is a homeomorphism
onto its image. Moreover, for each $a\in U$ the restriction 
\begin{equation}
g|_{\{a\} \times K} : \{a\} \times K \longrightarrow \Pi^{-1}(\mu_{L_\bullet}^{-1}(a))
\end{equation}
takes values in $\Pi^{-1}(\mu_{L_\bullet}^{-1}(a))$ and is an orientation-preserving diffeomorphism onto its image.
\end{theorem}
\begin{proof}
This theorem follows from the contents of \cite[Appendix B]{pardon2016algebraic}, in the simple case in which the obstruction bundle (which is denoted $E$ in loc. cit.) is trivial.

The gluing map $g$ is defined by Equations (B.10.3) and (B.10.4) in \cite{pardon2016algebraic}.
The gluing map is a homeomorphism onto its image by the proof of Theorem B.1.1 at the end of \cite[Section B.12]{pardon2016algebraic}.
Equation (B.10.1) in \cite{pardon2016algebraic} also ensures that $\mu_{\scrT'} \circ g$ is equal to the projection map to $U$.

The restriction of $g$ to $\{a\} \times K$
is obtained by Newton-Picard iteration
with initial conditions smoothly depending on $K$ for each $a \in U$
(See Equations (B.9.6), (B.9.7), (B.10.1) and (B.10.3) in
\cite{pardon2016algebraic}).
Hence this map is smooth for each fixed $a$.
The map $g$ respects orientations by \cite[Section B.13]{pardon2016algebraic}.
\end{proof}

We may phrase the above result in terms of the space $\scrT$ as follows: for any small neighborhood $K \subset \Pi(\Pi^{-1}(u))$ of $u$ and any small neighborhood
	$U \subset \scrF$ of $\Pi(u)$
	there is a continuous map
	\begin{equation}g : U \times K \lra{} \scrT\end{equation}
	which is a homeomorphism onto a small neighborhood of $u$
	so that
	\begin{enumerate}
		\item the projection map to $U$ is equal to $\Pi \circ g$ and
		\item the restriction of $g$ to $\{a\} \times K$
		is a smooth codimension $0$ embedding into $\Pi^{-1}(a)$ for each $a \in U$.
	\end{enumerate}
This can be summarised by the following statement:
\begin{corollary} \label{cor:fibrewise_smooth}
For $k$ large enough, 
there is an open neighborhood  of $\frs^{-1}(0)$ in  $\scrT$ 
so that
the map $\Pi|_U$ has a fiberwise smooth structure.
\end{corollary}

From now on, we may replace $\scrT$ by this neighbourhood, which we may assume, without loss of generality, to be invariant under the action of $U(N)$.

\begin{corollary} \label{cor:C1_loc}
The fibrewise-smooth topological submersion $\Pi: \scrT \to \scrF$ admits a $C^1_{loc}$ structure.
\end{corollary}

\begin{proof}
This follows from \cite[Proposition B.11.1]{pardon2016algebraic}.  Recall that the proof of the gluing theorem proceeds by constructing a pre-glued curve which is an approximate solution, and then applying a Newton iteration scheme to obtain an actual solution.  By unique continuation, we may view all the curves obtained from pregluing applied to a fixed stable map as elements of a space of smooth maps from the fixed domain which is the complement of small balls around the nodes (this viewpoint has been popularised in pseudoholomorphic curve theory by Fukaya, within the more delicate context of constructing a smooth structure). The referenced Proposition shows that, given a $J$-holomorphic neck in the gluing region, the $L^2$-norm of the solution near the ends of the neck controls all $C^k$-norms throughout the neck region.   It follows that the preglued curve is close to the actual solution in any $C^k$-norm on this space. Since the preglued curve does not depend on the choice of gluing parameter away from the necks, the comparison between charts associated to different gluing parameters are locally $C^1$-smooth with respect to any fixed choice of domain. Unique continuation again implies that the resulting smooth structure does not depend on this choice.
\end{proof}

\begin{rem}
We reiterate the point alluded to in the middle of the above proof regarding smooth structures: Corollary \ref{cor:C1_loc} is not saying anything about the differentiability of gluing with respect to the gluing parameter $\alpha$ itself, but that the (pre-)gluing map from the space of gluing parameters  to the space of solutions  is continuous when the target has the $C^1$- (or any $C^k$)-topology. One of the insights of \cite{AbouzaidBlumberg2021} is that this structure is sufficient to define a vector bundle on the thickening that plays the r\^ole of the tangent bundle; that paper and the present one differ in how this insight is implemented in practice.
\end{rem}

Corollary \ref{cor:C1_loc} shows that the vertical tangent bundle $T^{vt}\scrT$ of $\Pi: \scrT\to\scrF$ is well-defined.  It is then standard that $T^{vt}\scrT$ admits a stable almost complex structure. On each stratum of $\scrT$ corresponding to a fixed topological type of domain, $T^{vt}\scrT$ is stably isomorphic to the pullback of an index bundle defined over the whole configuration space. Over a compact subset $K \subset \scrT$ lying in a stratum of reducible domains, one can find a fixed stabilisation making all the linearised operators surjective, and then (pre-)gluing defines a map to the configuration space for maps from a glued domain, which is compatible with the index bundles by \cite[Diagram B.7.7]{pardon2016algebraic}. The result then follows from the argument of \cite[Proposition 4.3]{McDuff:examples}, using that all the linearised operators are compact perturbations of $\bC$-linear operators. 

\subsection{Incorporating marked points\label{Sec:marked_points}}

We will now explain how to use the ideas of Section \ref{section pullbacks of kuranishi charts}
to obtain a global Kuranishi chart presentation for the moduli space $\ccMbar_{0,n}(X,\beta)$ of stable genus zero curves with a non-empty ordered collection of marked points.

We begin with the moduli space $\ccMbar_{0,n}(\bC\bP^d,d)$ of $n$-marked degree $d$ curves in $\bC\bP^d$, the Zariski open subset $\scrF_n(d)$ of curves not lying in a hyperplane.
There is a natural map $f_d : \scrF(d) \to \scrF$ where we forget all the marked points.
This is a morphism of smooth quasi-projective varieties.
Let $Z_n$ be the space of pairs $(v,w)$ where $v = (u,\Sigma,F)$ is a
framed curve (Definition \ref{defn:framed_curve})
and $w \in \scrF_n$ satisfying $\phi_F = f_d(w)$
where $\phi_F$ is given in Equation \eqref{equation F map}.
Then $\ccMbar_{0,n}(\bC\bP^d,d)$ is naturally homeomorphic to the quotient
$Z_n / U(d+1)$.
In Section \ref{Sec:description} we constructed a natural global Kuranishi chart
$\K = (G,\scrT,\scrE,\frs)$ for the moduli space $\ccMbar_{0,0}(\bC\bP^d,d)$
(Definition \ref{defn moduli space kuranishi}).
The thickening $\scrT$ admits a natural topological submersion $\Pi : \scrT \to \scrF$
(Equation \eqref{equation map to scrF}).
Hence the pullback $f_d^*\K := (G,f_d^*\scrT,f_d^*\scrE,f_d^*\frs)$
(Definition \ref{DEFN pullback of kuranishi charts})
is a global Kuranishi chart for $\ccMbar_{0,n}(\bC\bP^d,d)$.

\subsection{Relating different presentations}

Recall the general operations of germ equivalence, stabilisation and group enlargement on global Kuranishi charts.

\begin{Proposition} Fix $(X,\omega,J,\beta)$ and  construct global Kuranishi charts for the moduli space $\ccMbar_{0,0}(\beta)$ with respect to distinct choices $(\Omega_i, \scrL_i, k_i)$ for $i=1,2$. Then these global charts are related by a sequence of the three given moves.
\end{Proposition}

\begin{proof}
Let $d_i = \langle [\Omega_i],\beta\rangle$; we have moduli spaces $\scrF(d_i)$ of framed curves of each degree, where the framing of a $J$-holomorphic curve $u:\Sigma\to X$ involves a choice of basis $F_i$ for $H^0(u^*L_i)$.  Write $\scrC_i \to \scrF(d_i)$ for the corresponding universal curve. The choice of $(\Omega_i,\scrL_i,k_i)$ is thus sufficient to construct the thickened moduli space $\scrT_i$ of solutions to
\begin{equation}
\overline{\partial}_J u +  \langle \eta_i \rangle \circ d\iota_{F_i} = 0
\end{equation}
as in Definition \ref{defn large space}, where we recall that 
	\begin{enumerate}
		\item $(u,\Sigma,F_i) \in \scrF(d_i)$ is a framed curve as in Definition \ref{defn:framed_curve} and
		\item $\eta_i \in H^0(\overline{\Hom}(\iota_{F_i}^* T\scrC_i, u^*TX) \otimes \iota_{F_i}^*\scrL_i^{k_i}) \otimes_\C \overline{H^0(\iota_{F_i}^*\scrL_i^{k_i})}$
	\end{enumerate}
	and this thickened moduli space comes with a forgetful topological submersion $\scrT_i \to \scrF(d_i)$ which is $G_i = U(d_i+1)$-equivariant, and with the tautological section $\frak{s}_i$ cutting out $Z_i \subset \scrT_i$ with $\ccMbar_{0,0}(J,\beta)= \scrM = Z_i/G_i$ for both $i=1,2$.
	
	We now introduce a `doubly thickened' moduli space $\scrT_{12}$ of `doubly framed' curves as follows.  A doubly framed curve is a quadruple $v=(u,\Sigma,F_1,F_2)$, where 
	where  $u: \Sigma \to X$ is a smooth curve representing $\beta$ so that $(u,\Sigma,F_i)$, $i=1,2$ are framed curves. An element of $\scrT_{12}$ is a tuple $(u,\Sigma,F_1,F_2,\eta_1,\eta_2)$ satisfying the equation
	\begin{equation}
	\overline{\partial}_J u + \langle \eta_1 \rangle \circ d\iota_{F_1} + \langle \eta_2 \rangle \circ d\iota_{F_2} \, = 0.
	\end{equation}
There is a tautological section $\frak{s}_{12}$ of the bundle over $\scrT_{12}$ with fibre 
\begin{equation}
\bigoplus_{i=1}^2 \, \left( H^0(\overline{\Hom}(\iota_{F_i}^* T\scrC_i, u^*TX) \otimes \iota_{F_i}^*\scrL_i^{k_i}) \otimes_\C \overline{H^0(\iota_{F_i}^*\scrL_i^{k_i})} \oplus \scrH_{d_i+1} \right)
\end{equation}
which records the perturbation data $\eta_i$ and the Hermitian matrices $\scrH(v_i) \in \scrH_{d_i+1}$ which exponentiate to the matrix of $L^2$-inner products of the framing basis. By construction,  $\frak{s}_{12}^{-1}(0) / (G_1\times G_2)$ again recovers the moduli space $\scrM$. It is sufficient to prove that each of the original global charts can be related to the doubly thickened global chart by the three moves introduced previously.  By symmetry, it therefore suffices to show that the global chart associated to $\scrT_1$ and that associated to $\scrT_{12}$ are so related.  To see that:

\begin{itemize}
\item Define $P_{12} \to \scrT_1$ to be the space of tuples $(u,\Sigma,F_1,F_2,\eta_1)$, where $(u,\Sigma,F_1,F_2)$ is a doubly framed curve and $(u,\Sigma,F_1,\eta_1) \in \scrT_1$.
 This is the total space of a $G_2$-principal bundle over $\scrT_1$, with fibre the possible space of framings $F_2$.

\item The map $(u,\Sigma,F_1,F_2,\eta_1) \mapsto (u,\Sigma,F_1,F_2,\eta_1,0)$ embeds $P_{12} \hookrightarrow \scrT_{12}$ into the locus where $\eta_2=0$. Because the perturbed $\cdbar$-equation is already transverse along this locus, one can solve the doubly-thickened $\cdbar$-equation for any sufficiently small $\eta_2$, which means an open neighbourhood $\scrU$ of the image of this embedding is an open neighbourhood of the zero-section in the total space of the vector bundle $\scrW \to P_{12}$ with fibre $H^0(\overline{\Hom}(\iota_{F_2}^* T\scrC_2, u^*TX) \otimes \iota_{F_2}^*\scrL_2^{k_2}) \otimes_\C \overline{H^0(\iota_{F_2}^*\scrL_2^{k_2})}$. 

\end{itemize}

Thus, germ equivalence relates $\scrT_{12}$ to  $\scrU$, which is also related by germ-equivalence to a stabilisation of a group enlargement of $\scrT_1$.  Tracking the vector bundles $\scrE_1$ and $\scrE_{12}$ through the moves shows that the global charts are related as claimed.
\end{proof}

It now follows from the general discussion of Section \ref{sec:vfc} that the virtual fundamental class is independent of all choices, except perhaps of the almost complex structure $J$. 

\subsection{Cobordism}

Up to this point, we have worked with a fixed compatible almost complex structure $J$. This suffices for the proofs of both Theorem \ref{thm:main} and \ref{thm:main_generalised}, so we will discuss independence of $J$ only very briefly.

\begin{Definition} \label{Def:thick} Two global Kuranishi charts $(G, \scrT_i, \scrE_i, \frs_i)$ are \emph{cobordant}  if there is a $G$-equivariant cobordism of thickenings $\scrT_i$ carrying a $G$-bundle $\scrE$ and section $\frs$ restricting to the given $\frs_i$ over the ends.
\end{Definition}

There are natural notions of cobordism of oriented, or smooth, or fibrewise smooth Kuranishi charts. 
Varying $J$ yields cobordisms of fibrewise-smooth $\bK$-oriented charts; a convenient set-up is to consider a larger symplectic manifold $X\times T^2$ equipped with a complex structure $\widehat{J}$ which induces given complex structures $J_0$ and $J_1$ over two particular fibres and which is such that projection to $T^2$ is holomorphic.\footnote{We choose $T^2$ to retain compactness, and to not introduce any new spherical homology classes.}  Our previous machinery builds a global chart for curves in the total space, and this comes with a map to $T^2$ which induces a cobordism of thickenings in the sense of Definition \ref{Def:thick}, which is compatible with the topological submersion over the moduli space of domains.

\begin{Lemma} \label{lem:cobordant}
Cobordant charts $(G, \scrT_i, \scrE_i, \frs_i)$ for $\ccMbar_{0,n}(X,\beta)$ yield the same composite map 
\begin{equation} \label{eqn:final}
H^*(X^n;\bK) \stackrel{ev}{\longrightarrow} H^*(\ccMbar_{0,n}(X,\beta,J_i; \bK) \stackrel{f_*^{\ccMbar_i}}{\longrightarrow} \bK_*
\end{equation}
\end{Lemma}

\begin{proof}[Sketch] 
Apply the construction of Section \ref{section moduli kuranishi} to the global chart associated to the total space of the cobordism.  The resulting push-forward is compatible with restriction to a fibre of the map to $T^2$ by combining Lemmas \ref{lemma germ vfc well behaved} and \ref{lemma stabilization}.
\end{proof}

\subsection{Enumerative invariants}

Suppose $\bK$ is one of the `original' periodic Morava $K$-theories based on the ring $\bZ/p$, so its coefficients form a graded field.   For any space, $H^*(X;\bK)$ is both a left and right $\bK_*$-module. Before proceeding, we need the technical:

\begin{Lemma} \label{Lem:K-linear} The left and right module structures agree. \end{Lemma}

\begin{proof}
This is proven by Boardman for certain cohomology theories $P(n)$ in \cite{Boardman}, which are quotients of the theory $BP_*$ and of which $K(n)$ is a further quotient. We have a natural isomorphism
  \begin{equation}
        H^*(X;K(n)) \cong H^*(X;P(n)) \otimes_{P(n)_*} K(n)_*, 
      \end{equation}
      from which the result follows. Alternatively, one may note that Boardman's proof that the module structures agree is a fairly direct consequence of Mironov's formula \cite{Mironov1978}
      \begin{equation}
        x \cup y = y \cup x + v_n Q_{n-1}(x) \cup Q_{n-1}(y),        
      \end{equation}
      where $Q_{n-1}$ is a cohomology operation known as the $(n-1)$-Bockstein. Such a formula holds in $K(n)$ as well as in $P(n)$ (cf. \cite[Equation (3.2.1)]{KultzeWurgler1987}).
\end{proof}
\begin{rem}
  The paper \cite{KultzeWurgler1987} is known to have a mistake, as explained in \cite{Nassau2002}, but this does not affect the above argument.
\end{rem}

In light of this, we can combine the evaluation map and virtual class
\begin{equation}
ev: \ccMbar_{0,d}(X,\beta) \to X^d \qquad \mathrm{and} \qquad \vfc: H^*(\ccMbar_{0,d}(X,J, \beta);\bK) \to \bK_*
\end{equation}
with the K\"unneth theorem $H^*(X^d;\bK) \cong H^*(X;\bK)^{\otimes_{\bK_*} d}$ to obtain $\Phi_{\beta,d}: H^*(X;\bK)^{\otimes_{\bK_*} d} \to \bK_*$.  The previous discussions show:

\begin{Lemma} The map $\Phi_{\beta,d}: H^*(X;\bK)^{\otimes d} \to \bK$ is  a well-defined symplectic invariant of $(X,\omega,\beta, d)$, independent of the choices made in constructing the Morava virtual class. \qed
\end{Lemma}

The same argument as for Lemma \ref{lem:cobordant} shows that \eqref{eqn:final} and $\Phi_{\beta,d}$  depend only on the symplectic deformation class of $\omega$.  Now using duality on $H^*(X;\bK)$ and taking $d=3$ we get for each fixed $\beta$ a well-defined product operation 
\begin{equation} \label{eqn:product_for_one_beta}
\Phi_{\beta}^{\dagger}: H^*(X;\bK) \otimes_{\bK_*} H^*(X;\bK) \to H^*(X;\bK)
\end{equation}
(depending on $\beta$). Lemma \ref{Lem:K-linear} shows that

\begin{Lemma} \label{Lem:K-linear_2}
The product \eqref{eqn:product_for_one_beta} is $\bK$-linear. \qed
\end{Lemma}

We now work over the coefficient ring $\Lambda_{\bK_*}$ which is the Novikov ring over the coefficients of Morava $K$-theory, so
\begin{equation}
\Lambda_{\bK_*} = \left\{ \sum_i a_i q^{t_i} \, | \, a_i \in \bK_*, t_i \in \bR, t_i \to \infty  \right\}
\end{equation}
where $\bK_* = H^*(pt;\bK) = \bF_p[v^{\pm 1}]$.  This is a $\bK_*$-module.

\begin{Definition} \label{Defn:QHK} The `naive quantum Morava $K$-theory' of $(X,\omega)$ is  the group
\begin{equation}
QH^*(X;\Lambda_{\bK_*}) := H^*(X;\bK) \otimes_{\bK_*} \Lambda_{\bK_*}.
\end{equation}
This has a $\bK_*$-linear product defined on generators by 
\begin{equation}
 a \ast b  = \sum_{\beta} \Phi_{\beta}^{\dagger}(a,b)\cdot q^{\langle [\omega],\beta\rangle}.
\end{equation}
\end{Definition}

Whilst apparently \emph{ad hoc}, this definition is equivalent to one invoking a Novikov spectrum, by the arguments used in \cite{AbouzaidBlumberg2021} in the more complicated setting of Hamiltonian Floer homology. The key point is that $\bK$ is field-like in the stable homotopy category, which provides the required Mittag-Leffler and flatness conditions.  However, the `naive' theory is \emph{not} expected to be associative. It is helpful to compare to the situation in quantum $K$-theory. Suppose $X$ is compact.  The Poincar\'e pairing 
\begin{equation}
(E,F) \mapsto \chi(E\otimes F) = \int_X ch(E\otimes F) \cdot \mathrm{Td}(X) 
\end{equation} defines a metric on $K(X)$ and the naive (or `fake') quantum $K$-ring is obtained by dualising the 3-point quantum function with respect to this metric. This is well-known to fail associativity; to obtain an associative ring, one must `correct' (i.e. $q$-deform) the metric using the two-pointed $K$-theoretic invariants, cf.  \cite{Lee:QKtheory}.  `True' quantum Morava $K$-theory  incorporates a similar deformation of the metric to obtain an associative algebra.   This is the subject of a sequel \cite{AMS}.
Note that since the classical cohomology $H^*(X;\bK)$ is not commutative for the prime $p=2$, see \cite{Boardman}, one expects to obtain an associative but non-commutative algebra $QH^*(X;\Lambda_{\bK_*})$ in general.

\begin{rem}
  Since the coefficients for the theories $\bK=K_{p^k}(n)$ with $k>1$ are no longer graded fields, they do not have K\"unneth isomorphisms in general. Therefore the virtual class
\begin{equation}
\begin{tikzcd}
H^*(X\times X;\bK) \ar[r] & H^*(\ccMbar;\bK) \ar[r, "\mathrm{vfc}"] & \bK_*
\end{tikzcd}
\end{equation}
does not give rise  to a (even naive) $\bK$-quantum product on $H^*(X;\bK)$ without further hypotheses. 
\end{rem}

\section{Review of Cheng's spectrum-level construction\label{Sec:review}}
\label{sec:revi-stable-homot}

This Section briefly recalls the setting for Cheng's results from \cite{Cheng,Cheng2} as summarised in Section \ref{Sec:Cheng}, and in particular the construction of the maps entering into the equivalence of \eqref{eqn:Cheng_equivalence}.  For a spectrum $\bK$, and a space $X$
we will write 
\begin{itemize}
\item $C_*(X;\bK)$ for the spectrum $\Sigma^{\infty}_+X \wedge \bK$, whose homotopy groups are the generalised homology groups $H_*(X;\bK)$;
\item $C^{-*}(X;\bK)$ for the mapping spectrum $F(\Sigma^{\infty}_+X,\bK)$,  whose homotopy groups are the generalised cohomology groups $H^{-*}(X;\bK)$.
\end{itemize}
For experts, we note that we omit any notation for the cofibrant and fibrant replacements of $\bK$ which are respectively applied in the constructions of $C_*(X;\bK)$ and $C^*(X;\bK)$.

We denote by $\bS$ the sphere spectrum\footnote{In Section \ref{Sec:outline} $\bS$ denoted the blow-up of $\bP^1\times\bP^1$, the base of a fibration $X \to \widetilde{P} \to \bS$. Since that fibration will not recur in this section, we hope this will not cause confusion.}.  We will sometimes suppress $\Sigma^{\infty}_+$ when talking about a suspension spectrum. A \emph{finite spectrum} is (one weakly equivalent to) a desuspension of the suspension spectrum of a finite CW complex.  

If a Lie group $G$ acts on a space $X$, we define the equivariant cochains by the Borel construction:
\begin{equation}
C^*_G(X;\bK) = C^*(EG \times_G X; \bK), \qquad C_*^G(X;\bK) = EG \times_G X \wedge \bK.
\end{equation}

 \subsection{The Thom diagonal\label{Subsec:thom_review}}
 
 We briefly recall the construction of Thom isomorphisms in extraordinary cohomology theories. 
 For this subsection, given a vector bundle $\xi: V \to X$ over a finite cell complex $X$, we denote by $Th(\xi)$ the Thom space of $\xi$ (we wrote $X^{\xi}$ for this space in the main body of this text, but we are now reserving this notation for the corresponding spectrum). The Thom space has the following naturality and stability properties:
 \begin{itemize}
 \item if $f: Y \to X$ and $\xi: V \to X$, there is a map $Th(f^*\xi) \to Th(\xi)$;
 \item given $\xi: V\to X$ and $\xi' : V' \to Y$, $Th(\xi \boxplus \xi') = Th(\xi) \wedge Th(\xi')$;
 \item if $\underline{\bR^k}$ is a trivial bundle, $Th(\xi \oplus \underline{\bR}^k) = \Sigma^k(Th(\xi))$.
 \end{itemize}
  We therefore define $X^{\xi}$,  the Thom spectrum, to be the suspension spectrum obtained from the pre-spectrum whose $j$-th space $\Sigma^j(Th(\xi))$ for each $j \in \N$.  By construction
  \begin{equation}
  X^{\xi\oplus\underline{\bR}} = \Sigma X^{\xi}
  \end{equation}
 (and we still have $X^{\xi \boxplus \xi'} = X^{\xi} \wedge X^{\xi'}$).   If $\xi$ is a virtual vector bundle with $\xi = \eta - \underline{\bR}^k$, then we define $X^{\xi}$ to be the spectrum obtained from the pre-spectrum with $j$-th space $\Sigma^{j-k} Th(\eta)$ for $j\geq k$ and the point for $j<k$.
 
\begin{Definition} The \emph{Thom diagonal} of $\xi$ is the map $\Delta_{Th}: Th(\xi) \to X \wedge Th(\xi)$ induced by the map $v \mapsto (p(v),v)$ where $p: D(V) \to X$ is the projection of the disc bundle. 
The Thom diagonal defines a map of spectra $\Delta: X^{\xi} \to \Sigma_+^{\infty}X \wedge X^{\xi}$. 
\end{Definition}
If $X$ is equipped with a $G$-action, then the Thom diagonal is equivariant.

Now fix a ring spectrum $\bK$, defining a generalised cohomology theory $H^*(\bullet;\bK)$. We say a bundle $\xi$ is $\bK$-oriented if we have a Thom class $u_\xi: X^{\xi} \to \bK$ with the property that the composite map
\begin{equation}
X^{\xi} \wedge \bK \stackrel{\Delta}{\longrightarrow}\Sigma_+^{\infty}X \wedge X^{\xi} \wedge  \bK \stackrel{u_{\xi}}{\longrightarrow} X_+ \wedge \bK \wedge \bK  \stackrel{\mu_{\bK}}{\longrightarrow} X_+ \wedge \bK
\end{equation}
is an equivalence, where $\mu_{\bK}: \bK\wedge \bK \to \bK$ is the product. On (co)homology this gives the Thom isomorphisms \eqref{eqn:thom}. 
 This coincides with the definition of $\bK$-orientation in Definition \ref{DEFN oriented vector bundle} (\cite[Chapter V, Theorem 1.15]{Rudyak}).

\subsection{Cheng's theorems\label{Sec:Cheng_review}}

We briefly recall some notions from stable equivariant homotopy theory. See \cite{LMSM} for an introduction. Let $G$ be a compact Lie group.
We fix a universe $\scrU$, i.e a real $G$-inner product space containing countably many copies of each finite-dimensional representation. Observe that the subspace $\scrU^G$ corresponding to all the trivial representations is isomorphic to  $\bR^\infty$.
Let  $G\scrU$ be the category of $G$-spectra and let $G\scrU^G$ be the subcategory indexed by $\scrU^G$ and $\scrS\scrU^G$ the subcategory of $G\scrU^G$ corresponding to spectra whose associated $G$-action is trivial (in other words, just spectra).
Let $\iota^* : G\scrU \to G\scrU^G$ be the natural restriction map and let $\iota_* : G\scrU^G \lra{} G\scrU$ be its left adjoint. Also, let $\epsilon^* : \scrS \scrU^G \lra{} G\scrU^G$ be the functor which assigns a trivial $G$-action to a spectrum and let the \emph{fixed point functor} $X \to X^G$ be its right adjoint or the right adjoint of $\iota_* \circ \epsilon^*$ depending on the context. 
The elements of $G\scrU$ corresponding to $G$-CW spectra are built out of cells, each one equal to the cone of $\Sigma^{\infty}_+ (G/H)\wedge S^k$ for some $k\in \bZ$ and $H\leq G$ a subgroup, whose attaching maps have domain the base of such cones.

In this section, we recall Cheng's Poincar\'{e} duality results for a Morava $K$-theory $\bK = K_p(n)$, formulated at the level of spectra rather than on (co)homology, with the goal of seeing that all results still hold if we set  $\bK$ to be a Morava $K$-theory $K_{p^k}(n)$, with the same proof as the case $k=1$ which Cheng considered.

\begin{Theorem}[Cheng] \label{thm:Cheng}
Let $M$ be a smooth closed $m$-manifold with a smooth action of a compact Lie group $G$. There is a map of spectra
\begin{equation}
\lambda_{G,M} : \Sigma^{\infty}_+(EG\times_G M) \to F((EG\times_G M)^{-TM\oplus \frg}, \bS)
\end{equation}
which, if the $G$-action has finite stabilisers, induces an isomorphism of groups
\begin{equation}
H_*(EG\times_G M;\bK) \cong \widetilde{H}^{-*}((EG\times_G M)^{-TM\oplus \frg};\bK).
\end{equation}
\end{Theorem}
\vspace{0.2cm}

The isomorphism in the statement of the above theorem is the map on homotopy groups obtained from $\lambda_{G,M}$ by smashing with $\bK$ and composing with the natural map
\begin{equation}
     F((EG\times_G M)^{-TM\oplus \frg}, \bS) \wedge \bK \to  F((EG\times_G M)^{-TM\oplus \frg}, \bK).
\end{equation}


If $M$ is open, and $M^c$ denotes the one-point compactification, then there is a similar map
\begin{equation}\label{eqn:relative_lambda}
\lambda_{G,M} : \Sigma^{\infty}(EG\times_G M^c) \longrightarrow F((EG\times_G M)^{-TM\oplus \frg},\bS)
\end{equation}
again induces an equivalence on (co)homology over $\bK$. If $M$ is a manifold with boundary, then $\mathrm{int}(M)^c = M/\partial M$, and there is also a statement in this case, cf. \cite[Remark 1.2]{Cheng}.  We also recall the Addenda to the result, giving compatibility with `open inclusions' and `change of groups'. 

\begin{Addendum} \label{Add:relative_Cheng} \cite[Lemma 4.5.1, Proposition 4.5.2]{Cheng2} 
If $U \subset M$ is a $G$-invariant open submanifold of a closed manifold $M$ there is a homotopy commutative diagram
\begin{equation}
  \begin{tikzcd}
\Sigma_+^{\infty}(EG\times_G M) \ar[d] \ar[r,"\lambda_{G,M}"] &  F((EG\times_G M)^{-TM\oplus \frg},\bS) \ar[d] \\ 
\Sigma^{\infty} (EG\times_G U^c) \ar[r,"\lambda_{G,U}"] & F((EG\times_G U)^{-TU\oplus \frg}, \bS)    
  \end{tikzcd}
\end{equation}
where the vertical arrows are induced by the collapse map $M_+ \to U^c$ and restriction respectively.
\end{Addendum}

\begin{Addendum} \label{Add:change_of_group}
Let $M$ and $N$ be closed manifolds with smooth $G$ and $H$-actions respectively for some Lie groups $G$ and $H$.
If $M/G \cong N/H$, then there is a commutative diagram \cite[Proposition 1.3]{Cheng}
\begin{equation}
  \begin{tikzcd}
\Sigma^{\infty}_+ (EG\times_G M) \ar[r,"\lambda_{G,M}"] \ar[d,dotted] & F((EG\times_G M)^{-TM\oplus \frg},\bS) \ar[d,dotted] \\
\Sigma^{\infty}_+(EH \times_H N) \ar[r,"\lambda_{H,N}"]& F((EH\times_H N)^{-TN\oplus \frh},\bS)    
  \end{tikzcd}
\end{equation}
where the vertical maps are induced by equivalences $M/G \simeq P/(G\times H) \simeq N/H$ for the fibre-product $P$ of $M \to M/G\cong N/H \leftarrow N$.
\end{Addendum}

We now recall the construction of $\lambda_{G,M}$, and the proof that the induced map with coefficients in $\bK$ vanishes. The ingredients which enter into Cheng's construction are the norm map, Atiyah duality, and the Adams isomorphism. We review these in order:

For any $X \in G\scrU$, the norm is the composite
\begin{equation}
   EG_+ \wedge X \to S^0 \wedge X \to F(EG_+,X) 
\end{equation}
of the natural maps induced by the projection $EG_+ \to S^0$ and the identification $X = F(S^0,X)$. The key fact about the norm map that we use is the following result, originally due to Greenlees and Sadofsky \cite{Greenlees-Sadofsky} for finite groups:
\begin{Lemma}\cite[Proposition 4.2]{Cheng}  \label{lem:tate-vanishes} If $Y$ is a $G$-CW spectrum built from cells of the form $(G/H)_+\wedge S^k$ with $H\leq G$ finite, then norm map on $Y \wedge \iota_*\bK $ induces an equivalence on (derived) $G$-fixed points. \qed 
\end{Lemma}
Taking $X = \Sigma^{-\frg}\Sigma^\infty_+ M \wedge \iota_* \bK$, one get an equivalence of spectra
\begin{equation}\label{eqn:alpha}
\left(EG_+ \wedge  \Sigma^{-\frg}\Sigma^\infty_+ M \wedge \iota_* \bK\right)^G \longrightarrow F(EG_+, \Sigma^{-\frg}\Sigma^\infty_+ M\wedge \iota_* \bK)^G,
\end{equation}
where it is crucial on the left hand side to take \emph{derived fixed points}, which in the setting of \cite{LMSM} are computed by replacing the given $G$-suspension spectrum by the associated $G$-$\Omega$ spectrum.

The rest of the argument is a computation of the two sides of the above equivalence. For the left hand side, the Adams isomorphism asserts that the fact the $G$ action is free (because of the presence of the wedge factor $EG_+$), yields an equivalence
\begin{equation} \label{eq:Adams}
(EG\times_G M)_+ \wedge \bK \simeq (EG_+ \wedge  \Sigma^{-\frg}\Sigma^\infty_+ M \wedge \iota_* \bK)^G.    
\end{equation}
On the right hand side, equivariant Atiyah duality gives an equivalence in $G\scrU$
\begin{equation}\label{eqn:Spanier}
\Sigma^{\infty}_+M \simeq F(M^{-TM},\bS),
\end{equation}
which, using the triviality of the $G$ action on $ \iota_*\bK $, induces equivalence of $G$-spectra
\begin{align} \label{eqn:adjointrel}
F(EG_+, \Sigma^{-\frg}\Sigma^\infty_+ M \wedge \iota_* \bK)^G & \simeq F(EG_+ \wedge \Sigma^{\frg} M^{-TM}, \iota_*\bK)^G  \\
&  \label{eqn:quotientrel}
\simeq F((EG_+\wedge \Sigma^{\frg} M^{-TM})/G, \bK) \\
& \simeq F((EG\times_G M)^{-TM\oplus \frg}, \bK)
\end{align}

Since the exact form of Addendum \ref{Add:relative_Cheng} is not stated in \cite{Cheng}, we briefly indicate its proof. First, we note that both the Adams isomorphism and the norm map are functorial with respect to maps of $G$-spectra; it thus suffices to establish the desired homotopy commutative square of Atiyah duality maps. As \cite{Cheng} formulates all constructions in the smooth setting, the Atiyah duality step is obtained from a Pontrjagin-Thom construction for the $G^2$-manifold $G\times M$.
Replacing the map
\begin{equation}
\Delta: G\times M \to M\times M, \quad (g,x) \mapsto (gx,x)
\end{equation}
by a $G^2$-equivariant embedding into a vector bundle over $M \times M$, and using equivariant Atiyah duality, gives a map $\Sigma^{\infty}_+ M \times M^{-TM\oplus \frg} \to \Sigma^{\infty}_+(G\times M)$ in a category $G^2\scrU$. Smash with $EG^2_+$ and use freeness of the domain to infer that this comes from a morphism in the category $G^2\scrU^{G^2}$, 
\begin{equation}
\Sigma^{\infty}_+(EG\times M) \wedge (EG\times M)^{-TM\oplus \frg} \to \Sigma^{\infty}_+(EG^2\times G\times M)
\end{equation}
Composing the quotient by the $G^2$ action with the collapse map from $\Sigma^{\infty}_+(EG^2\times_{G^2} (G\times M)) \to \Sigma^{\infty}_+(\mathrm{pt}) = \bS$ yields (the adjoint of) $\lambda_{G,M}$. The compatibility of $\lambda_{G,M}$ with passage to open subsets, Addendum \ref{Add:relative_Cheng}, then  follows from the compatibility of the Pontrjagin-Thom collapse map underlying Atiyah duality with restriction to open subsets.

 \subsection{Morava $K$-theories with ring coefficients}
\label{sec:morava-k-theories}
\label{Sec:more_coefficients}
The goal of this section is to recall enough basic facts from (stable) equivariant homotopy theory in order to give the construction of the Morava $K$-theories with ring coefficients based on $\bZ/(p^k)$ and establish that they satisfy the analogue of Theorem  \ref{thm:Cheng}. 

We begin with existence and complex orientability:
\begin{proof}[Proof of Proposition \ref{Prop:coefficients-text}]
  Fix a prime $p$ and let $\bZ_p$ denote the $p$-adic integers and $\bZ_{(p)} = \bZ_p \cap \bQ$.  The Johnson-Wilson spectrum $E(n)$ has coefficients $\bZ_{(p)}[v_1,\ldots,v_{n-1},v_n,v_n^{-1}]$. There is an `integral' version of Morava $K$-theory with coefficients $\bZ_p[v_n,v_n^{-1}]$, obtained as a completion of the quotient of the $E(n)$-spectrum which kills the $v_i$ with $1\leq i\leq n-1$, see e.g. \cite[Section 1.3]{Greenlees-Sadofsky}. Further quotienting by the (regular) element $p^k \in \bZ_p$ gives a theory $K_{p^{k}}(n)_*$ with the required coefficients. The existence of such a quotient as a complex-oriented theory follows from \cite[Theorem 2.6]{Strickland}, when $p>2$, and from \cite[Proposition 2.9]{Strickland} when $p=2$.  (Note that in case $p=2$ the actual construction involves a quotient of $MU_{(2)}$ which does not factor through the usual construction of $E(2)$.) In either case, we have by construction maps of spectra
  \begin{equation}\label{eqn:iterated_cone}
        \cdots  \to  K_{p^k}(n) \to K_{p^{k-1}}(n) \to \cdots \to K_{p^{2}}(n) \to  K_{p}(n)
      \end{equation}
      with fibre $K_{p}(n)$ at each step (the corresponding degree $1$ cohomology operation is the Bockstein operation $Q_0$ in the terminology of \cite[Theorem 2.2]{BakerWurgler1991}). This shows that $K_{p^k}(n)  $ is an iterated extension of $K_p(n)$, hence lie in the $K_p(n)$-local category.
\end{proof}

We now consider the properties discussed in Section \ref{Sec:Cheng}.
The summary of the construction  shows that Cheng's theorem and its Addenda hold for a spectrum $\bK$ if one knows the analogue of Lemma \ref{lem:tate-vanishes}. Cheng's proof \cite[Section 4]{Cheng} of this result is reduced to the case in which $M = G/H \times S^k$ by an induction over cells. Furthermore, there is an evident commutative diagram
\begin{equation}
\begin{tikzcd}
\left(EG_+ \wedge \Sigma^{\infty}_+ G/H \right)^{G} \arrow[rr] && F(EG_+, \Sigma^{\infty}_+ G/H) \arrow[d, "\sim"] \\
\left(E(G\times H)_+ \wedge \Sigma^{\infty}_+ G \right)^{G \times H}  \arrow[rr] \arrow[u, "\sim"] \arrow[d, "\sim"] && F(E(G\times H)_+, \Sigma^{\infty}_+ G)^{G \times H} \\ 
\left(EH_+ \wedge \bS \right)^{H}  \arrow[rr] && F(EH_+, \bS)^{H}. \arrow[u,"\sim"] 
\end{tikzcd}
\end{equation}

As noted by Cheng, this reduces the statement that the desired norm map with coefficients in a (non-equivariant) spectrum $\bK$ is an equivalence to the statement that the map
\begin{equation}
  \left(EH_+ \wedge \iota_*\bK \right)^{H} \to F(EH_+,\iota_*\bK)^{H}  
\end{equation}
is an equivalence whenever $H$ is a finite group (this is implicit in the above discussion because we assumed that the $G$ action has finite stabilisers). When $\bK$ is the usual Morava $K$-theory, then this is precisely the main result of \cite{Greenlees-Sadofsky} (which in turn builds on  Ravenel's theorem \cite{Ravenel} that $K(n)^*(BH)$ has finite rank for finite $H$). The following generalisation follows from the subsequent work of Hovey and Sadofsky \cite{HoveySadofsky1996}, but can be extracted more directly from \cite{HoveyStrickland1999}:
\begin{lemma}[Corollary 8.7 of \cite{HoveyStrickland1999}] \label{lem:norm_map_equivalence_K-local}
  If $\bK$ is a $K(n)$-local spectrum, then the norm map with $\bK$ coefficients is an equivalence for finite groups.
\end{lemma}
We conclude that the analogue of Lemma \ref{lem:tate-vanishes} holds for $K(n)$-local theories, and hence that so does Theorem \ref{thm:Cheng}. By Proposition \ref{Prop:coefficients-text}, proved at the beginning of this section, this applies, in particular, to the theories $\bK = K_{p^k}(n)$.

\section{Additive splitting for complex orientable theories\label{Sec:chromatic}}

In this section we prove Theorem \ref{thm:main_generalised}.
Since we build upon Cheng's work which we explained in the previous section, we use classical models of spectra as in \cite{Adams1974}.

\begin{rem} \label{rem:could_upgrade}
  One could alternatively upgrade the constructions of Cheng to the setting of orthogonal spectra; for finite group actions, this is implemented in \cite{AbouzaidBlumberg2021}, with explicit spectrum level constructions. These methods can be extended to the study of compact Lie group actions. Such an extension would make it possible to describe the compatibility of Cheng's construction with ring structures \cite{BlumbergPersonal}, and give more formal proofs of the results of this section. In particular, in Section \ref{Sec:nullhomotopy} below, we work with the homological reformulation of our results from Section \ref{Sec:vfc_homological} to avoid proving that Cheng's duality map, for a $K(n)$-local spectrum $\bE$, is a map of $\bE$-modules.
\end{rem}

Since complex orientable generalised cohomology theories $\bE$ have Gysin maps and corresponding Wang sequences, the Serre spectral sequence for $P_{\phi} \to S^2$ degenerates with $\bE$-coefficients exactly when the sweepout map from Corollary \ref{Cor:main} vanishes, i.e. it is equivalent to vanishing of 
\begin{equation}
\delta_{\phi}: H_i(X;\bE) \to H_{i+1}(X;\bE)
\end{equation}
associated to the loop $\phi: S^1\to \Diff(X)$.  As in Section \ref{Sec:Splitting}, going from this to additive splitting of the $\bE$-cohomology of $P_{\phi}$ requires a stronger conclusion, which involves splitting the restriction map.  In this section, we shall study this splitting at the level of spectra.
\medskip

The sweepout is associated to the map of spaces 
\begin{equation}
\Phi: S^1 \times X \to X, \qquad (\theta,x) \mapsto \phi(\theta)(x).
\end{equation}
We can reconstruct $P_\phi$ from this map as the (homotopy) pushout in the diagram
\begin{equation}
  \begin{tikzcd}
    S^1 \times X \ar[r,"\Phi"] \ar[d] & X \ar[d] \\
   X \ar[r] & P_\phi,
  \end{tikzcd}
\end{equation}
where the left vertical map is the projection to the second factor.

Since $\Phi|_{\{1\}\times X} = \id$, by subtracting the second projection from $\Phi$, the sweepout defines a \emph{stable sweepout} map 
\begin{equation}
\eta_{\phi}: S^1 \wedge X_+ \wedge \bS \to  X_+ \wedge \bS
\end{equation}
on the corresponding suspension spectra. The different notation for suspension spectra presages the fact that we shall presently discuss the same construction with coefficients in a spectrum more general than the sphere spectrum $\bS$. Note that $\eta_\phi$ is null homotopic if and only if $\Phi$ is stably homotopy equivalent to the projection to $X$. Using the above diagram, we conclude:
\begin{Lemma} 
If the stable sweepout $\eta_\phi$ is null homotopic then $P_{\phi}$ is stably homotopy equivalent to $S^2 \times X$. \qed
\end{Lemma}
\begin{rem} \label{rem:sharpen}
Working in a structured context, the above result can be sharpened using different methods: all perspectives are equivalent by Koszul duality \cite{BlumbergMandell2011,HessShipley2016}, and the description of local systems in terms of holonomy \cite{Waldhausen1982}. In the most geometric perspective cf. \cite{Clapp1981}, one assembles the suspension spectra of the fibres of the projection $P_{\phi} \to S^2$ into a parametrised spectrum over $S^2$. A null homotopy of  $\eta_\phi$ is then equivalent to a trivialisation of this parametrised spectrum.

  The datum of a parametrised spectrum is encoded by the holomony action of the suspension spectrum of the based loop space on the fibre, i.e. by equipping $X \wedge \bS$ with the structure of a module over $\Omega S^2 \wedge \bS$; by the James construction, cf. \cite[Section 4J]{hatcheralgebraictopology}, \cite{cohenstableproofs},\cite[Section VII.5]{LMSM}, $\Omega S^2$ is a free algebra on $S^1$, so that this module structure is encoded by a single map $ S^1 \wedge X_+ \wedge \bS \to  X_+ \wedge \bS$, which is precisely $\eta_\phi$.

One can express the desired data entirely in terms of the total space $P_\phi$ and the base $S^2$ (rather than the fibre $X$ and the based loop space of $S^2$). As the base of our fibration is simply connected, this may be done in two equivalent ways, which are related by duality: on the homological side, recall that the diagonal map equips every space with the structure of a coalgebra, which is a structure inherited by suspension spectra. The suspension spectrum $ (P_{\phi})_+ \wedge \bS$ is then a comodule spectrum over $S^2_+ \wedge \bS$, which is equivalent to the comodule $ \left( S^2 \times X \right)_+ \wedge \bS $ if and only if $\eta_\phi$ is null homotopic. Dually, the spectrum $F(S^2_+; \bS)$ which gives rise to cohomology is naturally a ring spectrum, and pullback makes $F((P_{\phi})_+, \bS)$ into a module over it, and the existence of an isomorphism of modules with $F(\left(S^2 \times X\right)_+ , \bS) $ is then equivalent to the triviality of $\eta_\phi$. 
\end{rem}

 We shall be using an analogous result in the homotopy category of $\bE$-modules, for a ring spectrum $\bE$ associated to a generalised cohomology theory. Our result is stated in terms of the composition of $\eta_\phi$ with the unit map from $\bS$ to  $\bE$
\begin{equation} \label{eq:stable_sweepout-E}
\eta_\phi \wedge \bE :  S^1 \wedge X_+ \wedge \bS \to  X_+ \wedge \bS \to  X_+ \wedge \bE.
\end{equation}

   The next result can be implemented in a many different setting, but we will use it in the context of the homotopy category of spectra (as in \cite{Adams1974}): the spectrum $\bE$ is equipped with a product which is homotopy associative and homotopy unital.

\begin{lemma} \label{Lem:null}
There is a bijection between the set of homotopy classes of null homotopies for $\eta_\phi \wedge \bE$, and homotopy classes of equivalences of $\bE$-modules
\begin{equation}
    (P_{\phi})_+ \wedge \bE  \cong X_+ \wedge \bE \vee \Sigma^2 X_+ \wedge \bE .
  \end{equation}
  (The right hand side is naturally equivalent to $\left(S^2 \times X\right)_+ \wedge \bE$.)
\end{lemma}
\begin{proof}
It is convenient to work in a more general context: let $f: X \to Y$ be a map of spectra, and $\bE$ a ring spectrum. We seek to prove that an $\bE$-module splitting of the induced map
  \begin{equation} \label{eq:map_free_E-modules}
    X \wedge \bE  \to Y \wedge \bE
  \end{equation}
  is equivalent to the datum of a null-homotopy of the composite map of spectra
  \begin{equation}\label{eq:null_homotopy_cofibre}
    \mathrm{Cofibre}(f) \to \Sigma X \to \Sigma X \wedge \bE.    
  \end{equation}
There are two steps to prove this: first,  we reformulate Equation \eqref{eq:null_homotopy_cofibre} as a nullhomotopy for the map
  \begin{equation}
    \mathrm{Fibre}(f) \to X \wedge \bE 
  \end{equation}
 by using the triangulated structure of the category of spectra.  Second, by considering the long exact sequence associated to taking maps to $X\wedge \bE$, such a null-homotopy is equivalent to the datum of a map
  \begin{equation}
        Y \to X \wedge \bE, 
  \end{equation}
 with the property that the composite
  \begin{equation}
    X \to Y \to X \wedge \bE
  \end{equation}
  is homotopic to the natural map $X \to X \wedge \bE$ induced by the unit. This is equivalent to an $\bE$-module splitting in Equation \eqref{eq:map_free_E-modules} because the the datum of a map of $\bE$-modules from $Y \wedge \bE$ to an $\bE$-module $\bM$ is equivalent to the datum of a map of spectra from $Y$ to $\bM$.
\end{proof}

 We stress that the above result provides a criterion for a splitting of modules in terms of a map of spectra. 

The proof of Theorem \ref{thm:main_generalised} can now be summarised as follows:
\begin{enumerate}
\item the arguments of sections \ref{Sec:outline} and \ref{Sec:Morava_virtual_class}, combined with Lemma \ref{Lem:null},  show that 
the stable sweepout map  with target $ X_+ \wedge \bK$ is nullhomotopic, when $\bK$ is a $K(n)$-local theory;
\item results related to the `chromatic fracture square', in particular a theorem due to Hovey \cite[Theorem 3.1]{Hovey}, show that the sweepout is null after wedging with the $p$-localisation of complex cobordism;
\item The fact that every abelian group injects in the product of its $p$-localisations then yields a null-homotopy of the sweepout map in complex cobordism;
\item the null-homotopy for complex cobordism yields a null-homotopy for any complex-oriented spectrum.
\end{enumerate}
The details are presented in the rest of this section after a brief review of localisation.

\subsection{Localization}

Fix a spectrum $E$. Recall that a spectrum $Y$ is $E$-local if $[P,Y] = 0$ whenever $P$ is $E$-acyclic, i.e. $P \wedge E$ is contractible. There is a Bousfield localization operation $X \mapsto L_E(X)$ on spectra with the following key properties:

\begin{itemize}
\item $L_E(X)$ is characterised by a universal property which in particular says that any map from $X$ to an $E$-local spectrum factors through $L_E(X)$;
\item if $X$ is finite, then $L_E(X\wedge  \bK) = L_E(X) \wedge  \bK$.
\end{itemize}

For the second property,  compare to \cite[Proposition 9.23]{Lawson2020}, which shows that it also holds if $\bK$ is finite type, meaning its homotopy groups are finitely generated in every degree.

We now recall how Bousfield localisation extends both the notions of localisation and completion of abelian groups: If $G$ is an abelian group, the Moore spectrum $SG$ is characterised by having homotopy groups $\pi_i(SG) = 0$ for $ i<0$, $G$ for $i=0$ and $H_i(SG) = 0$ for $i>0$. Let $\Z_p$ denote the $p$-adic integers and $\Z_{(p)} = \Z_p \cap \bQ$. 

\begin{Example} \label{ex:localize}
The $p$-localization $X_{(p)}$ of a spectrum $X$ is the Bousfield localisation with respect to $S\bZ_{(p)}$. 
This has homotopy groups $\pi_*(X) \otimes \bZ_{(p)}$. 
\end{Example}

We shall use the following standard result to conclude the existence of a null-homotopy from its existence after $p$-localisation:
\begin{Lemma} \label{lem:localisation-injective}
  If $Y$ is a finite spectrum, and $\bE$ is arbitrary, the natural map
  \begin{equation}
    [Y,\bE] \to \prod_{p}  [Y, \bE_{(p)}]   
  \end{equation}
is injective.
\end{Lemma}
\begin{proof}
  For any spectrum $Y$, there is a map
\begin{equation}
[Y, \bE]_{(p)} \to [Y, \bE_{(p)}]. 
\end{equation}
For spheres this is an isomorphism, because localisation of spectra acts by localisation on homotopy groups, cf. Example \ref{ex:localize};  an induction over cells then shows that this is an isomorphism for all finite spectra.  The result now follows from the fact that any abelian group injects into the direct product of its localisations.  
\end{proof}

The result above shows that a map with finite domain and arbitrary target is nullhomotopic exactly when all its $p$-localizations are nullhomotopic.   (Compare to the corresponding unstable result for spaces in \cite[Theorem 13.1.1]{May-Ponto}.)  
The $p$-localisation of $MU$ plays a particularly important role in our work, because the first step in the construction of the Morava $K$-theories is the following result:
\begin{thm}[Quillen \cite{Quillen1969}] \label{theoremlocalizationBP}
  The $p$-localisation of $MU$ decomposes as a finite wedge of shifts of the Brown-Peterson spectrum $BP$. \qed
\end{thm}

\subsection{A nullhomotopy for $K(n)$-local theories\label{Sec:nullhomotopy}}
The proof of Theorem \ref{thm:main} shows that the map $H^*(P_{\phi};\bK) \to H^*(X;\bK)$ is surjective when $\bK = K_{p^k}(n)$ is a Morava $K$-theory.  In fact, since the argument hinges on Lemma \ref{lem:norm_map_equivalence_K-local}, we have in fact proved the corresponding surjectivity whenever $\bK$ is a complex oriented $K(n)$-local cohomology theory. Our goal in this section is to explain how to lift this statement to a splitting of $\bK$ modules. Following Section \ref{Sec:vfc_homological} and Remark \ref{rem:could_upgrade}, we will state our result on the homological side.

As a starting point, we first have to explain how the homological virtual class from Section \ref{Sec:vfc_homological}, which arises as a cap product, lifts to spectra. The key point is that for any ring spectrum $\bE$ and space $Y$, the cap product  $\bE^*(Y) \otimes \bE_*(Y) \to \bE_*(Y)$ is itself induced by a map of spectra
\begin{equation}
\begin{tikzcd}
F(Y,\bE) \wedge Y \wedge \bE \arrow[r,"\Delta_Y"] &  F(Y,\bE) \wedge Y\times Y \wedge \bE  \arrow[r, "\mathrm{eval}"] & Y \wedge \bE \wedge \bE \arrow[r, "\mu_\bE"] &  Y \wedge \bE.
\end{tikzcd}
\end{equation}

\begin{lemma} \label{lem:hom_split}
There is a map of $\bK$-modules $(P_{\phi})_+\wedge \bK \to X_+ \wedge \bK$ which up to homotopy splits the natural map induced by inclusion.
\end{lemma}

\begin{proof} We consider a global Kuranishi chart $(\scrK_{\phi},\partial \scrK_{\phi})$ for spaces of curves with two marked points on $P_{\phi}$ and $(X \times S^2)_{hor}$ which is smooth and $\bK$-oriented, as in Section \ref{Sec:vfc_homological}.  View the homological virtual class on this moduli space as a map
\begin{equation}
\begin{tikzcd}
\bS \arrow[r] & \left(\scrK_{\phi} \times_G EG\right)_+ \wedge \bK,
\end{tikzcd}
\end{equation}
and denote by $[\vfc]$ is image under the evaluation map
\begin{equation}
  \scrK_{\phi} \times_G EG \to     P_{\phi} \times (X \times S^2)_{hor}.
\end{equation}
Suppressing shifts, and trivialisations of Thom spectra, we consider the following map of $\bK$-modules:
\begin{multline} \label{eq:splitting_map}
  P_{\phi,+} \wedge \bK \stackrel{\mathrm{Dual}}{\longrightarrow}   F(P_+;\bK) \stackrel{\mathrm{pull}}{\longrightarrow} F \left( (P_\phi \times (X\times S^2)_{hor})_+;\bK \right) \\
 \stackrel{{\frown [\vfc]} }{\longrightarrow}  \left(P_\phi \times(X\times S^2)_{hor}\right)_+ \wedge \bK  \stackrel{\mathrm{project} }{\longrightarrow}  (X\times S^2)_{hor,+} \wedge \bK.
  \end{multline}

  In order to show that this map gives the desired splitting, one reproduces the framework of Section \ref{sec:splitting} at the level of spectra, replacing the moduli spaces with their global charts, and formulating all results homologically. Briefly, we consider the thickening $\scrK_h$ for the moduli space $ \ccMbar_h$ of stable spheres, in the section class, with $2$ marked points, one of which is constrained to map to $(S^2 \times X)_{hor}) $ and one to $\tilde{P}$. The compatibility of the fundamental classes of $\scrK_h $ and $\scrK_\phi $
  shows that the composition in Equation \eqref{eq:splitting_map} factors as
  \begin{equation}
    P_{\phi,+} \wedge \bK \to  \tilde{P}_{+} \wedge \bK \to (X\times S^2)_{hor,+} \wedge \bK. 
  \end{equation}
  On the other hand, the inclusion of $X$ in $P_{\phi}$ fits in a homotopy commutative diagram
  \begin{equation}
    \begin{tikzcd}
      X \ar[r] \ar[d]& X \times S^2 \ar[d] \\
      P_{\phi} \ar[r] & \tilde{P},
    \end{tikzcd}
  \end{equation}
  so it suffices to show that the composite
  \begin{equation}
        X_+ \wedge \bK  \to  \left(X \times S^2\right)_+  \wedge \bK \to  \tilde{P}_{+} \wedge \bK \to (X\times S^2)_{hor,+} \wedge \bK. 
      \end{equation}
      is homotopic to the map induced by the inclusion $X \to X\times S^2$ as the fibre over $\infty$.  This follows from the fact that Diagram \eqref{eqn pushpull compatibility} is induced by a diagram of $\bK$-modules.
\end{proof}

Using Lemma \ref{Lem:null}, we conclude:
\begin{Corollary}\label{Cor:null_for_K-local}
  If $\bK$ is a $K(n)$-local complex oriented spectrum, then
  \begin{equation}
        S^1 \wedge X_+ \wedge \bS \to X_+ \wedge \bK
      \end{equation}
      is nullhomotopic.  \qed
\end{Corollary}

\subsection{Chromatic splitting}

For a fixed prime $p$, the Morava $K$-theories $K(n) =K_p(n)$ are obtained as quotients of the $p$-local spectrum $BP$ which arises as one of the isomorphic wedge summands of the $p$-localization $MU_{(p)}$, see \cite{Ravenel:book}.

We shall use the following result about the $K(n)$-localisations of these spectra:
\begin{Theorem}[Theorem 3.1 of \cite{Hovey}]\label{Thm:Hovey}
The natural map $BP \to \prod_{n=1}^{\infty} L_{K(n)} BP$ is the inclusion of a wedge summand.
\end{Theorem}

The sweepout map associated to $BP$ thus factors as
\begin{equation}
  S^1 \wedge X_+  \wedge \bS \to X_+ \wedge \prod_{n=1}^{\infty} L_{K(n)} BP \to  X_+ \wedge BP,
\end{equation}
so that the null-homotopies from Corollary \ref{Cor:null_for_K-local} yield a null-homotopy of $\eta_{\phi}\wedge BP_p $. Decomposing $MU_p$ into its wedge summands (Theorem \ref{theoremlocalizationBP}), we conclude:
\begin{Corollary} \label{cor:p-local-MU}
For each prime $p$, the stable sweepout $\eta_{\phi}$  is nullhomotopic after smashing with $ MU_{(p)}$. \qed
\end{Corollary}

\subsection{Null homotopy for complex oriented cohomology theories}
\label{sec:null-homot-compl}

We are now ready to complete the proof of our main result:
\begin{proof}[Proof of Theorem \ref{thm:main_generalised}]
  Since $ S^1 \wedge X_+$ is finite, Lemma \ref{lem:localisation-injective} and Corollary \ref{cor:p-local-MU} imply that the $MU$ sweepout map
  \begin{equation}
      S^1 \wedge X_+ \to  X_+ \wedge MU.
    \end{equation}
    is null homotopic.  On the other hand, the assumption that $\bE$ is a complex oriented cohomology theory includes the unitality statement that the unit map $\bS \to \bE$ factors through $MU$, so the result follows from the factorisation
   \begin{equation}
      S^1 \wedge X_+ \to  X_+ \wedge MU \to X_+ \wedge \bE.
  \end{equation}
\end{proof}

\subsection{Steenrod operations}\label{Sec:steenrod}

We conclude by explaining the proof of Corollary \ref{Cor:Steenrod}.

Let $\scrA_p = \Hom_{\bS}(H\bZ/p, H\bZ/p)$ denote the Steenrod algebra, where $\Hom_{\bS}$ denotes graded morphisms in the stable homotopy category; equivalently, $\scrA_p = \pi_*(F(H\bZ/p, H\bZ/p))$.  We let $\scrA_p^{\vee}$ be  its $\bZ/p$-linear dual.   It is well known that $\scrA_p^{\vee} = \pi_*(H\bZ/p \wedge H\bZ/p) = (H\bZ/p)_*(H\bZ/p)$.  Let $X$ be a finite CW complex. The Steenrod action
\begin{equation}
\scrA_p \otimes H^*(X;\bZ/p) \to H^*(X;\bZ/p)
\end{equation}
dualises to a coaction
\begin{equation}
H_*(X;\bZ/p) \to \scrA_p^{\vee} \otimes H_*(X;\bZ/p)
\end{equation}
which is given more explicitly by
\begin{equation}
(H\bZ/p)_*(X) \to (H\bZ/p)_*(\bS\wedge X) \to (H\bZ/p)_*(H\bZ/p \wedge X) \to \scrA_p^{\vee} \otimes (H\bZ/p)_*(X)
\end{equation}
where the second arrow comes from the unit $\bS \to H\bZ/p$. A map of spaces $X\to Y$ gives a map $H_*(X;\bZ/p) \to H_*(Y;\bZ/p)$ which is a map of coalgebras.  We are interested in a situation where we instead have a map $X \to Y\wedge MU$.  Let 
\begin{equation}
(\scrA_p^{\vee})_{MU} = \pi_*(H\bZ/p \wedge_{MU} H\bZ/p)
\end{equation}
which comes with a natural quotient map
\begin{equation}
\scrA_p^{\vee} = \pi_*(H\bZ/p \wedge H\bZ/p) \to  \pi_*(H\bZ/p \wedge_{MU} H\bZ/p) = (\scrA_p^{\vee})_{MU}.
\end{equation}
Essentially from the definition, we then have:

\begin{Lemma}\label{lem:senger}
Let $X$ and $Y$ be finite spectra and suppose we have an equivalence $X \wedge MU \simeq Y \wedge MU$ of $MU$-modules.
Then the induced equivalence $X\wedge H\bZ/p \simeq Y \wedge H\bZ/p$ coming from the $MU$ module structure on $H\bZ/p$ yields an isomorphism of homotopy groups
\begin{equation}
H_*(X;\bZ/p) \to H_*(Y;\bZ/p)
\end{equation}
which respects the Steenrod coaction by the quotient $(\scrA_p^{\vee})_{MU}$. \qed
\end{Lemma}
\begin{rem}
  The main references \cite{Lawson:BP,Senger} that we shall use in this section  formulate many constructions in terms of a modern category of spectra (orthogonal or symmetric spectra), and others in terms of the classical notions of $A_\infty$ and $E_\infty$ ring spectra (e.g. as in \cite{LMSM}). In order to directly quote their results, one would therefore have to improve Lemma \ref{Lem:null} to this context; this is completely straightforward.

  Alternatively, one may note that the computations that we use are established by proving the degeneration of a $\mathrm{Tor}$ spectral sequence; this implies that these computation apply to the setting of homotopy modules, where $ H\bZ/p \wedge_{MU} H\bZ/p$, is defined to be the co-equaliser
  \begin{equation}
        H\bZ/p \wedge MU \wedge H\bZ/p \rightrightarrows  H\bZ/p \wedge H\bZ/p.
  \end{equation}
\end{rem}

 Suppose $X\wedge MU$ is given as an $MU$-module, so $X\wedge H = (X\wedge MU) \wedge_{MU} H$ is an $H \wedge_{MU} H$ comodule. We have a commutative diagram
\begin{equation}
\begin{tikzcd}
H_*(X;\bF_p) \ar{r} & H_*(X;\bF_p) \otimes_{\bF_p} \pi_*(H\bZ/p \wedge_{MU} H\bZ/p) \\
H_*(X;\bF_p) \ar[equal]{u} \ar{r} & H_*(X;\bF_p) \otimes_{\bF_p} \pi_*(H\bZ/p \wedge_{\bS} H\bZ/p) \ar{u}
\end{tikzcd}
\end{equation}
Suppose first that $p>2$ is an odd prime.  Senger \cite{Senger} shows that the map
\begin{equation} \label{eqn:compare}
 \pi_*(H\bZ/p \wedge H\bZ/p) \to \pi_*(H\bZ/p \wedge_{MU} H\Z/p) 
 \end{equation}
 is explicitly a map 
\begin{equation} \label{eqn:pass_to_MU_Senger}
\bF_p[\xi_1,\xi_2,\ldots] \otimes\Lambda_{\bF_p}(\tau_1,\tau_2,\ldots)\to  \Lambda_{\bF_p}(\sigma m_i \, | \, i \neq p^k-1) \otimes \Lambda_{\bF_p}(\tau_1,\tau_2,\ldots)
\end{equation}
sending $\tau_i \mapsto \tau_i$ and $\xi_j \mapsto 0$. Let $\phi_i: \pi_*(H\bZ/p \wedge H\bZ/p) \to \bF_p$ be the element which selects $\tau_i$ (i.e. sends $\tau_i \mapsto 1$ and kills other basis generators).  Note that $\phi_i  \in (\scrA_p^{\vee})^{\vee} = \scrA_p$. Then $\phi_i$ factors as
\begin{equation}
\begin{tikzcd}
\pi_*(H\bZ/p \wedge_{\bS} H\bZ/p) \ar{d} \ar{r} & \bF_p \\
\pi_*(H\bZ/p \wedge_{MU} H\bZ/p) \ar{ru} & 
\end{tikzcd}
\end{equation}
which means (taking $\bF_p$-linear duals) that the action of $\phi_i \in \scrA_p$ on $H^*(X;\bF_p)$ is determined by the $MU$-module $X \wedge MU$. The element $\phi_i$ is exactly $Q_i$ in the Milnor basis \cite{Milnor:Steenrod} of $\scrA_p$, determined inductively in terms of the Bockstein $\beta$ and power operations $P^j$ via 
\begin{equation}
Q_0 = \beta; \ Q_{i+1} = [P^{\,p^i}, Q_i].
\end{equation}
Thus, the equivalence of Theorem \ref{thm:main} entwines the action of the subalgebra of the Steenrod algebra generated by the $\{Q_i\}$. 

The map \eqref{eqn:pass_to_MU_Senger} is understood from the naturality of the K\"unneth spectral sequence 
\begin{equation}
\mathrm{Tor}^{MU_*}(H\bF_p,H\bF_p) \Rightarrow \pi_*(H\bF_p \wedge_{MU} H\bF_p)
\end{equation}
with respect to the map \eqref{eqn:compare}, and the argument given above relies only on the easier fact that $\tau_i \mapsto \tau_i + $  lower order (decomposable) elements. 
When $p=2$, the corresponding spectral sequence and its target were studied in \cite{Lawson:BP, Tilson}. It again degenerates, and there are isomorphisms
\begin{equation}
\pi_*(H\bF_2 \wedge H\bF_2) = \scrA_2^* = \bF_2[\xi_1,\xi_2,\ldots], \qquad |\xi_i| = 2^i-1
\end{equation}
and 
\begin{equation}
\pi_*(H\bF_2 \wedge_{MU} H\bF_2) = \Lambda_{\bF_2}(\sigma(x_1),\sigma(x_2),\ldots), \qquad |\sigma(x_i)| = 2i+1
\end{equation}
for which the map $\pi_*(H\bF_2 \wedge H\bF_2) \to \pi_*(H\bF_2 \wedge_{MU} H\bF_2) $ sends 
\begin{equation} \label{eqn:Tilson}
\xi_i \mapsto \sigma(x_{2^{i-1}-1}) + \mathrm{decomposables}.
\end{equation}
This map is injective on indecomposables (but not globally injective; we know from the example of $S^2\times S^2$ versus the Hirzebruch surface $F_1$ that $Sq^2$ does not lie in the image of the dual $\pi_*(F_{MU}(H\bF_2,H\bF_2)) \to \scrA_2$).  In any case, \eqref{eqn:Tilson}
  is enough to deduce that the element $\phi_i \in (\scrA_2^{\vee})^{\vee}$ which picks out $\xi_i$ again factors through the $MU$-module. Again from \cite{Milnor:Steenrod}, one identifies $\phi_i$ as the element $Q_i$, and this completes the proof of Corollary \ref{Cor:Steenrod}.  
 
\bibliographystyle{plain}
\bibliography{mybib}

\end{document}